\documentclass[reqno]{amsart}
\usepackage{amsthm,amsmath,amssymb}
\usepackage[all]{xy}
\usepackage{appendix}
\usepackage{cite}
\usepackage{tikz}
\usepackage{tikz-cd}
\usepackage{bbm}
\usepackage{mathtools}
\usepackage{mathrsfs}

\newtheoremstyle{mystyle}%   % Name
  {}%                         % Space above
  {}%                         % Space below
  {}%                         % Body font
  {}%                         % Indent amount
  {\bfseries}%                % Theorem head font
  {.}%                        % Punctuation after theorem head
  { }%                        % Space after theorem head, ' ', or \newline
  {}%                         % Theorem head spec (can be left empty, meaning `normal')

\theoremstyle{mystyle}
\theoremstyle{definition}
\newtheorem{theorem}{Theorem}[section]
\newtheorem{proposition}{Proposition}[section]
\newtheorem{lemma}{Lemma}[section]

\newtheorem{definition}{Definition}[section]
\newtheorem{remark}{Remark}[section]
\numberwithin{equation}{section}
\newcommand{\qbinom}[2]{{\begin{bmatrix}#1\\#2\end{bmatrix}}_q}
\newcommand{\Li}{\operatorname{Li}_{2}}
\newcommand{\trace}{\operatorname{trace}}
\newcommand{\fluct}{\operatorname{fluct}}
\newcommand{\full}{\operatorname{full}}
\newcommand{\Var}{\operatorname{Var}}
\newcommand{\Lip}{\operatorname{Lip}}
\newcommand{\supp}{\operatorname{supp}}
\newcommand{\Id}{\operatorname{Id}}
\usepackage{xcolor}
\usepackage[colorlinks=true, citecolor=blue, filecolor=black, linkcolor=black, urlcolor=black]{hyperref}

\begin{document}
\title[$q$-Deformed polyanalytic Ginibre point processes]{$q$-Deformed polyanalytic Ginibre point processes:\\ construction and central limit theorems for linear statistics}

%%%%%%%%%%%%%%%%%% author %%%%%%%%%%%%%%%%%%
\author[Y.-G. Jung]{Yeong-Gwang Jung}
\address{Department of Mathematical Sciences, Seoul National University, Seoul 151-747, Republic of Korea}
\email{wollow21@snu.ac.kr}

\author[R. Sato]{Ryosuke SATO}
\address{Department of Mathematics, Faculty of Science, Hokkaido University, Kita 10, Nishi 8, Kita-ku, Sapporo, Hokkaido, 060-0810, Japan}
\email{r.sato@math.sci.hokudai.ac.jp} 
%%%%%%%%%%%%%%%%%% author %%%%%%%%%%%%%%%%%%

\begin{abstract}
    Based on the representation theory of the $q$-CCR algebra, we develop a representation-theoretic construction of the $q$-deformed polyanalytic Fock spaces, their reproducing kernels, and the associated $q$-deformed polyanalytic Bargmann transforms. The associated determinantal point processes provide $q$-deformations of the pure and full polyanalytic Ginibre ensembles. Moreover, we determine their limiting density and establish central limit theorems for the fluctuations of their linear statistics as the number of particles tends to infinity and the deformation parameter $q$ tends to 1 simultaneously. Remarkably, although the $q$-deformation changes the macroscopic density and the size of the droplet, the limiting covariance structures coincide, after spatial rescaling, with those of the classical polyanalytic Ginibre ensembles.
\end{abstract}

\maketitle

\allowdisplaybreaks
%%%%%%%%%%%%%%%%%%%%%%%%%%%%%%%%%%%%%%%%%%%%%%%%%%%%%%%%%%%%%%%%%%%%%%%%%%%%%%%%%%%%%%%%%%%%%%%%%%%%%%%%%%%%%%%%%%%%%%%%%%%%%%%%%%%%%%%%%%%%%%%%%%%%%%%%%%%%%%%%%%%%

\section{Introduction}

The Ginibre point process is one of the fundamental determinantal point processes
(DPPs) on the complex plane \cite{BF24,Fo10,Gin65,HKPV09}. From the perspective
of random matrix theory, it arises as the large-size limit of the eigenvalue
point process of the complex Ginibre ensemble introduced in \cite{Gin65}, whose
matrix entries are independent complex Gaussian random variables. Notably, the statistical properties of this point process are determined by the so-called correlation kernel, given by $\mathcal{K}(z, w)= e^{z\bar w}$. Here, the complex plane $\mathbb{C}$ is equipped with the complex Gaussian measure 
\begin{equation*}
    d\mathsf{m}(z):=e^{-|z|^2}\,dA(z),\qquad
    dA(z)=\frac{dx\,dy}{\pi}
\end{equation*}
as a reference measure. On the other hand, from a functional-analytic viewpoint, $\mathcal{K}$ is precisely the reproducing kernel of the \emph{analytic Bargmann--Fock space $\mathcal{A}_1$}, consisting of holomorphic functions that are square-integrable with respect to $\mathsf{m}$. Indeed, $\mathcal{K}$ can be recovered from the canonical orthonormal basis $\{z^n/\sqrt{n!}\}_{n\geq 0}$ of $\mathcal{A}_1$ as follows:
\begin{equation*}
    \mathcal{K}(z,w)
    =\sum_{n=0}^{\infty}\frac{z^{n}\overline{w}^{n}}{n!}.
\end{equation*}
It is worth noting that the Bargmann--Fock space $\mathcal{A}_1$ carries the standard analytic realization of the \emph{canonical commutation relations} (CCR), which is unitarily equivalent to the Fock representation via the \emph{Bargmann transform} \cite{Ba61,Fo89}.

After the standard Gaussian gauge transform, the Bargmann--Fock space $\mathcal{A}_1$ is identified with the eigenspace of the Landau Hamiltonian corresponding to its lowest eigenvalue, known as the lowest Landau level. More precisely, the reproducing kernels of the higher Landau level spaces and their finite-rank truncations serve as correlation kernels of determinantal point processes, called the \emph{polyanalytic Ginibre point processes} and the \emph{polyanalytic Ginibre ensembles}, respectively; see \cite{Ab10,AF14,HH13,Sh15}. These Landau-level projection processes also admit an interpretation as systems of planar noninteracting fermions in a uniform magnetic field \cite{HH13,HW19}. Closely related determinantal models arise for fermions in rotating two-dimensional traps; these include the exact lowest-level correspondence with the complex Ginibre ensemble \cite{LCMS19}, as well as regimes involving several occupied Landau levels \cite{KMS21,SLMS22}. See also the review \cite{DLMS19}.

The purpose of this paper is to extend this framework to the $q$-deformed setting. Here, the underlying CCR and the Bargmann transform are replaced by their $q$-deformed counterparts, the $q$-CCR \cite{AC76} and the $q$-deformed Bargmann transform; see Section \ref{sec:construction} for more details. This extension allows us to introduce $q$-deformed polyanalytic Ginibre point processes and $q$-deformed polyanalytic Ginibre ensembles, which are the main objects of this paper. 

In random matrix theory, $q$-deformed random matrix ensembles have been investigated by replacing the underlying classical orthogonal polynomials and their orthogonality measures with their $q$-analogs in the basic hypergeometric family; see \cite{BF25a,BFO26,BJM26,BJO26,Fo22,FL20,FLSY23,MCI01,MCIN93}. Similarly, in the previous works \cite{AEM18, EM21}, $q$-deformations of polyanalytic Bargmann transforms were introduced using $q$-analogs of Hermite and Laguerre polynomials. In this paper, we reveal the algebraic structure underlying the $q$-deformed polyanalytic Bargmann transforms through the Fock representation of the $q$-CCR algebra. This yields a natural construction of the correlation kernels for the associated $q$-deformed Ginibre point processes and their finite-particle ensembles.

We note that the orthogonality measure for the $q$-Laguerre polynomials is not unique \cite{Mo81,IR98}. Consequently, its choice determines the measure underlying the $q$-deformed Bargmann--Fock space, which also serves as the reference measure for the corresponding $q$-deformed Ginibre point processes. In the previous work \cite{EM21}, the orthogonality measure was chosen to be an absolutely continuous, rotationally invariant measure on the complex plane. This construction was proposed as a possible starting point for $q$-deformed Ginibre-type processes, but the associated determinantal models and their large-particle asymptotics were not developed in that work.

In the present paper, we choose a different orthogonality measure for the $q$-Laguerre polynomials. With this choice, our reference measure is given by the Jackson measure in the radial variable and the uniform measure in the angular variable; see \eqref{eqn:def of mq}. Thus, the reference measure is supported on a $q$-geometric family of concentric circles in the complex plane, and our model is radially discrete but angularly continuous. We therefore construct $q$-deformations of polyanalytic Ginibre processes arising from this Jackson-radial realization and investigate their large-particle asymptotics.

\subsection{$q$-deformed polyanalytic Ginibre ensembles}

Within the representation-theoretic formalism connecting the polyanalytic Ginibre ensemble and the Fock representation of the CCR algebra, the $q$-CCR algebra and its representation play a central role in the construction of the \emph{$q$-deformed polyanalytic Ginibre ensemble}. While the full details of this construction are deferred to Section \ref{sec:construction}, here we introduce only the essential notation required to state our main results.

We first collect some basic notations from $q$-analysis; see \cite{GR04, KLS10, KS97} for more details. Throughout this paper, $q\in(0,1)$ denotes the deformation parameter.

The $q$-integer, its factorial, and the $q$-binomial coefficient are defined by 
\begin{equation*}
    [n]_{q}:=\frac{1-q^{n}}{1-q}, \quad
    [n]_{q}!:=[n]_{q}[n-1]_{q}\cdots [1]_{q}, \quad
    \qbinom{n}{k}:=\frac{[n]_{q}!}{[k]_{q}![n-k]_{q}!},
\end{equation*}
for every $n\geq k\geq 0$, with the convention $[0]_{q}!=1$. The $q$-Pochhammer symbols $(a; q)_\infty$ and $(a; q)_n$ are given as
\begin{equation*}
    (a;q)_{\infty}
    :=\prod_{k=0}^{\infty}(1-aq^{k}), \quad
    (a;q)_{n}:=\frac{(a;q)_{\infty}}{(aq^{n};q)_{\infty}}
    :=(1-a)(1-aq)\cdots(1-aq^{n-1}).
\end{equation*}
We also abbreviate
\begin{equation*}
    (a_{1},\cdots,a_{k};q)_{n}
    =\prod_{l=1}^{k}(a_{l};q)_{n},\qquad
    1\leq n\leq \infty.
\end{equation*}
Two standard $q$-exponential functions are given by
\begin{equation*}
    e_{q}(z)
    :=\sum_{n=0}^{\infty}\frac{z^{n}}{[n]_{q}!}
    =\frac{1}{((1-q)z; q)_{\infty}}, \quad
    E_{q}(z)
    :=\sum_{n=0}^{\infty}\frac{q^{n(n-1)/2}}{[n]_{q}!}z^{n}
    =(-(1-q)z;q)_{\infty}.
\end{equation*}
The power series defining $e_{q}(z)$ converges absolutely for $|z|<(1-q)^{-1}$. Since $E_q(z)$ is an entire function and $e_{q}(z)=1/E_{q}(-z)$, $e_{q}(z)$ extends to a meromorphic function on the complex plane with poles at $z=q^{-k}/(1-q)$ for each $k\geq 0$.

The $q$-differentiation and the Jackson $q$-integral are defined by 
\begin{equation*}
    D_{q}f(x)
    :=\frac{f(x)-f(qx)}{x-qx}, \quad
    \int_{0}^{\infty}f(x)\,d_{q}x
        :=(1-q)\sum_{k=-\infty}^{\infty}f(q^{k})q^{k}.
\end{equation*}
The following change-of-variables formula is useful: for every $n\in \mathbb{Z}$
\begin{equation}\label{eq:change_of_variable}
    \int_{0}^{\infty}f(q^{-n}x)\,d_{q}x
    =q^{n}\int_{0}^{\infty}f(x)\,d_{q}x.
\end{equation}

As a $q$-analog of the Gaussian weight, we introduce a Radon measure $\mathsf{m}_{q}$ on $\mathbb{C}$ defined by
\begin{equation}\label{eqn:def of mq}
    d\mathsf{m}_{q}(re^{i\theta})
    =\frac{d\theta}{2\pi}e_{q}(-r^{2})\,d_{q}(r^{2}).
\end{equation}
This measure $\mathsf{m}_q$ is concentrated on the union of concentric circles with geometrically spaced radii
\begin{equation*}
    \bigcup_{k\in\mathbb{Z}}\{z\in\mathbb{C}:|z|^{2}=q^{k}\}.
\end{equation*}
In the continuum limit $q\to1$, this measure recovers the ordinary Gaussian measure $\mathsf{m}$, while the support becomes dense in the complex plane simultaneously. This $q$-deformed Gaussian weight will serve as a reference measure for DPPs considered in this paper.

For an abstract Fock space $\mathcal{H}$ with orthonormal basis $\{|n\rangle\}_{n\geq0}$, the \emph{$q$-deformed order-$r$ Bargmann transform} is an isometric linear map
\begin{equation*}
    \mathcal{B}_{r}^{q}:\mathcal{H}\to L^{2}(\mathbb{C},\mathsf{m}_{q}),
\end{equation*}
where $r+1\geq1$ indicates the polyanalytic order of the image of $\mathcal{B}^{q}_{r}$. In analogy with the classical case, $\{\mathcal{B}_{r}^{q}|n\rangle\}_{n,r\geq0}$ forms an orthonormal system in $L^{2}(\mathbb{C},\mathsf{m}_{q})$.

Accordingly, for $m\geq1$, the \emph{$q$-deformed pure $m$-analytic Ginibre ensemble} is the DPP with respect to $\mathsf{m}_{q}$ whose correlation kernel is
\begin{equation}\label{eqn:qGinibre pure ensemble kernel}
    \mathcal{K}_{m,N}^{q}(z,w)
    =\sum_{n=0}^{N-1}\mathcal{B}_{m-1}^{q}|n\rangle(z,\overline{z})\overline{\mathcal{B}_{m-1}^{q}|n\rangle(w,\overline{w})},
\end{equation}
whereas the correlation kernel of the \emph{$q$-deformed full $m$-analytic Ginibre ensemble} is given by
\begin{equation}\label{eqn:qGinibre full ensemble kernel}
    \mathcal{K}_{m,N}^{q,\full}(z,w)
    =\sum_{r=1}^{m}\sum_{n=0}^{N-1}\mathcal{B}_{r-1}^{q}|n\rangle(z,\overline{z})\overline{\mathcal{B}_{r-1}^{q}|n\rangle(w,\overline{w})}.
\end{equation}
The pure ensemble has $N$ particles, whereas the full ensemble has $mN$ particles almost surely.

\subsection{Main results and discussion}

In this subsection, we state the main results of this paper on the asymptotic statistics of the $q$-deformed polyanalytic Ginibre ensembles: a $q$-deformed circular law for the limiting density measure and central limit theorems for linear statistics. For the asymptotic analysis of the $q$-deformed ensemble, we adopt the double-scaling on $q$ as
\begin{equation}\label{eqn:def of q scaling}
    q=e^{-\lambda/N}
\end{equation}
with $\lambda>0$ fixed. Under this scaling, the continuum limit $q\uparrow1$ is taken as the system size $N$ tends to infinity.

In analogy with the Ginibre ensemble, where the Gaussian weight is scaled to $e^{-N|z|^2}$ to obtain a non-trivial macroscopic limit, we rescale the reference measure by
\begin{equation}\label{eqn:def of mqhat}
    d\widehat{\mathsf{m}}_{q}(re^{i\theta})
    =\frac{d\theta}{2\pi}\widehat{e}_{q}(-r^{2})\,d_{q}(r^{2}),\qquad
    \widehat{e}_{q}(x)
    =e_{q}\Big(\frac{x}{1-q}\Big)
    =\frac{1}{(x;q)_{\infty}}.
\end{equation}
Accordingly, we rescale the correlation kernels of $q$-deformed polyanalytic Ginibre ensembles of the pure type and full type as
\begin{equation}\label{eqn:def of Khat}
    \widehat{\mathcal{K}}_{m,N}^{q}(z,w)
    :=\frac{1}{1-q}\mathcal{K}_{m,N}^{q}\Big(\frac{z}{\sqrt{1-q}},\frac{w}{\sqrt{1-q}}\Big), \qquad
    \widehat{\mathcal{K}}_{m,N}^{q,\full}(z,w)
    :=\frac{1}{1-q}\mathcal{K}_{m,N}^{q,\full}\Big(\frac{z}{\sqrt{1-q}},\frac{w}{\sqrt{1-q}}\Big).
\end{equation}
The normalization factor is chosen so that the rescaled kernels remain projection kernels in $L^{2}(\mathbb C,\widehat{\mathsf m}_q)$.

\begin{remark}
    In the analytic case $m=1$, $\widehat{\mathcal{K}}_{1,N}^{q}$ is the projection kernel onto the space of analytic polynomials with degree lower than $N$. Hence, the joint law of the analytic ensemble has a radially discrete Coulomb gas form
    \begin{equation}\label{eqn:analytic joint law}
        d\mathbb{P}_{N,q}(z_{1},\cdots,z_{N})
        =\frac{1}{Z_{N,q}}\prod_{1\leq j<k\leq N}|z_{j}-z_{k}|^{2}\prod_{j=1}^{N}d\widehat{\mathsf{m}}_{q}(z_{j}),
    \end{equation}
    where $Z_{N,q}>0$ is a normalization constant.
\end{remark}

For the rescaled pure and full ensembles, we define the normalized random empirical measures by
\begin{equation}\label{eqn:empirical measure}
    \widehat{\mu}_{m,N}^{q}
    :=\frac{1}{N}\sum_{j=1}^{N}\delta_{z_{j}}, \qquad
    \widehat{\mu}_{m,N}^{q,\full}
    :=\frac{1}{mN}\sum_{j=1}^{mN}\delta_{z_{j}},
\end{equation}
where $\{z_{j}\}_{j}$ denotes the random point configuration of the rescaled $q$-deformed polyanalytic Ginibre ensemble of pure and full type, respectively.
Our first main result identifies the macroscopic limit of \eqref{eqn:empirical measure} as a one-parameter deformation of the circular law. For $\lambda>0$, the \emph{$q$-deformed circular law} is the probability measure on the complex plane defined by its density function
\begin{equation}\label{eqn:limiting law}
    \widehat{\rho}_{\lambda}(z)\,dA(z)
    =\frac{1}{\lambda (1+|z|^{2})}\mathbbm{1}_{D_{\lambda}}(z)\,dA(z), \qquad
    D_{\lambda}:=\big\{z\in\mathbb{C}:|z|\leq\sqrt{e^{\lambda}-1}\big\}.
\end{equation}

\begin{theorem}\label{thm:double scale limiting law}
    Let $q=e^{-\lambda/N}$ for a fixed $\lambda\in(0,\infty)$. For a fixed $m\geq1$, the normalized empirical measures of the rescaled $q$-deformed pure and full $m$-analytic Ginibre ensemble converge weakly to the $q$-deformed circular law
    \begin{equation*}
        d\widehat{\mu}^{q}_{m,N}(z)
        \to\widehat{\rho}_{\lambda}(z)\,dA(z),\qquad
        d\widehat{\mu}^{q,\full}_{m,N}(z)
        \to\widehat{\rho}_{\lambda}(z)\,dA(z)
    \end{equation*}
    in probability.
\end{theorem}

\begin{remark}
    Although the $q$-deformed finite-particle ensembles are radially discrete, their normalized empirical measures converge in probability to the absolutely continuous measure $\widehat{\rho}_\lambda(z)dA(z)$. Moreover, unlike the classical Ginibre ensemble, whose empirical measure converges to the uniform measure on the unit disk by the circular law \cite{Ba97,BC12}, the limiting density in the $q$-deformed setting is non-uniform. Nevertheless, the limiting measure is supported on the disk $D_\lambda$, whose radius depends only on the scaling parameter $\lambda$, and is therefore independent of both the polyanalytic order $m$ and the type of the ensemble.

    On the other hand, the limiting density recovers the circular law in the rescaled continuum limit
    \begin{equation*}
        \lambda\widehat{\rho}_{\lambda}(\sqrt{\lambda}z)
        \rightarrow\mathbbm{1}_{\{|z|\leq 1\}}(z),\qquad
        \lambda\downarrow 0.
    \end{equation*}
    In this sense, Theorem \ref{thm:double scale limiting law} provides a $q$-deformed circular law.
\end{remark}

\begin{remark}
     For continuous analytic ensembles on the complex plane, logarithmic potential theory provides a powerful tool for identifying the macroscopic equilibrium measure. In the analytic case $m=1$, the joint law \eqref{eqn:analytic joint law} and limiting density \eqref{eqn:limiting law} can be naturally related to such a continuum equilibrium problem.

    Indeed, the asymptotic expansion of the $q$-Pochhammer symbol \cite[Theorem 4]{Mc99} suggests that, under the double scaling $q=e^{-\lambda/N}$, the radial Jackson weights are governed at the exponential scale by the effective potential
    \begin{equation*}
        Q_{\lambda}(z)
        :=-\frac{1}{\lambda}\Li(-|z|^{2}).
    \end{equation*}
    If one considers the continuous normal matrix ensemble with external potential $Q_{\lambda}$, its equilibrium measure is determined by Frostman's theorem \cite{ST97,AHM15}. The radial Frostman condition gives the equilibrium droplet
    \begin{equation*}
        D_{\lambda}
        =\big\{z\in\mathbb{C}:|z|^{2}\leq e^{\lambda}-1\big\},
    \end{equation*}
    and
    \begin{equation*}
        \partial\overline{\partial}Q_{\lambda}(z)
        =\frac{1}{\lambda(1+|z|^{2})}.
    \end{equation*}
    Hence the equilibrium measure of the associated continuous model is
    \begin{equation*}
        \frac{1}{\lambda(1+|z|^{2})}\mathbbm{1}_{D_{\lambda}}(z)\,dA(z),
    \end{equation*}
    which coincides with the limiting measure obtained in Theorem \ref{thm:double scale limiting law}.

    However, for a fixed $N$, the ensemble considered here is supported on a $q$-geometric family of concentric circles and is therefore radially discrete. Moreover, to the best of our knowledge, a large deviation principle relating the empirical measure of this discrete model to the continuum logarithmic-energy problem above has not yet been established. Thus, although the potential-theoretic argument above provides a natural continuum interpretation of the limiting measure, it does not constitute an alternative proof of Theorem \ref{thm:double scale limiting law}. We leave the development of such a large deviation theory for future work. For related large deviation problems in one-dimensional $q$-deformed and discrete biorthogonal ensembles, see \cite{BJM26,HMO25}.
\end{remark}

\medskip
Following the macroscopic characterization in Theorem \ref{thm:double scale limiting law}, our next contribution is the analysis of the second-order fluctuations of the system via linear statistics. Fluctuations of smooth linear statistics in planar random matrix models have been studied extensively. Gaussian fluctuations for analytic test functions in broader classes of non-Hermitian random matrices were established in \cite{RS06}. For $C^{1}$ linear statistics of the complex Ginibre ensemble and the decomposition into bulk and boundary noises, see \cite{RV07}. For general random normal matrix ensembles, bulk and global fluctuation results were developed in \cite{AHM11,AHM15}; see also \cite{LS18} for two-dimensional Coulomb gases at general inverse temperature. Microscopic-scale radial linear statistics for the two-dimensional Coulomb gas at inverse temperature $\beta=2$ were subsequently studied in \cite{LS25}. Recent extensions include radially symmetric models with spectral gaps \cite{ACC25,ACC26}, as well as crossover regimes between strong and weak non-Hermiticity in the elliptic Ginibre ensemble \cite{MAD26}. The central limit theorem for the classical pure and full polyanalytic Ginibre ensembles was established in \cite{HW19}.

For a real-valued smooth test function $g$ with compact support, the linear statistics of the rescaled $q$-deformed $m$-analytic Ginibre ensemble are defined by
\begin{equation*}
    \trace_{m,N}(g)
    =\sum_{j=1}^{N}g(z_{j}),
\end{equation*}
where the particles $\{z_{j}\}_{j=1}^{N}$ are distributed according to the DPP associated with $\widehat{\mathcal{K}}^q_{m,N}$.
Similarly, the linear statistics of the full ensemble are defined by
\begin{equation*}
    \trace^{\full}_{m,N}(g)
    =\sum_{j=1}^{{Nm}}g(z_{j}),
\end{equation*}
where $\{z_j\}_{j=1}^{Nm}$ are governed by the DPP with the correlation kernel $\widehat{\mathcal{K}}^{q,\full}_{m, N}$. We write
\begin{equation*}
    \fluct_{m,N}(g)
    :=\trace_{m,N}(g)-\mathbb{E}\big[\trace_{m,N}(g)\big],\qquad
    \fluct^{\full}_{m,N}(g)
    :=\trace^{\full}_{m,N}(g)-\mathbb{E}\big[\trace^{\full}_{m,N}(g)\big].
\end{equation*}

Let $\mathcal{N}(0,\sigma^{2})$ denote the centered Gaussian distribution with variance $\sigma^{2}$. The Dirichlet seminorm of $g$ on the interior $D_\lambda^\circ$ of the droplet $D_\lambda$ in Theorem \ref{thm:double scale limiting law} is defined by
\begin{equation*}
    \|g\|_{H^{1}(D_\lambda^\circ)}
    :=\Big(\int_{D_\lambda^\circ}|\overline{\partial}g|^{2}\,dA(z)\Big)^{\frac{1}{2}},
\end{equation*}
while the $H^{1/2}$ seminorm of $g$ on the boundary $\partial D_\lambda$ is expressed as
\begin{equation*}
    \|g\|_{H^{1/2}(\partial D_{\lambda})}
    :=\Big(\sum_{k\in\mathbb{Z}}|k||\widehat{g}_{\lambda}(k)|^{2}\Big)^{\frac{1}{2}},\qquad
    \widehat{g}_\lambda(k)
    :=\frac{1}{2\pi}\int_{0}^{2\pi}g(\sqrt{e^\lambda-1}e^{i\theta})e^{-ik\theta}\,d\theta.
\end{equation*}

Our second main result is the following central limit theorem for the linear statistics of the $q$-deformed pure and full $m$-analytic Ginibre ensembles.
\begin{theorem}\label{thm:CLT}
    Let $q=e^{-\lambda/N}$ for a fixed $\lambda>0$. Let $g$ be a real-valued smooth test function with compact support. For a fixed $m\geq1$, we have
    \begin{align}
        \fluct_{m,N}(g)
        &\longrightarrow\mathcal{N}\Big(0,(2m-1)\|g\|_{H^{1}(D_{\lambda}^{\circ})}^{2}+\frac{1}{2}\|g\|_{H^{1/2}(\partial D_{\lambda})}^{2}\Big),\label{eqn:pure CLT}\\
        \fluct_{m,N}^{\full}(g)
        &\longrightarrow\mathcal{N}\Big(0,m\big(\|g\|_{H^{1}(D_{\lambda}^{\circ})}^{2}+\frac{1}{2}\|g\|_{H^{1/2}(\partial D_{\lambda})}^{2}\big)\Big),\label{eqn:full CLT}
    \end{align}
    in distribution as $N\to\infty$.
\end{theorem}

Although the droplet depends on $\lambda$, the limiting variances retain the classical bulk–boundary decomposition. Dilating $D_{\lambda}$ to the unit disk, the variances recover their classical counterparts; see \cite[Theorem 1.1]{HW19}.

\begin{remark}
    Both Theorems~\ref{thm:double scale limiting law} and \ref{thm:CLT} are proved with the polyanalytic order $m$ fixed as $N\to\infty$. A different asymptotic regime arises when $m$ grows as $N$ tends to infinity. The arguments in the present paper are not uniform in $m$ and hence do not cover this regime. Determining the limiting density and the fluctuation behavior in this growing-order regime is left for future work. Related large-order phenomena for the classical polyanalytic Ginibre ensemble are discussed in \cite{Ha14, HH13}; see also the numerical illustration in \cite[Figure~1]{HW19}. Related high- and many-Landau-level regimes arise for noninteracting fermions in rotating traps; see \cite{KMS21,SLMS22}.
\end{remark}

Let us briefly comment on the proof of Theorem \ref{thm:CLT}. The argument follows the general strategies of Rider and Vir\'ag \cite{RV07} and Haimi and Wennman \cite{HW19}. Following \cite{RV07}, we first express the cumulants for polynomial test functions as cyclic integrals in the analytic case $m=1$. For the higher Landau levels, following \cite{HW19}, we use the $q$-raising-operator representation of the polyanalytic kernels, which allows us to reduce the cumulants for the higher Landau levels to those involving only the analytic kernel. Thus, we can determine their asymptotic behavior from the analytic case via the large-$N$ expansions of the relevant $q$-difference operators. 

The extension from polynomial test functions to smooth ones differs from the argument in \cite{RV07}. Rather than deriving the limiting variance for general smooth test functions, we control the approximation error using a uniform Lipschitz variance bound and a localization estimate outside any fixed disk containing the droplet.

\subsection{Further directions}
Our construction of DPPs is closely related to a general framework in which DPPs arise from group representations and coherent-state transforms. This work provides a concrete step toward extending this framework from classical group representations to their $q$-analogues. Related $q$-deformed coherent-state transforms and reproducing kernels were studied in \cite{AEM18,EM21,ME21}. In particular, \cite{EM21,ME21} suggested their possible use in constructing $q$-deformed Ginibre-type processes. Motivated by these works, we construct projection DPPs for the Jackson-radial realization developed here and study their large-particle asymptotics.

In the classical setting, the representations of the Heisenberg group, $SU(2)$, and $SU(1,1)$ give rise to parallel constructions of coherent-state transforms and reproducing-kernel Hilbert spaces. See \cite{Fo89} for the classical Heisenberg group and the Bargmann--Fock representation, and \cite{Mo05a, Mo05b, Mo12} for related constructions associated with $SU(2)$ and $SU(1,1)$. The corresponding reproducing kernels also give rise to determinantal point processes, including Ginibre- and polyanalytic Ginibre-type ensembles \cite{HH13,Sh15}, finite Weyl--Heisenberg ensembles \cite{AGR19,APRT17}, spherical ensembles \cite{Kr09}, and hyperbolic-type processes \cite{DL19}.

Motivated by these classical parallels, it would be natural to develop analogous constructions based on representations of $SU_{q}(2)$ and $SU_{q}(1, 1)$. In particular, one may ask whether suitable coherent-state transforms give rise to $q$-deformed reproducing kernel Hilbert spaces and associated determinantal point processes.

\subsection*{Organization of the paper}

This paper is organized as follows. In Section \ref{sec:construction}, we discuss the Fock representation of the $q$-CCR algebra and reveal a representation-theoretic aspect of $q$-deformed polyanalytic Bargmann transforms. Based on this, we introduce $q$-deformed polyanalytic Fock spaces and derive explicit formulas for their orthonormal bases and reproducing kernels in terms of $q$-Laguerre polynomials and $q$-hypergeometric functions. Consequently, we are able to introduce our main objects in this paper, namely $q$-deformed polyanalytic Ginibre point processes and their finite-particle truncations. 

In Section \ref{sec:double_scaled_regime}, we derive exact formulas for the spectral moments of the rescaled ensembles and identify the weak limit of their normalized intensity measures. A variance bound for finite-rank projection DPPs improves this limit to convergence in probability of the random empirical measures, completing the proof of Theorem \ref{thm:double scale limiting law}.

The remaining three sections are devoted to the proof of the central limit theorem. In Section~\ref{sec:CLT lowest}, we treat the lowest Landau level and establish the limiting cumulants for real-valued polynomial test functions by an exact analysis of cyclic traces. In Section~\ref{sec:Raising}, we use the $q$-raising operators introduced in Section~\ref{sec:construction} together with exact Jackson adjoint identities to reduce cyclic traces in higher Landau levels to cyclic integrals involving the analytic kernel. Finally, in Section~\ref{sec:CLT}, we derive the asymptotic expansions of the reduction operators and use them to establish the polynomial central limit theorems for the pure and full ensembles. A uniform Lipschitz variance bound and localization outside the droplet are used to extend these results simultaneously to compactly supported smooth test functions, completing the proof of Theorem \ref{thm:CLT}.

In Appendix \ref{app:IFS}, we give a brief discussion of possible generalizations of the Bargmann transform, the associated reproducing kernels, and the associated DPPs from the perspective of interacting Fock spaces.

\subsection*{Acknowledgments}
The authors are grateful to the organizers of \emph{The 18th HU-SNU Joint Symposium on Mathematics}, held at Hokkaido University in October 2025, from which this collaborative work stemmed.

Y.-G. J. was supported by the National Research Foundation of Korea grant (RS2025-00516909, RS-2026-25518141). Y.-G. J. is deeply grateful to Sung-Soo Byun for valuable comments.
R. S. was funded by JSPS KAKENHI Grant-in-Aid for Research Activity Start-up Grant Number 25K23325. R. S. would like to thank Zouha\"ir Mouayn, whose intensive lecture at Chuo University proved highly valuable for this project. He is also grateful to Makoto Katori, Tomoyuki Shirai, and Kohei Noda for their fruitful comments on the draft, and to Takahiro Hasebe for bringing \cite{ABH16} to his attention.

%%%%%%%%%%%%%%%%%%%%%%%%%%%%%%%%%%%%%%%%%%%%%%%%%%%%%%%%%%%%%%%%%%%%%%%%%%%%%%%%%%%%%%%%%%%%%%%%%%%%%%%%%%%%%%%%%%%%%%%%%%%%%%%%%%%%%%%%%%%%%%%%%%%%%%%%%%%%%%%%%%%%
\section{Construction of $q$-deformed polyanalytic Ginibre point processes}\label{sec:construction}
%%%%%%%%%%%%%%%%%%%%%%%%%%%%%%%%%%%%%%%%%%%%%%%%%%%%%%%%%%%%%%%%%%%%%%%%%%%%%%%%%%%%%%%%%%%%%%%%%%%%%%%%%%%%%%%%%%%%%%%%%%%%%%%%%%%%%%%%%%%%%%%%%%%%%%%%%%%%%%%%%%%%

\subsection{CCR representations and polyanalytic Ginibre processes}

We briefly review the representation-theoretic construction of the classical polyanalytic Ginibre processes, which will serve as the undeformed model for our construction.

The \emph{oscillator algebra} $\mathfrak{A}$ is the unital $\mathbb{C}$-algebra generated by $a$ and $a^{\ast}$, subject to the \emph{canonical commutation relation} (CCR)
\begin{equation*}
    aa^{\ast}-a^{\ast}a
    =1.
\end{equation*}
For the Fock space $\mathcal{H}$ with orthonormal basis $\{|n\rangle \}_{n\geq0}$, the \emph{Fock representation} of $\mathfrak{A}$ on $\mathcal{H}$ is given as
\begin{equation*}
    a^{\ast}|n\rangle
    =\sqrt{n+1}|n+1\rangle,\qquad
    a|n\rangle
    =\sqrt{n}|n-1\rangle,
\end{equation*}
with the convention $|-1\rangle=0$. The operator $\mathsf{N}:=a^{\ast}a$ is called a \emph{number operator} and it is diagonal in the Fock basis as $\mathsf{N}|n\rangle=n|n\rangle$. For every $z\in\mathbb{C}$, the eigenvalue problem $a\psi=z\psi$ has a solution
\begin{equation*}
    |z\rangle
    =e^{-\frac{|z|^{2}}{2}}\sum_{n=0}^{\infty}\frac{z^{n}}{\sqrt{n!}}|n\rangle
\end{equation*}
which is called a \emph{coherent state}. The analytic Bargmann transform $\mathcal{B}_{0}:\mathcal{H}\to L^{2}(\mathbb{C},\mathsf{m})$ can be written in terms of the coherent state
\begin{equation}\label{eqn:analytic BF coherent rep}
    \mathcal B_0h(z)
    :=e^{\frac{|z|^{2}}{2}}\langle \overline{h}|z\rangle,
\end{equation}
where $\overline{h}=\sum_{n=0}^{\infty}\overline{\langle n|h\rangle}|n\rangle$ for every $h\in\mathcal{H}$. In particular, for every $n\geq 0$ we have 
\begin{equation}\label{eqn:analytic BF basis}
    \mathcal{B}_{0}|n\rangle(z)
    =\frac{z^{n}}{\sqrt{n!}}.
\end{equation}
Thus, $\{\mathcal{B}_0|n\rangle\}_{n\geq 0}$ forms an orthonormal system in $L^2(\mathbb{C}, \mathsf{m})$. Thus, $\mathcal{B}_{0}$ is the unitary map from $\mathcal{H}$ onto the analytic Bargmann--Fock space $\mathcal{A}_{1}\subset L^2(\mathbb{C}, \mathsf{m})$. Moreover, $\mathcal{B}_0$ intertwines the Fock representation on $\mathcal{H}$ with multiplication and differentiation operators
\begin{equation}\label{eqn:intertwining classical}
    \mathcal{B}_{0}a^{\ast}
    =z\mathcal{B}_{0},\qquad
    \mathcal{B}_{0}a
    =\partial_{z}\mathcal{B}_{0}.
\end{equation}

The general polyanalytic Bargmann transforms are obtained from the same representation. For $r\geq0$, define the \emph{order-$r$ coherent states}
\begin{equation}\label{eqn:r coherent state}
    |z,\overline{z};r\rangle
    :=\frac{1}{\sqrt{r!}}(a^{\ast}-\overline{z})^{r}|z\rangle
\end{equation}
and the \emph{order-$r$ Bargmann transform}
\begin{equation}\label{eq:r Bargmann transform}
    \mathcal{B}_{r}h(z,\overline{z})
    :=e^{\frac{|z|^{2}}{2}}\langle\overline{h}|z,\overline{z};r\rangle\qquad
    (h\in\mathcal{H}).
\end{equation}
Here, $\{\mathcal{B}_{r}|n\rangle\}_{n,r\geq0}$ forms an orthonormal basis of $L^{2}(\mathbb{C},\mathsf{m})$. 

\begin{remark}
    Originally, higher-order coherent states were defined using \emph{displacement operators}. Namely, we define $D(z):=\exp(za^{\ast}-\overline{z}a)$ and $|z, \overline{z}; r\rangle:= D(z)|r\rangle$. By CCR, we have $D(z)a^{\ast}=(a^{\ast}-\overline{z})D(z)$, and hence
    \begin{equation*}
        |z,\overline{z};r\rangle
        =\frac{1}{\sqrt{r!}}D(z)(a^{\ast})^{r}|0\rangle
        =\frac{1}{\sqrt{r!}}(a^{\ast}-\overline{z})^{r}|z\rangle.
    \end{equation*}
    In this paper, we take this identity as the definition, which in turn motivates our definition of a $q$-deformation of higher-order coherent states.
\end{remark}

Iterating the intertwining relation \eqref{eqn:intertwining classical} yields
\begin{equation}\label{eqn:classical raising}
    \mathcal{B}_{r}h
    =\frac{1}{\sqrt{r!}}(\partial_{z}-\overline{z})^{r}\mathcal{B}_{0}h,
\end{equation}
which motivates the definition of the \emph{raising operator}
\begin{equation}\label{eqn:classical raising operator}
    \mathcal{R}
    :=\partial_{z}-\overline{z}.
\end{equation}
This operator admits the gauge representation
\begin{equation}\label{eqn:raising gauge}
    \mathcal{R}f
    =e^{|z|^{2}}\partial_{z}\big[e^{-|z|^{2}}f\big],
\end{equation}
which implies that $\mathcal{R}$ is the adjoint of the partial differential operator $-\partial_{\overline{z}}$ as an operator acting on $L^{2}(\mathbb{C},\mathsf{m})$. These operators reveal another CCR structure in $L^{2}(\mathbb{C},\mathsf{m})$. Indeed, they satisfy the CCR
\begin{equation*}
    (-\partial_{\overline{z}})\mathcal{R}-\mathcal{R}(-\partial_{\overline{z}})
    =1.
\end{equation*}
By \eqref{eqn:classical raising} and $(-\partial_{\overline{z}})\mathcal{R}^{r}=r\mathcal{R}^{r-1}+\mathcal{R}^{r}(-\partial_{\overline{z}})$, we have 
\begin{equation*}
    (-\partial_{\overline{z}})\mathcal{B}_{r}|n\rangle
    =\sqrt{r}\mathcal{B}_{r-1}|n\rangle,\qquad
    \mathcal{R}\mathcal{B}_{r}|n\rangle
    =\sqrt{r+1}\mathcal{B}_{r+1}|n\rangle.
\end{equation*}
Thus, for each fixed $n\geq0$, the family $\{\mathcal{B}_{r}|n\rangle\}_{r\geq0}$
admits the Fock representation of $-\partial_{\overline{z}}$ and $\mathcal{R}$ acting as the annihilation and creation operators, respectively. By analogy with the number operator $\mathsf{N}$, we define the corresponding operator
\begin{equation}\label{eqn:mathcalL}
    \mathcal{L}
    :=\mathcal{R}(-\partial_{\overline{z}})
    =-(\partial_{z}-\overline{z})\partial_{\overline{z}}
\end{equation}
so that
\begin{equation*}
    \mathcal{L}\mathcal{B}_{r}|n\rangle
    =r\mathcal{B}_{r}|n\rangle.
\end{equation*}
Under the standard Gaussian gauge transform and normalization, $\mathcal{L}$ is unitarily equivalent to the Landau Hamiltonian on the plane. Accordingly, for each $r\geq0$, the kernel space
\begin{equation}\label{eqn:Landau}
    \mathcal{A}_{r+1}
    :=\ker(\mathcal{L}-r)
    =\overline{\operatorname{span}}\{\mathcal{B}_{r}|n\rangle:n\geq0\}
    \subset L^{2}(\mathbb{C},\mathsf{m})
\end{equation}
is identified with the $r$th Landau eigenspace. Analytically, for $m\geq1$, the space $\mathcal{A}_{m}$ is referred to as the pure (or true) $m$-analytic Bargmann--Fock space, which coincides with the orthogonal complement of $\mathcal{A}_{m-1}^{\full}$ in $\mathcal{A}_{m}^{\full}$. Here, the full $m$-analytic Bargmann--Fock space is given as
\begin{equation*}
    \mathcal{A}_{m}^{\full}
    :=\big\{f:\partial_{\overline{z}}^{m}f=0\big\}\cap L^{2}(\mathbb{C},\mathsf{m}),\qquad
    \mathcal{A}_{0}^{\full}
    =\{0\}.
\end{equation*}
Throughout this paper, we use $r$ to indicate the order of the ($q$-deformed) Bargmann transform or corresponding Landau levels, whereas $m$ denotes the polyanalytic order.

These pure and full polyanalytic Bargmann--Fock spaces induce two types of polyanalytic extension of the Ginibre process: the pure type and the full type. The pure $m$-analytic Ginibre process is the planar determinantal process with respect to the Gaussian measure $\mathsf{m}$ whose correlation kernel is given by the projection kernel onto $\mathcal{A}_{m}$:
\begin{equation*}
    \mathcal{K}_{m}(z,w)
    =\sum_{n=0}^{\infty}\mathcal{B}_{m-1}|n\rangle(z)\overline{\mathcal{B}_{m-1}|n\rangle(w)},
\end{equation*}
while the correlation kernel of the full $m$-analytic Ginibre process is 
\begin{equation*}
    \mathcal{K}_{m}^{\full}(z,w)
    =\sum_{r=1}^{m}\sum_{n=0}^{\infty}\mathcal{B}_{r-1}|n\rangle(z)\overline{\mathcal{B}_{r-1}|n\rangle(w)}.
\end{equation*}
We use the term \emph{process} for the infinite-rank determinantal point processes associated with the resulting projection kernels, and the term \emph{ensemble} for their finite-rank truncations. Accordingly, the pure and full $m$-analytic Ginibre ensembles are the determinantal point processes governed by the finite-rank projection kernels
\begin{equation*}
    \mathcal{K}_{m,N}(z,w)
    =\sum_{n=0}^{N-1}\mathcal{B}_{m-1}|n\rangle(z)\overline{\mathcal{B}_{m-1}|n\rangle(w)},\qquad
    \mathcal{K}_{m,N}^{\full}(z,w)
    =\sum_{r=1}^{m}\sum_{n=0}^{N-1}\mathcal{B}_{r-1}|n\rangle(z)\overline{\mathcal{B}_{r-1}|n\rangle(w)},
\end{equation*}
respectively. Thus, the pure ensemble has $N$ particles, while the full ensemble has $mN$ particles.

\subsection{The $q$-oscillator algebra and the analytic $q$-Bargmann transform}

We next recall the Fock representation of the $q$-oscillator algebra; see, e.g., \cite[Chapter 5]{KS97}. This representation provides the algebraic basics for our construction of $q$-deformed polyanalytic Ginibre point processes.

The \emph{$q$-oscillator algebra} $\mathfrak{A}_{q}$ is the unital $\mathbb{C}$-algebra generated by $a, a^{\ast}, q^{\mathsf{N}}, q^{-\mathsf{N}}$ subject to the $q$-CCR \cite{AC76}
\begin{equation}\label{eqn:qCCR}
    aa^{\ast}-qa^{\ast}a=\Id,\quad
    q^{\mathsf{N}}q^{-\mathsf{N}}=q^{-\mathsf{N}}q^{\mathsf{N}}=\Id,\quad
    q^{\pm \mathsf{N}}a^{\ast}=q^{\pm 1}a^{\ast}q^{\pm \mathsf{N}},\quad
    q^{\pm \mathsf{N}}a=q^{\mp 1}aq^{\pm \mathsf{N}}.
\end{equation}

For a Fock space $\mathcal{H}$, the \emph{Fock representation} of $\mathfrak{A}_{q}$ on $\mathcal{H}$ is given as follows:
\begin{equation*}
    a^{\ast}|n\rangle
    :=\sqrt{[n+1]_{q}}|n+1\rangle, \quad
    a|n\rangle
    :=\sqrt{[n]_{q}}|n-1\rangle, \quad
    q^{\pm \mathsf{N}}|n\rangle = q^{\pm n}|n\rangle,
\end{equation*}
where $|-1\rangle:=0$. In contrast to the classical case, the annihilation operator $a$ is bounded with
\begin{equation*}
    \|a\|
    =\lim_{n\to\infty}\sqrt{[n]_{q}}
    =(1-q)^{-\frac{1}{2}}.
\end{equation*}
For a complex number $z$ with $|z|^{2}<(1-q)^{-1}$, the \emph{$q$-coherent state}
\begin{equation*}
    |z\rangle
    =e_{q}(|z|^{2})^{-\frac{1}{2}}\sum_{n=0}^{\infty}\frac{z^{n}}{\sqrt{[n]_{q}!}}|n\rangle
    \in\mathcal{H}
\end{equation*}
is characterized as the solution of the eigenvalue problem $a\psi=z\psi$ satisfying $\langle z|z\rangle = 1$.

\begin{remark}
    There are several standpoints to introduce coherent states; see \cite{DK92} for more details. In this paper, we view coherent states as solutions to the eigenvalue problem for the annihilation operator, called \emph{Glauber coherent states}. On the other hand, coherent states obtained by applying displacement operators to the vacuum state are called \emph{Perelomov coherent states}. In the $q$-deformed setting, this approach has been explored in \cite{FV91}. These two notions agree in the classical limit $q\to1$, but are generally different for $q \in (0,1)$.
\end{remark}

Next, we construct a $q$-counterpart of the analytic Bargmann transform \eqref{eqn:analytic BF coherent rep}. As a natural $q$-analog of $\mathsf{m}$, we use a Radon probability measure $\mathsf{m}_q$ defined in \eqref{eqn:def of mq}. The integration-by-parts formula for the Jackson $q$-integral provides the mixed moments formula
\begin{equation}\label{eqn:q mixed moments}
    \int_{\mathbb{C}}z^{n}\overline{z}^{m}\,d\mathsf{m}_{q}(z)
    =\delta_{n,m}\int_{0}^{\infty}r^{2n}e_{q}(-r^{2})\,d_{q}r^{2}
    =\delta_{n,m}q^{-\frac{n(n+1)}{2}}[n]_{q}!.
\end{equation}
Thus, we have $\mathbb{C}[z, \bar{z}]\subset L^2(\mathbb{C}, \mathsf{m}_q)$.

\begin{remark}\label{rem:indeterminate}
    Although the orthogonal polynomials with respect to $x^ne_q(-x)d_qx$ are the $q$-Laguerre polynomials, the associated moment problem is indeterminate, i.e., an orthogonality measure for the $q$-Laguerre polynomials is not unique; see \cite{IR98, Mo81}. Moreover, $x^n e_q(-x)d_qx$ is not an extremal solution to the moment problem for every $n\geq 0$. In fact, by \cite[Theorem 14]{Mo81}, the support of any extremal solution must be the zeros of a certain entire function. However, the support of $x^n e_q(-x)d_qx$ is $\{q^k\mid k\in \mathbb{Z}\}$, which has an accumulation point at 0. Therefore, by \cite[Theorem H]{Mo81},  the subspace $\mathbb{C}[z,\bar{z}]$ of polynomials is not dense in $L^2(\mathbb{C}, \mathsf{m}_q)$.
\end{remark}

In virtue of Remark \ref{rem:indeterminate}, we define the \emph{$q$-deformed analytic Bargmann--Fock space}
\begin{equation*}
    \mathcal{A}_{1}^{q}
    :=\overline{\operatorname{span}}\{z^{n}:n\geq0\}
    \subset L^{2}(\mathbb{C},\mathsf{m}_{q}).
\end{equation*}
By \eqref{eqn:q mixed moments}, the analytic polynomials $\big\{q^{n(n+1)/4}z^{n}/\sqrt{[n]_{q}!}\big\}_{n\geq0}$
form an orthonormal basis of $\mathcal{A}_{1}^{q}$. In analogy to \eqref{eqn:analytic BF basis}, we define the \emph{$q$-deformed analytic Bargmann transform} $\mathcal{B}_{0}^{q}:\mathcal{H}\to \mathcal{A}_{1}^{q}$ as the unitary map determined by
\begin{equation}\label{eqn:q analytic transform basis}
    \mathcal{B}_{0}^{q}|n\rangle(z,\overline{z})
    =q^{\frac{n(n+1)}{4}}\frac{z^{n}}{\sqrt{[n]_{q}!}}\qquad
    (n\geq0),
\end{equation}
equivalently, as an analog of \eqref{eqn:analytic BF coherent rep},
\begin{equation*}
    [\mathcal{B}_{0}^{q}h](z,\overline{z})
    :=e_q(|z|^2)^{1/2} \langle \overline{h}| z, \overline{z}; 0\rangle
    \qquad (h\in\mathcal{H}).
\end{equation*}
Here, the regularized $q$-coherent state $|z, \overline{z}; 0\rangle$ is defined by
\begin{equation*}
    |z,\overline{z};0\rangle
    :=q^{\frac{\mathsf{N(N+1)}}{4}}|z\rangle
    =e_{q}(|z|^{2})^{-\frac{1}{2}}\sum_{n=0}^{\infty}q^{\frac{n(n+1)}{4}}\frac{z^{n}}{\sqrt{[n]_{q}!}}|n\rangle
\end{equation*}
and $0$ indicates the lowest Landau level.

Prior to the construction of the higher-order transformations, we record the intertwining property of the $q$-deformed analytic Bargmann transform. Let $\partial_{q,z}$ denote the partial $q$-derivative, i.e.,
\begin{equation*}
    \partial_{q,z}f(z,\overline{z})
    =\frac{f(z,\overline{z})-f(qz,\overline{z})}{(1-q)z},
\end{equation*}
and $\mathsf{S}_{q,z}$ be the $q$-shift operator
\begin{equation*}
    [\mathsf{S}_{q,z}f](z,\overline{z})
    =f(qz,\overline{z}).
\end{equation*}
Then, $q$-extension of the intertwining relation \eqref{eqn:intertwining classical} is given by
\begin{equation}\label{eq:q_intertwing}
    \mathcal{B}^{q}_{0}q^{-\frac{\mathsf{N}}{2}}a^{\ast}
    =z\mathcal{B}^{q}_{0}, \qquad
    \mathcal{B}^{q}_{0}q^{\frac{\mathsf{N}+1}{2}}a
    =\partial_{q,z}\mathcal{B}^{q}_{0}, \qquad 
    \mathcal{B}^{q}_{0}q^{\mathsf{N}}
    =\mathsf{S}_{q,z}\mathcal{B}^{q}_{0}.
\end{equation}

\subsection{Higher-order $q$-Bargmann transforms and $q$-raising maps}\label{sec:higher_q_raising_maps}
In this subsection, we construct higher-order $q$-Bargmann transforms in a way analogous to \eqref{eq:r Bargmann transform}. As a natural $q$-deformation of \eqref{eqn:r coherent state}, for $r\geq 0$ we define the \emph{order-$r$ $q$-coherent state} 
\begin{align*}
    |z, \overline{z}; r\rangle 
    := \frac{q^{\frac{1}{4}(\mathsf{N}+r+(\mathsf{N}-r)^2)}}{\sqrt{[r]_q!}}(a^{\ast}-\overline{z}q^\mathsf{N})^r|z\rangle.
\end{align*}
From $q^{\mathsf{N}}e_{q}(|z|^{2})^{1/2}|z\rangle=e_{q}(|qz|^{2})^{1/2}|qz\rangle$ and the $q$-binomial theorem for $q$-commuting variables \cite[Proposition 2.2]{KS97}, we have
\begin{equation*}
\begin{split}
    e_q(|z|^{2})^{1/2} |z,\overline{z};r\rangle 
    &=\frac{q^{(\mathsf{N}+r+(\mathsf{N}-r)^{2})/4}}{\sqrt{[r]_q!}}\sum_{k=0}^{r}(-1)^{k}\qbinom{r}{k}\overline{z}^{k}e_{q}(|q^{k}z|^{2})^{\frac{1}{2}}(a^{\ast})^{r-k}|q^{k}z\rangle\smallskip \\
    &=\frac{q^{(\mathsf{N}+r+(\mathsf{N}-r)^{2})/4}}{\sqrt{[r]_{q}!}}\sum_{k=0}^{r}(-1)^{r-k}\qbinom{r}{k}\overline{z}^{r-k}e_{q}(|q^{r-k}z|^{2})^{\frac{1}{2}}(a^{\ast})^{k}|q^{r-k}z\rangle.
\end{split}
\end{equation*}

Analogously to \eqref{eq:r Bargmann transform}, we introduce a $q$-deformation of the order-$r$ Bargmann transform $\mathcal{B}_r$ as follows:

\begin{definition}
    For $r\geq0$, the \emph{$q$-deformed order-$r$ Bargmann transform} is a linear map $\mathcal{B}_{r}^{q}:\mathcal{H}\to L^{2}(\mathbb{C},\mathsf{m}_{q})$ defined by
    \begin{equation*}
        [\mathcal{B}_{r}^{q}h](z, \overline{z})
        :=e_q(|z|^2)^{1/2}\langle \overline{h}| z, \overline{z}; r\rangle \qquad(h \in \mathcal{H}).
    \end{equation*}
\end{definition}

Since $q^{\frac{1}{4}(\mathsf{N}+r+(\mathsf{N}-r)^2)}= q^{\frac{r(r+1)}{4}}q^{-\frac{r}{2}\mathsf{N}}q^{\frac{1}{4}\mathsf{N}(\mathsf{N}+1)}$ and $q^{\frac{1}{4}\mathsf{N}(\mathsf{N}+1)}a^{\ast}=a^{\ast}q^{\frac{\mathsf{N}+1}{2}}q^{\frac{1}{4}\mathsf{N}(\mathsf{N}+1)}$, we have 
\begin{equation*}
    |z,\overline{z};r\rangle
    =\frac{q^{\frac{r(r+1)}{4}}}{\sqrt{[r]_{q}!}}q^{-\frac{r\mathsf{N}}{2}}(a^{\ast}q^{\frac{\mathsf{N}+1}{2}}-\overline{z}q^\mathsf{N})^{r}|z,\overline{z}; 0\rangle.
\end{equation*}
Thus, for every $h \in \mathcal{H}$ we have
\begin{equation*}
    [\mathcal{B}^{q}_{r}h](z,\overline{z})
    =\frac{q^{\frac{r(r+1)}{4}}}{\sqrt{[r]_{q}!}}[\mathcal{B}^{q}_{0}((q^{\frac{\mathsf{N}+1}{2}}a-\overline{z}q^{\mathsf{N}})^{r}q^{-\frac{r\mathsf{N}}{2}})h](z,\overline{z}).
\end{equation*}
Here, in analogy with \eqref{eqn:classical raising operator}, we introduce the $ q$-raising operator $\widetilde{\mathcal{R}}_{q}:=\partial_{q,z}-\overline{z}\mathsf{S}_{q,z}$. Similarly to \eqref{eqn:raising gauge}, $\widetilde{\mathcal{R}}_q$ admits the gauge representation
\begin{equation}\label{eqn:Rtilde gauge}
    \widetilde{\mathcal{R}}_{q}f(z, \overline{z})
    =E_{q}(|z|^{2})\partial_{q,z}\big[e_{q}(-|z|^{2})f(z, \overline{z})\big] \qquad (f\in \mathbb{C}[z, \overline{z}]).
\end{equation}

The following is a consequence of \eqref{eq:q_intertwing} and \eqref{eqn:Rtilde gauge}:
\begin{proposition}\label{prop:r_Bargmann_0}
    For every $r\geq 0$ we have
    \begin{equation}\label{eqn:BRSB}
        \mathcal{B}^{q}_{r}
        =\frac{q^{\frac{r(r+1)}{4}}}{\sqrt{[r]_{q}!}}\widetilde{\mathcal{R}}_{q}^{r}\mathsf{S}_{q,z}^{-\frac{r}{2}}\mathcal{B}^{q}_{0}.
    \end{equation}
    Moreover, if $h\in\mathcal{H}$ is a finite linear combination of the basis elements $\{|n\rangle\}_{n\geq0}$, then
    \begin{equation}\label{eqn:qraising gauge}
        \mathcal{B}_{r}^{q}h(z, \overline{z})
        =\frac{q^{\frac{r(r+1)}{4}}}{\sqrt{[r]_{q}!}}E_{q}(|z|^{2})\partial_{q,z}^{r}\big[e_{q}(-|z|^{2})\mathsf{S}^{-\frac{r}{2}}_{q,z}\mathcal{B}_{0}^{q}h\big](z, \overline{z}).
    \end{equation}
\end{proposition}

Let $\mathsf{J}_{q}:=\mathsf{S}_{q,z}\mathsf{S}_{q,\overline{z}}^{-1}$. Since $\mathsf{S}_{q,z}$ and $\mathsf{S}_{q,\overline{z}}$ commute and are invertible, the fractional power of $\mathsf{J}_{q}$ is well-defined:
\begin{equation*}
    \mathsf{J}_{q}^{\alpha}f(z,\overline{z})
    =f(q^{\alpha}z,q^{-\alpha}\overline{z})
\end{equation*}
for a fixed $\alpha\in\mathbb{R}$ and every polynomial $f\in \mathbb{C}[z, \overline{z}]$. By definition, we have $\mathsf{J}_{q}^\alpha\widetilde{\mathcal{R}}_{q}=q^{-\alpha}\widetilde{\mathcal{R}}_{q}\mathsf{J}_{q}^\alpha$ on $\mathbb{C}[z, \overline{z}]$.
    
Accordingly, we define the normalized $q$-raising operator
\begin{equation}\label{eqn:qraising}
    \mathcal{R}_{q}
    =\mathsf{J}_{q}^{-\frac{1}{2}}\widetilde{\mathcal{R}}_{q}
    =q^{\frac{1}{2}}\widetilde{\mathcal{R}}_{q}\mathsf{J}_{q}^{-\frac{1}{2}}.
\end{equation}
We note that the polyanalytic order of the image of $\mathcal{B}^q_0$ is zero. Thus, by Proposition \ref{prop:r_Bargmann_0}, we have 
\begin{equation}\label{eq:r_Bargmann_RB_0}
    \mathcal{B}^q_r= \frac{q^{\frac{r(r+1)}{4}}}{\sqrt{[r]_{q}!}}\widetilde{\mathcal{R}}_q^r \mathsf{J}_q^{-\frac{r}{2}}\mathcal{B}^q_0=\frac{1}{\sqrt{[r]_q!}}\mathcal{R}_q^r\mathcal{B}^q_0.
\end{equation}

In the remainder of this subsection, we prove that $\{\mathcal{B}^q_r|n\rangle\}_{n\geq 0}$ forms an orthonormal family in $L^2(\mathbb{C}, \mathsf{m}_q)$, and hence $\mathcal{B}^{q}_{r}$ is isometric. In preparation for this, we show the following basic property of $\mathcal{R}_{q}$.

\begin{proposition}\label{prop:R_properties}
    On $\mathbb{C}[z,\overline{z}]\subset L^{2}(\mathbb{C},\mathsf{m}_{q})$, the formal adjoint of $\mathcal{R}_{q}$ is given by
    \begin{equation}\label{eqn:qraising adjoint}
        \mathcal{R}_{q}^{\ast}
        =-\partial_{q,\overline{z}}\mathsf{S}_{q,z}^{-\frac{1}{2}}\mathsf{S}_{q,\overline{z}}^{-\frac{1}{2}} = -q^{-\frac{1}{2}}\mathsf{S}_{q,z}^{-\frac{1}{2}}\mathsf{S}_{q,\overline{z}}^{-\frac{1}{2}}\partial_{q,\overline{z}} .
    \end{equation}
    Moreover, $\mathcal{R}_{q}^{\ast}$ and $\mathcal{R}_{q}$ satisfy the $q$-CCR \eqref{eqn:qCCR}:
    \begin{equation}\label{eqn:qraising qCCR}
        \mathcal{R}_{q}^{\ast}\mathcal{R}_{q}-q\mathcal{R}_{q}\mathcal{R}_{q}^{\ast}
        =\Id.
    \end{equation}
\end{proposition}

\begin{proof}
    It suffices to verify the first assertion for monomials. By the mixed moment formula \eqref{eqn:q mixed moments}, we have 
    \begin{align*}
        \langle\mathcal{R}_{q}z^{a}\overline{z}^{b},z^{c}\overline{z}^{d}\rangle_{L^{2}(\mathbb{C}, \mathsf{m}_{q})}
        &=q^{-\frac{a-b-1}{2}}([a]_{q}\langle z^{a-1}\overline{z}^{b},z^{c}\overline{z}^{d}\rangle_{L^{2}(\mathbb{C},\mathsf{m}_{q})}-q^{a}\langle z^{a}\overline{z}^{b+1},z^{c}\overline{z}^{d}\rangle_{L^{2}(\mathbb{C},\mathsf{m}_{q})})\\
        &=-\delta_{a-b-1,c-d}q^{-\frac{c+d}{2}}[d]_{q}q^{-\frac{(a+d-1)(a+d)}{2}}[a+d-1]_{q}!\\
        &=-q^{-\frac{c+d}{2}}[d]_{q}\langle z^{a}\overline{z}^{b},z^{c}\overline{z}^{d-1}\rangle_{L^{2}(\mathbb{C},\mathsf{m}_{q})}\\
        &=\langle z^{a}\overline{z}^{b},(-\partial_{q,\overline{z}}\mathsf{S}_{q,z}^{-\frac{1}{2}}\mathsf{S}_{q,\overline{z}}^{-\frac{1}{2}})z^{c}\overline{z}^{d}\rangle_{L^{2}(\mathbb{C},\mathsf{m}_{q})}.
    \end{align*}
    For the second assertion, we have
    \begin{equation}\label{eqn:RRast}
    \begin{split}
        \mathcal{R}_{q}^{\ast}\mathcal{R}_{q}
        &=-\mathsf{S}_{q,z}^{-1}\partial_{q,z}\partial_{q,\overline{z}}+\Id+ q\overline{z}\partial_{q,\overline{z}},\smallskip\\
        \mathcal{R}_{q}\mathcal{R}_q^{\ast}
        &=-q^{-1}\mathsf{S}_{q,z}^{-1}\partial_{q,z}\partial_{q,\overline{z}}+\overline{z}\partial_{q,\overline{z}}.
    \end{split}
    \end{equation}
    Thus, it immediately follows that $\mathcal{R}_{q}^{\ast}\mathcal{R}_{q}-q\mathcal{R}_{q}\mathcal{R}_{q}^{\ast}=\Id$.
\end{proof}

\begin{proposition}
    The set $\{\mathcal{B}_{r}^{q}|n\rangle\}_{n,r\geq0}$ forms an orthonormal family in $L^{2}(\mathbb{C},\mathsf{m}_{q})$.
\end{proposition}

\begin{proof}
    By \eqref{eq:r_Bargmann_RB_0}, it suffices to show that for every $r, s, n, l\geq 0$
    \begin{equation*}
        \big\langle\mathcal{R}_{q}^{r}\mathcal{B}_{0}^{q}|n\rangle,
        \mathcal{R}_{q}^{s}\mathcal{B}_{0}^{q}|l\rangle\big\rangle_{L^{2}(\mathbb{C},\mathsf{m}_{q})}
        =[r]_{q}!\delta_{r,s}\delta_{n,l}.
    \end{equation*}
    We may assume $r>s$. Since the degree in $\overline{z}$ in $\mathcal{R}_{q}^{s}\mathcal{B}_{0}^{q}|l\rangle$ is $s$, by Proposition \ref{prop:R_properties}, we have $(\mathcal{R}_q^r)^{\ast} \mathcal{R}_{q}^{s}\mathcal{B}_{0}^{q}|l\rangle=0$. Thus, $\mathcal{R}_{q}^{r}\mathcal{B}_{0}^{q}|n\rangle$ and $\mathcal{R}_{q}^{s}\mathcal{B}_{0}^{q}|l\rangle$ are orthogonal. By symmetry, the same orthogonality holds when $r<s$. Therefore, it remains only to consider the case $r=s$.

    Inductively applying the $q$-CCR relation \eqref{eqn:qraising qCCR}, we obtain $\mathcal{R}_q^\ast\mathcal{R}_q^r= [r]_q\mathcal{R}_q^{r-1}+q^r\mathcal{R}_q^r \mathcal{R}_q^\ast$. Since $\mathcal{B}_{0}^{q}|l\rangle$ is analytic, we have $\mathcal{R}_q^\ast\mathcal{R}_q^r\mathcal{B}_{0}^{q}|l\rangle=[r]_q\mathcal{R}_q^{r-1}\mathcal{B}_{0}^{q}|l\rangle$, and hence
    \begin{equation*}
        \big\langle\mathcal{R}_{q}^{r}\mathcal{B}_{0}^{q}|n\rangle,\mathcal{R}_{q}^{r}\mathcal{B}_{0}^{q}|l\rangle\big\rangle_{L^{2}(\mathbb{C},\mathsf{m}_{q})}
        =[r]_{q}!\langle\mathcal{B}^{q}_{0}|n\rangle,\mathcal{B}^{q}_{0}|l\rangle \big\rangle
        =\delta_{n,l}[r]_{q}!.
    \end{equation*}
\end{proof}

Thus, the $q$-deformed order-$r$ Bargmann transform $\mathcal{B}^{q}_{r}$ is isometric. 

\begin{definition}
The range $\mathcal{A}_{r+1}^{q}:=\mathcal{B}^q_r\mathcal{H}$ of $\mathcal{B}^q_r$ is called the \emph{$q$-deformed pure $(r+1)$-analytic Bargmann--Fock space}. For $m\geq1$, the \emph{$q$-deformed full $m$-analytic Bargmann--Fock space} is 
\begin{equation*}
    \mathcal{A}_{m}^{q,\full}
    =\bigoplus_{r=0}^{m-1}\mathcal{A}_{r+1}^{q}
    \subset L^{2}(\mathbb{C},\mathsf{m}_{q}).
\end{equation*}
\end{definition}

Moreover, by the $q$-ladder relations 
\begin{equation}\label{eqn:ladder}
    \mathcal{R}_{q}^{\ast}\mathcal{B}_{r}^{q}|n\rangle
    =\sqrt{[r]_{q}}\mathcal{B}_{r-1}^{q}|n\rangle,\qquad
    \mathcal{R}_{q}\mathcal{B}_{r}^{q}|n\rangle
    =\sqrt{[r+1]_{q}}\mathcal{B}_{r+1}^{q}|n\rangle
\end{equation}
for fixed $n\geq0$, the family $\{\mathcal{B}_{r}^{q}|n\rangle\}_{r\geq0}$ admits the Fock representation of $\mathcal{R}_q^*$ and $\mathcal{R}_q$.

\subsection{$q$-deformed coherent states and the $q$-Laguerre polynomials}

For $\alpha>-1$ and $n\geq0$, the \emph{$q$-Laguerre polynomial} $L^{(\alpha)}_{n}(x;q)$ is defined by
\footnote{In this paper, we follow the original convention in \cite{Mo81}. In another convention, the $q$-Laguerre polynomials are defined as the polynomials replacing $(1-q)x$ by $x$.} 
\begin{equation}\label{eqn:qLagueere expansion}
    L^{(\alpha)}_{n}(x;q)
    :=\frac{(q^{\alpha+1}; q)_{n}}{(q; q)_{n}}\sum_{k=0}^{n}\frac{(q^{-n}; q)_{k} q^{k(k-1)/2}}{(q^{\alpha+1};q)_{k}[k]_{q}!}(q^{n  +\alpha+1}x)^{k}.
\end{equation}
It is an orthogonal polynomial with respect to the Jackson $q$-measure $x^{\alpha}e_{q}(-x)\,d_{q}x$ and satisfies the $q$-Rodrigues-type formula
\begin{equation}\label{eqn:qLaguerre Rodrigues}
    L_{n}^{(\alpha)}(x;q)
    =\frac{1}{x^{\alpha}[n]_{q}!}E_{q}(x)\partial_{q,x}^{n}\big[x^{\alpha+n}e_{q}(-x)\big];
\end{equation}
see, e.g., \cite[Eq. (14.21.12)]{KLS10}, after the rescaling $x\mapsto(1-q)x$.

The $q$-Bargmann basis $\mathcal{B}_{r}^{q}|n\rangle(z,\overline{z})$ has the representation in terms of the $q$-Laguerre polynomials.

\begin{proposition}
    For $r$, $n\geq0$, let
    \begin{equation*}
        C_{r,n}
        :=(-1)^{r-r\wedge n}q^{\frac{n+r+(n-r)^2}{4}}\frac{\sqrt{[r]_{q}![n]_{q}!}}{[r\vee n]_{q}!}
    \end{equation*}
    where $r\wedge n:=\min\{r, n\}$ and $r\vee n:=\max\{r, n\}$. Then
    \begin{equation}\label{eq:q_Laguerre}
        \mathcal{B}^{q}_{r}|n\rangle(z,\bar{z})
        =C_{r,n}z^{n-r\wedge n}\bar{z}^{r-r\wedge n}L^{(|n-r|)}_{r\wedge n}(|z|^{2};q),
    \end{equation}
    equivalently,
    \begin{equation}\label{eqn:q laguerre expansion}
        \mathcal{B}^{q}_{r}|n\rangle(z,\bar{z})
        =\frac{1}{\sqrt{[r]_{q}!}}q^{\frac{n+r+(n-r)^{2}}{4}}\sum_{j=0}^{r\wedge n}(-1)^{r-j}\qbinom{r}{j}\frac{\sqrt{[n]_{q}!}}{[n-j]_{q}!}q^{(r-j)(n-j)}z^{n-j}\overline{z}^{r-j}.
    \end{equation}
\end{proposition}

\begin{proof}
    Substituting \eqref{eqn:q analytic transform basis} into \eqref{eqn:qraising gauge} yields
    \begin{equation*}
        \mathcal{B}_{r}^{q}|n\rangle
        =\frac{1}{\sqrt{[r]_{q}![n]_{q}!}}q^{\frac{r(r+1)}{4}-\frac{rn}{2}+\frac{n(n+1)}{4}}E_{q}(|z|^{2})\partial_{q,z}^{r}\big[e_{q}(-|z|^{2})z^{n}\big].
    \end{equation*}
    The first assertion follows from the $q$-Rodrigues-type formula of the $q$-Laguerre polynomial \eqref{eqn:qLaguerre Rodrigues}. The second is obtained by \eqref{eqn:qLagueere expansion}.
\end{proof}

The $q$-Laguerre representation recovers the orthogonality of the $q$-Bargmann basis $\{\mathcal{B}^q_r|n\rangle\}_{n,  r\geq 0}$. Moreover, the explicit expansion \eqref{eqn:q laguerre expansion} implies that the highest degree in $z$ of the $q$-Bargmann basis $\mathcal{B}_{r}^{q}|n\rangle(z,\overline{z})$ is $n$, while that in $\overline{z}$ is $r$. This observation gives a concrete description of the polyanalytic Bargmann--Fock spaces.

\begin{proposition}
    For $m\geq1$,
    \begin{equation*}
        \mathcal{A}_{m}^{q,\full}
        =\overline{\operatorname{span}}\{z^{n}\overline{z}^{r}:n\geq0,\,0\leq r\leq m-1\},\qquad
        \mathcal{A}_{m}^{q}
        =\mathcal{A}_{m}^{q,\full}\ominus\mathcal{A}_{m-1}^{q,\full}
    \end{equation*}
    with the convention $\mathcal{A}_{0}^{q,\full}=\{0\}$.
\end{proposition}

As explained in Remark \ref{rem:indeterminate}, the space of polyanalytic polynomials is not dense in $L^{2}(\mathbb{C},\mathsf{m}_{q})$. Consequently, unlike the classical case, the closed span of all the $q$-Landau levels is a proper subspace of $L^{2}(\mathbb{C},\mathsf{m}_{q})$.

\subsection{$q$-deformed polyanalytic Ginibre processes}
In this subsection, we construct two types of $q$-deformed polyanalytic Ginibre processes and the corresponding finite particle ensembles, as the determinantal point processes associated with the $q$-deformed polyanalytic Bargmann--Fock spaces. For $m\geq1$, the \emph{$q$-deformed pure $m$-analytic Ginibre kernel} is defined by
\begin{equation*}
    \mathcal{K}_{m}^{q}(z,w)
    =\sum_{n=0}^{\infty}\mathcal{B}_{m-1}^{q}|n\rangle(z,\overline{z})\overline{\mathcal{B}_{m-1}^{q}|n\rangle(w,\overline{w})},
\end{equation*}
whereas the \emph{$q$-deformed full $m$-analytic Ginibre kernel} is
\begin{equation*}
    \mathcal{K}_{m}^{q,\full}(z,w)
    =\sum_{r=1}^{m}\sum_{n=0}^{\infty}\mathcal{B}_{r-1}^{q}|n\rangle(z,\overline{z})\overline{\mathcal{B}_{r-1}^{q}|n\rangle(w,\overline{w})}.
\end{equation*}
They are the reproducing kernels, equivalently, the orthogonal projection kernel onto $\mathcal{A}_{m}^{q}$ and $\mathcal{A}_{m}^{q,\full}$, respectively. For the analytic case $m=1$, the $q$-deformed Ginibre kernel can be written as
\begin{equation}\label{eqn:K1Eq}
    \mathcal{K}_{1}^{q}(z,w)
    =\sum_{n=0}^{\infty}\frac{1}{[n]_{q}!}q^{\frac{n(n+1)}{2}}(z\overline{w})^{n}
    =E_{q}(qz\overline{w}).
\end{equation}
More generally, the pure kernel has a representation in terms of the basic hypergeometric series. Recall that the basic hypergeometric series or the $q$-hypergeometric series is defined by
\begin{equation*}
    {}_{r}\phi_{s}\Big(\begin{matrix}a_{1},\cdots,a_{r}\\b_{1},\cdots,b_{s}\end{matrix};q,z\Big)
    =\sum_{k=0}^{\infty}\frac{(a_{1},\cdots,a_{r};q)_{k}}{(b_{1},\cdots,b_{s},q;q)_{k}}
    (-1)^{(1+s-r)k}q^{(1+s-r)\frac{k(k-1)}{2}}z^{k},
\end{equation*}
see, e.g., \cite{GR04,KLS10}.

\begin{theorem}
    For $m\geq1$ and $z,w\in\supp(\mathsf{m}_{q})\backslash\{0\}$, we have
    \begin{equation}\label{eqn:kernel hypergeometric}
        \mathcal{K}_{m}^{q}(z,w)
        =q^{m-1}E_{q}(qz\overline{w})
        {}_{3}\phi_{2}\Big(\begin{matrix} q^{1-m}, \, qz/w, \, q\bar{w}/\bar{z}\\ q, \, -(1-q)qz\bar{w}\end{matrix}; q, -(1-q)q^{m-1}w\bar{z}\Big).
    \end{equation}
    In particular, the diagonal of the kernel admits the representation
    \begin{equation}\label{eqn:q kernel diagonal}
        \mathcal{K}_{m}^{q}(z,z)
        =\frac{q^{m-1}}{1+(1-q)q^{m-1}|z|^{2}}E_{q}(|z|^{2}).
    \end{equation}
\end{theorem}

\begin{proof}
    Since the series defining $\mathcal{K}_{1}^{q}$ converges locally uniformly, the higher-order kernel is written as
    \begin{equation*}
        \mathcal{K}_{m}^{q}(z,w)
        =\frac{1}{[m-1]_{q}!}\mathcal{R}_{q,z}^{m-1}\mathcal{R}_{q,\overline{w}}^{m-1}\mathcal{K}_{1}^{q}(z,w).
    \end{equation*}
    Substituting \eqref{eqn:K1Eq} and the gauge representation of $\mathcal{R}_{q,z}$, we obtain
    \begin{equation*}
        \mathcal{K}_{m}^{q}(z,w)
        =q^{\frac{m(m-1)}{2}}\frac{1}{[m-1]_{q}!}E_{q}(|z|^{2})E_{q}(|w|^{2})\partial_{q,z}^{m-1}\partial_{q,\overline{w}}^{m-1}\big[e_{q}(-|z|^{2})e_{q}(-|w|^{2})E_{q}(q^{2-m}z\overline{w})\big].
    \end{equation*}
    We first evaluate $q$-difference operators in $z$-variable. Using
    \begin{equation*}
        \partial_{q,z}e_{q}(z)
        =e_{q}(z),\qquad
        \partial_{q,z}E_{q}(z)
        =E_{q}(qz),
    \end{equation*}
    and the $q$-Leibniz rule, we have
    \begin{equation*}
    \begin{split}
        \partial_{q,z}^{m-1}\big[e_{q}(-|z|^{2})E_{q}(q^{2-m}z\overline{w})\big]
        &=e_{q}(-|z|^{2})E_{q}(qz\overline{w})\sum_{j=0}^{m-1}\qbinom{m-1}{j}q^{j(2-m)+\frac{j(j-1)}{2}}(-\overline{z})^{m-1-j}\overline{w}^{j}\\
        &=e_{q}(-|z|^{2})E_{q}(qz\overline{w})(-\overline{z})^{m-1}(q^{2-m}\overline{w}/\overline{z};q)_{m-1}.
    \end{split}
    \end{equation*}
    In the last equality, the finite $q$-binomial theorem is applied. For $\overline{w}$-variable, using
    \begin{equation*}
        \partial_{q,x}^{j}(ax;q)_{r}
        =(-a)^{j}q^{\frac{j(j-1)}{2}}\frac{[r]_{q}!}{[r-j]_{q}!}(aq^{j}x;q)_{r-j},\qquad
        \partial_{q,x}^{j}
        \frac{(ax;q)_{\infty}}{(bx;q)_{\infty}}
        =\Big(\frac{b}{1-q}\Big)^{j}(a/b;q)_{j}\frac{(aq^{j}x;q)_{\infty}}{(bx;q)_{\infty}},
    \end{equation*}
    we have
    \begin{equation*}
    \begin{split}
        &\partial_{q,\overline{w}}^{m-1}\big[e_{q}(-|w|^{2})E_{q}(qz\overline{w})(q^{2-m}\overline{w}/\overline{z};q)_{m-1}\big]\smallskip\\
        &=\frac{(-q(1-q)z\overline{w};q)_{\infty}}{(-(1-q)|w|^{2};q)_{\infty}}\sum_{j=0}^{m-1}\qbinom{m-1}{j}\frac{[m-1]_{q}!}{[j]_{q}!}q^{\frac{(m-j-1)(m-j-2)}{2}}\Big(-\frac{q^{2-m}}{\overline{z}}\Big)^{m-1-j}(-w)^{j}\frac{(qz/w,q\overline{w}/\overline{z};q)_{j}}{(-q(1-q)z\overline{w};q)_{j}}.
    \end{split}
    \end{equation*}
    In summary, we obtain
    \begin{equation*}
        \mathcal{K}_{m}^{q}(z,w)
        =q^{m-1}E_{q}(qz\overline{w})\sum_{k=0}^{m-1}\frac{(q^{1-m};q)_{k}(qz/w;q)_{k}(q\overline{w}/\overline{z};q)_{k}}{(q;q)_{k}^{2}(-(1-q)qz\overline{w};q)_{k}}\big(-(1-q)q^{m-1}w\overline{z}\big)^{k},
    \end{equation*}
    which is equivalent to \eqref{eqn:kernel hypergeometric}.

    For the second assertion, letting $w=z$, the $q$-Chu--Vandermonde identity simplifies the basic hypergeometric function as
    \begin{equation*}
    \begin{split}
        {}_{3}\phi_{2}\Big(\begin{matrix} q^{1-m}, \, q, \, q\\ q, \, -(1-q)q|z|^{2}\end{matrix}; q, -(1-q)q^{m-1}|z|^{2}\Big)
        &={}_{2}\phi_{1}\Big(\begin{matrix} q^{1-m}, \, q\\ -(1-q)q|z|^{2}\end{matrix}; q, -(1-q)q^{m-1}|z|^{2}\Big)\smallskip\\
        &=\frac{1+(1-q)|z|^{2}}{1+(1-q)q^{m-1}|z|^{2}}.
    \end{split}
    \end{equation*}
    Then \eqref{eqn:q kernel diagonal} follows from
    \begin{equation*}
        \big(1+(1-q)|z|^{2}\big)E_{q}(q|z|^{2})=E_{q}(|z|^{2}).
    \end{equation*}
    This completes the proof.
\end{proof}

\begin{remark}
    Since the basic hypergeometric series in \eqref{eqn:kernel hypergeometric} terminates, the apparent singularities at $z=0$ or $w=0$ are removable. More precisely, by \eqref{eqn:kernel hypergeometric}, we immediately obtain, for all $z, w\in \mathbb{C}$
    \begin{equation}\label{eq:kernel_finite_sum}
        \mathcal{K}^q_m(z, w) = q^{m-1} \sum_{k=0}^{m-1}E_q(q^{k+1}z\overline{w}) \frac{(q^{1-m}; q)_k(-(1-q)q^{m-1})^k}{(q; q)_k^2} \prod_{l=1}^k(w-q^l z)(\overline{z}-q^l \overline{w}).
    \end{equation}
\end{remark}

For every region $D\subset \mathbb{C}$ we define the compression $\mathcal{K}^q_{m, D}$ by $\mathcal{K}^q_{m, D}(z, w):=\mathbbm{1}_D(z)\mathcal{K}^q_m(z, w)\mathbbm{1}_D(w)$. For every $R>0$ we denote by $\mathbb{D}_{R}$ the disk of radius $R$, i.e., $\mathbb{D}_{R}:=\{z\in \mathbb{C}\mid |z|\leq R\}$.

\begin{proposition}\label{prop:locally_trace_class}
    The integral operator of $\mathcal{K}^q_m$ acting on $L^2(\mathbb{C}, \mathsf{m}_q)$ is locally trace class, i.e., for every compact set $D\subset \mathbb{C}$ the integral operator of $\mathcal{K}^q_{m, D}$ is trace class.
\end{proposition}
\begin{proof}
    It suffices to show that $\mathcal{K}^q_{m, \mathbb{D}_{R}}$ gives a trace class operator for every $R>0$. By \eqref{eqn:q kernel diagonal}, we have 
    \begin{align*}
        \operatorname{Tr}(\mathcal{K}^q_{m, \mathbb{D}_{R}})
        =\int_{\mathbb{D}_{R}}\mathcal{K}^q_{m}(z,z)\,d\mathsf{m}_q(z)
        &=\int_0^\infty \mathbbm{1}_{[0,R^{2}]}(r^2)\frac{q^{m-1}}{1+(1-q)q^{m-1}r^{2}}\,d_{q}r^{2}\\
        &<q^{m-1}\int_0^\infty \mathbbm{1}_{[0,R^{2}]}(r^{2})\,d_{q}r^{2}
        <q^{m-1}R^{2}m
    \end{align*}
    where we used $E_q(r^2)e_q(-r^2)=1$. This completes the proof.
\end{proof}

By Proposition \ref{prop:locally_trace_class}, the integral operators of $\mathcal{K}^q_m$, and therefore $\mathcal{K}^{q,\full}_{m}$ acting on $L^2(\mathsf{m}_q)$ are locally trace class and orthogonal projections. Thus, these operators give rise to the associated determinantal point process on $\mathbb{C}$; see \cite{ST03, So00}.

\begin{theorem}\label{thm:existence}
    For every $m\geq 1$, each of the kernels $\mathcal{K}^q_m$ and $\mathcal{K}^{q,\full}_m$ is the correlation kernel of a determinantal point process on $\mathbb{C}$ with reference measure $\mathsf{m}_{q}$. We call them the \emph{$q$-deformed pure} and \emph{full $m$-analytic Ginibre point processes}, respectively.
\end{theorem}

The following result provides further justification for the terminology introduced in Theorem \ref{thm:existence}:
\begin{proposition}
    For every $m\geq 1$, as $q\uparrow 1$, the $q$-deformed pure and full $m$-analytic Ginibre point processes converge weakly to their classical counterparts.
\end{proposition}
\begin{proof}
    It suffices to prove the locally uniform convergence of the correlation kernels and the vague convergence of the reference measures; see, e.g., \cite[Proposition 2.15]{Di26}. Alternatively, the assertion follows directly from the fact that these convergences imply the convergence of the corresponding Laplace functionals.

    We first consider the pure kernel. In \eqref{eq:kernel_finite_sum}, for each $k=0, \dots, m-1$, we have
    \begin{equation*}
        \frac{(q^{1-m};q)_{k}(-(1-q)q^{m-1})^{k}}{(q;q)_{k}^{2}}
        \to\frac{(m-1)!}{(k!)^{2}(m-1-k)!},\quad
        E_{q}(q^{k+1}z\overline{w})
        \to e^{z\overline{w}},\quad
        \prod_{l=1}^{k}(w-q^{l}z)(\overline{z}-q^{l}\overline{w})\to(-|z-w|^{2})^{k}
    \end{equation*}
    as $q\uparrow 1$, where the latter two convergences are locally uniform on $\mathbb{C}^2$. Thus, it follows that as $q\uparrow 1$
    \begin{align*}
        \mathcal{K}^q_m(z, w) 
        &\to e^{z\overline{w}}\sum_{k=0}^{m-1} \frac{(m-1)!}{(k!)^2(m-1-k)!} (-|z-w|^2)^k
        =e^{z\overline{w}}L^{(0)}_{m-1}(|z-w|^2)
    \end{align*}
    locally uniformly on $\mathbb{C}^2$. The last expression is precisely the correlation kernel of the pure $m$-analytic Ginibre point process. For the full kernel, as $q\uparrow 1$ we have 
    \begin{align*}
        \mathcal{K}^{q, \textup{full}}_m(z, w)
        = \sum_{r=1}^{m} \mathcal{K}^q_r(z, w)
        \to e^{z\overline{w}}\sum_{r=1}^{m} L^{(0)}_{r-1}(|z-w|^2)
        = e^{z\overline{w}} L^{(1)}_{m-1}(|z-w|^2)
    \end{align*}
    locally uniformly on $\mathbb{C}^2$. This is precisely the correlation kernel of the full $m$-analytic Ginibre point process. 

    It remains to prove the vague convergence of the reference measures. For every compactly supported continuous function $f$ on $\mathbb{C}$, we have 
    \[\int_\mathbb{C} f(z) d\mathsf{m}_q(z)= (1-q) \sum_{j\in \mathbb{Z}} F(q^j)e_q(-q^j)q^j,\]
    where $F(r):= \frac{1}{2\pi} \int_0^{2\pi}f(\sqrt{r}e^{i\theta})d\theta$, and $F$ is a compactly supported continuous function on $[0, \infty)$. We remark that for every compactly supported Riemann-integrable function $g$ on $[0, \infty)$, 
    \[\lim_{q\uparrow 1} \sum_{j\in \mathbb{Z}} g(q^j)(q^j-q^{j+1})= \int_0^\infty g(x)dx.\]
    Moreover, $e_q(-x)$ converges locally uniformly to $e^{-x}$ as $q\uparrow 1$. Combining these observations, we obtain 
    \[\lim_{q\uparrow 1} \int_\mathbb{C} f(z)d\mathsf{m}_q(z)= \int_0^\infty F(x)e^{-x}dx = \int_\mathbb{C} f(z) d\mathsf{m}(z),\]
    where we recall that $d\mathsf{m}(z)= e^{-|z|^2}dA(z)$. Therefore, $\mathsf{m}_q$ converges vaguely to $\mathsf{m}$, which completes the proof.
\end{proof}

For each $N\geq1$, we consider finite-rank truncations \eqref{eqn:qGinibre pure ensemble kernel} and \eqref{eqn:qGinibre full ensemble kernel} so that they are the kernels of finite-rank orthogonal projections onto the linear span of $\{\mathcal{B}^{q}_{m-1}|n\rangle\}_{0\leq n\leq N-1}$ and $\{\mathcal{B}^{q}_{r-1}|n\rangle\}_{0\leq n\leq N-1,1\leq r\leq m}$, which have rank $N$ and $mN$, respectively. Hence, these kernels define the finite-particle ensembles.

\begin{definition}
    The determinantal point processes on $\mathbb{C}$ with reference measure $\mathsf{m}_{q}$ and correlation kernels $\mathcal{K}^{q}_{m,N}$ and $\mathcal{K}^{q,\full}_{m,N}$ are called the \emph{$q$-deformed pure} and \emph{full $m$-analytic Ginibre ensembles}, respectively. 
\end{definition}

\begin{remark}[A Kostlan-type representation]
    The pure polyanalytic ensembles admit a Kostlan-type representation, analogous to the classical result in \cite[Theorem 6.2]{AGR19}.

    Fix $m,N\geq1$ and set $r=m-1$. Let $z_{1},\cdots,z_{N}$ denote the particles of the $q$-deformed pure $m$-analytic Ginibre ensemble. Then the radial distribution of particles $\{z_{j}\}_{1\leq j\leq N}$ satisfies
    \begin{equation*}
        \sum_{j=1}^{N}\delta_{|z_{j}|^{2}}
        \overset{\textup{d}}{=}\sum_{n=0}^{N-1}\delta_{Q_{r,n}^{(q)}},
    \end{equation*}
    where $Q_{r,0}^{(q)},\cdots,Q_{r,N-1}^{(q)}$ are independent random variables with distributions
    \begin{equation*}
        \mathbb{P}\big(Q_{r,n}^{(q)}\in dx\big)
        =|C_{r,n}|^{2}x^{|n-r|}\Big(L_{r\wedge n}^{(|n-r|)}(x;q)\Big)^{2}e_{q}(-x)\,d_{q}x,\qquad
        0\leq n\leq N-1.
    \end{equation*}
    Here, $C_{r,n}$ is the coefficient in \eqref{eq:q_Laguerre}. These are probability measures by the normalization of $\mathcal{B}_{r}^{q}|n\rangle$.

    Indeed, by \eqref{eq:q_Laguerre}, the polyanalytic $q$-Bargmann basis $\{\mathcal{B}_{r}^{q}|n\rangle\}_{n=0}^{N-1}$ remains mutually orthogonal when restricted to any annulus centered at the origin. Thus, such annuli are simultaneously observable in the sense of \cite[Section 6.2]{AGR19}. For $A=\{z\in\mathbb{C}:a<|z|^2\leq b\}$, the corresponding eigenvalues of the restricted kernel are
    \begin{equation*}
        \int_{A}\big|\mathcal{B}_{r}^{q}|n\rangle(z,\overline{z})\big|^{2}\,d\mathsf{m}_{q}(z)
        =\mathbb{P}\big(a<Q_{r,n}^{(q)}\leq b\big).
    \end{equation*}
    Consequently, \cite[Proposition 4.5.9]{HKPV09} identifies the number of particles in any annulus centered at the origin with the corresponding counts of the independent variables $Q_{r,n}^{(q)}$.

    In the analytic case $r=0$, the distributions reduce to
    \begin{equation*}
        \mathbb{P}\big(Q_{0,n}^{(q)}\in dx\big)
        =\frac{q^{n(n+1)/2}}{[n]_{q}!}x^{n}e_{q}(-x)\,d_{q}x,
    \end{equation*}
    which converges to the gamma distribution $\Gamma(n+1,1)$ as $q\uparrow1$, recovering the classical Kostlan representation.
\end{remark}
\subsection{Algebraic identities}

We record some algebraic properties of the $q$-deformed Bargmann transform.

\begin{proposition}\label{prop:three_term_recurrence}
    For every $r, n\geq 0$ we have 
    \begin{align}
        q^{n}\overline{z}\mathcal{B}^{q}_{r}|n\rangle(z,\overline{z})
        &=-q^{\frac{n-r-1}{2}}\sqrt{[r+1]_{q}}\mathcal{B}^{q}_{r+1}|n\rangle(z,\overline{z})
        +q^{\frac{n-r}{2}}\sqrt{[n]_{q}}\mathcal{B}^{q}_{r}|n-1\rangle(z,\overline{z}),\label{eqn:threeterm barz}\smallskip\\
        q^{r}z\mathcal{B}^{q}_{r}|n\rangle(z,\overline{z})
        &=q^{-\frac{n-r+1}{2}}\sqrt{[n+1]_{q}}\mathcal{B}^{q}_{r}|n+1\rangle(z,\overline{z})
        -q^{-\frac{n-r}{2}}\sqrt{[r]_{q}}\mathcal{B}^{q}_{r-1}|n\rangle(z,\overline{z})\label{eqn:threeterm z}.
    \end{align}
\end{proposition}

\begin{proof}
    We first note that
    \begin{equation}\label{eqn:Rbarz}
        \mathcal{R}_{q}\overline{z}
        =q^{\frac{1}{2}}\overline{z}\mathcal{R}_{q}.
    \end{equation}
    For the analytic basis, a direct calculation gives
    \begin{equation*}
        \partial_{q,z}\mathcal{B}_{0}^{q}|n\rangle
        =q^{\frac{n}{2}}\sqrt{[n]_{q}}\mathcal{B}_{0}^{q}|n-1\rangle,
        \qquad
        \mathsf{S}_{q,z}\mathcal{B}_{0}^{q}|n\rangle
        =q^{n}\mathcal{B}_{0}^{q}|n\rangle,
    \end{equation*}
    which yields
    \begin{equation*}
        \mathcal{R}_{q}\mathcal{B}_{0}^{q}|n\rangle
        =\mathsf{J}_{q}^{-\frac{1}{2}}\big(\partial_{q,z}-\overline{z}\mathsf{S}_{q,z}\big)\mathcal{B}_{0}^{q}|n\rangle
        =q^{\frac{1}{2}}\sqrt{[n]_{q}}\mathcal{B}_{0}^{q}|n-1\rangle
        -q^{\frac{n+1}{2}}\overline{z}\mathcal{B}_{0}^{q}|n\rangle.
    \end{equation*}
    Therefore,
    \begin{equation*}
        q^{n}\overline{z}\mathcal{B}_{0}^{q}|n\rangle
        =-q^{\frac{n-1}{2}}\mathcal{R}_{q}\mathcal{B}_{0}^{q}|n\rangle+q^{\frac{n}{2}}\sqrt{[n]_{q}}\mathcal{B}_{0}^{q}|n-1\rangle.
    \end{equation*}
    By multiplying $\mathcal{R}_{q}^{r}$ on the left to both sides and iterating \eqref{eqn:Rbarz}, we obtain
    \begin{equation*}
        q^{n+\frac{r}{2}}\overline{z}\mathcal{B}_{r}^{q}|n\rangle
        =-q^{\frac{n-1}{2}}\sqrt{[r+1]_{q}}\mathcal{B}_{r+1}^{q}|n\rangle+q^{\frac{n}{2}}\sqrt{[n]_{q}}\mathcal{B}_{r}^{q}|n-1\rangle,
    \end{equation*}
    which is equivalent to the first assertion.

    Since the multiplication operator by $\overline{z}$ is the adjoint of the multiplication by $z$ and $\{\mathcal{B}_{r}^{q}|n\rangle\}_{n,r\geq0}$ forms an orthonormal system, the second relation follows directly from the first.
\end{proof}

As a $q$-counterpart of \eqref{eqn:mathcalL}, we introduce the $q$-Landau operator
\begin{equation}\label{eqn:qLandau}
    \mathcal{L}_{q}
    =\mathcal{R}_{q}\mathcal{R}_{q}^{\ast},
\end{equation}
which is defined on the $L^{2}(\mathbb{C},\mathsf{m}_{q})$-completion of the space of polyanalytic polynomials. The following proposition explains the meaning of the $q$-deformed Landau level in analogy with \eqref{eqn:Landau}.

\begin{proposition}\label{prop:qLandau}
    For each $m\geq1$, the $q$-deformed pure $m$-analytic Bargmann--Fock space is the eigenspace of the $q$-Landau operator \eqref{eqn:qLandau} corresponding to eigenvalue $[m-1]_{q}$, that is,
    \begin{equation*}
        \mathcal{A}_{m}^{q}
        =\ker(\mathcal{L}_{q}-[m-1]_{q}).
    \end{equation*}
\end{proposition}

\begin{proof}
    Since $\mathcal{A}_{m}^{q}$ is spanned by $\{\mathcal{B}_{m-1}^{q}|n\rangle\}_{n\geq0}$, it suffices to establish 
    \begin{equation}\label{eqn:LBrB}
        \mathcal{L}_{q}\mathcal{B}_{r}^{q}
        =[r]_{q}\mathcal{B}_{r}^{q}
    \end{equation}
    for each $r\geq0$. Applying ladder relations \eqref{eqn:ladder}, we obtain
    \begin{equation*}
        \mathcal{R}_{q}\mathcal{R}_{q}^{\ast}\mathcal{B}_{r}^{q}
        =\sqrt{[r]_{q}}\mathcal{R}_{q}\mathcal{B}_{r-1}^{q}
        =[r]_{q}\mathcal{B}_{r}^{q},
    \end{equation*}
    as desired.
\end{proof}

\begin{remark}
    The operator $\mathcal{L}_q$ gives a concrete $q$-difference realization of the $q$-oscillator Hamiltonian appearing in the $(p, q)$-deformed Landau problem studied in \cite{BGGH}. Indeed, after a suitable normalization, the Hamiltonian at $p=1$ in \cite{BGGH} is unitarily equivalent to $1+(1+q)\mathcal{L}_q$.
\end{remark}

From the $q$-Laguerre polynomial representation \eqref{eq:q_Laguerre} of the $q$-Bargmann vector $\mathcal{B}^{q}_{m}|n\rangle$, Propositions \ref{prop:three_term_recurrence} and \ref{prop:qLandau} naturally induce the algebraic properties of the $q$-Laguerre polynomial. As a direct consequence of Proposition \ref{prop:three_term_recurrence}, we have 
\begin{align*}
        q^{n+m}|z|^{2}\mathcal{B}^{q}_{m}|n\rangle(z,\overline{z})
        &=-q^{-\frac{3}{2}}\sqrt{[m+1]_{q}[n+1]_{q}}\mathcal{B}^{q}_{m+1}|n+1\rangle (z,\overline{z})\\
        &\quad+([m]_{q}+q^{-1}[n+1]_{q})\mathcal{B}^{q}_{m}|n\rangle(z,\overline{z})-q^{\frac{1}{2}}\sqrt{[m]_{q}[n]_{q}}\mathcal{B}^{q}_{m-1}|n-1\rangle(z,\overline{z}).
\end{align*}
This identity implies the three-term recurrence relation of the $q$-Laguerre polynomials; see \cite[Section 3]{Mo81}.

On the other hand, using \eqref{eqn:RRast}, the relation \eqref{eqn:LBrB} is equivalent to
\begin{equation*}
    -q^{-1}\mathsf{S}_{q,z}^{-1}\partial_{q,z}\partial_{q,\overline{z}}+\overline{z}\partial_{q,\overline{z}}
    =[r]_{q}
\end{equation*}
on the dense subspace $\mathrm{span}\{\mathcal{B}^q_r|n\rangle \mid n\geq 0\}$ of $\mathcal{A}^q_{r+1}$. This recovers the $q$-difference equation of the $q$-Laguerre polynomials; see \cite[Section 6]{Mo81}.

\section{Proof of Theorem \ref{thm:double scale limiting law}}\label{sec:double_scaled_regime}

Throughout this and the subsequent sections, we work in the double-scaling regime
\begin{equation*}
    q=e^{-\lambda/N},
\end{equation*}
where $\lambda>0$ is fixed. The asymptotic analysis is formulated in terms of the rescaled reference measure \eqref{eqn:def of mqhat} and the rescaled kernels \eqref{eqn:def of Khat}.

The proof of Theorem \ref{thm:double scale limiting law} proceeds in two steps. We first identify the weak limits of the averaged empirical measures using the method of moments. We then use the projection structure of the correlation kernels to show that the random empirical measures concentrate around their means, thereby obtaining convergence in probability.

We note that the integral formulas regarding the rescaled measure $\widehat{\mathsf{m}}_{q}$ are not obtained from those of $\mathsf{m}_{q}$ by a change of variables. Indeed, the dilation $z\mapsto\sqrt{1-q}z$ does not preserve the Jackson radial lattice. Nevertheless, the two measures satisfy an exact scaling identity on polynomials, which is sufficient for the finite-rank kernel calculations below.

\begin{lemma}
    For nonnegative integers $n$, $m$,
    \begin{equation}\label{eqn:scaled mixed moments}
        \int_{\mathbb{C}}z^{n}\overline{z}^{m}\,d\widehat{\mathsf{m}}_{q}(z)
        =\delta_{n,m}(1-q)^{n+1}q^{-\frac{n(n+1)}2}[n]_{q}!.
    \end{equation}
    Consequently, if $f$ is a polynomial in $z$ and $\overline{z}$,
    \begin{equation}\label{eqn:polynomial scaling}
        \frac{1}{1-q}\int_{\mathbb{C}}f\Big(\frac{z}{\sqrt{1-q}},\frac{\overline{z}}{\sqrt{1-q}}\Big)d\widehat{\mathsf{m}}_{q}(z)
        =\int_{\mathbb{C}}f(z,\overline{z})\,d\mathsf{m}_{q}(z).
\end{equation}
\end{lemma}

\begin{proof}
    Comparing the first assertion with \eqref{eqn:q mixed moments} and using linearity yields the second assertion. It therefore suffices to prove the first.
    Since $\widehat{\mathsf{m}}_{q}$ has radial symmetry, its mixed moment vanishes unless $n=m$. Let
    \begin{equation*}
        M_{n}
        =\int_{\mathbb{C}}|z|^{2n}\,d\widehat{\mathsf{m}}_{q}(z)
        =\int_{0}^{\infty}\frac{x^{n}}{(-x;q)_{\infty}}\,d_{q}x.
    \end{equation*}
    Using $(1+x)(-qx;q)_{\infty}=(-x;q)_{\infty}$ and the change-of-variable formula for Jackson $q$-integral, we obtain
    \begin{equation*}
        M_{n}
        =q^{n+1}\int_{0}^{\infty}\frac{1+x}{(-x;q)_{\infty}}x^{n}\,d_{q}x,
    \end{equation*}
    which gives the recurrence relation
    \begin{equation}\label{eqn:spectral moment recurrence}
        (1-q^{n+1})M_{n}=q^{n+1}M_{n+1}.
    \end{equation}
    Then \eqref{eqn:scaled mixed moments} follows from applying \eqref{eqn:spectral moment recurrence} repeatedly to $M_{0}=1-q$.
\end{proof}

\subsection{Spectral moments}
For nonnegative integers $p_{1}$, $p_{2}$, the \emph{mixed spectral moment} of the rescaled $q$-deformed $m$-analytic Ginibre ensemble is defined by
\begin{equation*}
    \mathfrak{m}_{m,N}^{q}(p_{1},p_{2})
    :=\mathbb{E}\Big[\sum_{j=1}^{N}z_{j}^{p_{1}}\overline{z}_{j}^{p_{2}}\Big],
\end{equation*}
where $z_{1}$, $\cdots$, $z_{N}$ denote the particles of the corresponding ensemble, and $\mathbb{E}$ denotes the expectation with respect to their joint law. The mixed spectral moment of the full ensemble is defined in a similar way:
\begin{equation*}
    \mathfrak{m}_{m,N}^{q,\full}(p_{1},p_{2})
    :=\mathbb{E}\Big[\sum_{j=1}^{mN}z_{j}^{p_{1}}\overline{z}_{j}^{p_{2}}\Big]
    =\sum_{r=1}^{m}\mathfrak{m}_{r,N}^{q}(p_{1},p_{2}).
\end{equation*}
Let
\begin{equation}\label{eqn:averaged empirical measure}
    \overline{\mu}_{m,N}^{q}
    =\mathbb{E}\big[\widehat{\mu}_{m,N}^{q}\big],\qquad
    \overline{\mu}_{m,N}^{q,\full}
    :=\mathbb{E}\big[\widehat{\mu}_{m,N}^{q,\full}\big]
\end{equation}
be averaged empirical measures of polyanalytic Ginibre ensembles. Due to the determinantal structure, \eqref{eqn:averaged empirical measure} admit the kernel representations
\begin{equation}
    \overline{\mu}_{m,N}^{q}(dz)
    =\frac{1}{N}\widehat{\mathcal{K}}_{m,N}^{q}(z,z)\,d\widehat{\mathsf{m}}_{q}(z),\qquad
    \overline{\mu}_{m,N}^{q,\full}(dz)
    =\frac{1}{mN}\widehat{\mathcal{K}}_{m,N}^{q,\full}(z,z)\,d\widehat{\mathsf{m}}_{q}(z).
\end{equation}
This yields the integral representation of the mixed spectral moments
\begin{equation*}
    \mathfrak{m}_{m,N}^{q,\#}(p_{1},p_{2})
    =\int_{\mathbb{C}}z^{p_{1}}\overline{z}^{p_{2}}\widehat{\mathcal{K}}_{m,N}^{q,\#}(z,z)\,d\widehat{\mathsf{m}}_{q}(z),
\end{equation*}
where \# is omitted or full. The following exact formula is obtained by combining the polynomial scaling identity \eqref{eqn:polynomial scaling} with the three-term relation in Proposition \ref{prop:three_term_recurrence}.

\begin{theorem}\label{thm:mixed moment}
    For nonnegative integers $p_{1}$, $p_{2}\geq 0$ and $m\geq1$, we have
    \begin{equation}\label{eqn:spectral moment formula}
        \mathfrak{m}_{m,N}^{q}(p_{1},p_{2})=
        \begin{cases}
            \displaystyle (1-q)^{p}q^{-\frac{p(2m+p-1)}{2}}\sum_{l=0}^{p}q^{l^{2}}\qbinom{p}{l}^{2}\frac{[l]_{q}![m+p-l-1]_{q}!}{[m-1]_{q}!}\sum_{j=0}^{N-1}q^{-pj}\qbinom{j}{l}, & \textup{if }p_{1}=p_{2}=:p,\smallskip \\
            0, & \textup{otherwise.}
        \end{cases}
    \end{equation}
\end{theorem}

\begin{proof}
    From the basis expansion of the diagonal of the projection kernel
    \begin{equation*}
        \widehat{\mathcal{K}}_{m,N}^{q}(z,z)
        =\frac{1}{1-q}\sum_{j=0}^{N-1}\Big|\mathcal{B}_{m-1}^{q}|j\rangle\Big(\frac{z}{\sqrt{1-q}},\frac{\overline{z}}{\sqrt{1-q}}\Big)\Big|^{2},
    \end{equation*}
    the proof reduces to evaluating integral
    \begin{equation*}
    \begin{split}
        &\frac{1}{1-q}\int_{\mathbb{C}}z^{p_{1}}\mathcal{B}_{m-1}^{q}|j\rangle\Big(\frac{z}{\sqrt{1-q}},\frac{\overline{z}}{\sqrt{1-q}}\Big)\overline{z^{p_{2}}\mathcal{B}_{m-1}^{q}|j\rangle\Big(\frac{z}{\sqrt{1-q}},\frac{\overline{z}}{\sqrt{1-q}}\Big)}\,d\widehat{\mathsf{m}}_{q}(z)\smallskip\\
        &=\delta_{p_{1},p_{2}}(1-q)^{p_{1}}\int_{\mathbb{C}}\Big|\overline{z}^{p_{1}}\mathcal{B}_{m-1}^{q}|j\rangle(z,\overline{z})\Big|^{2}\,d\mathsf{m}_{q}(z).
    \end{split}
    \end{equation*}
    Here, the radial symmetry of the underlying measure and \eqref{eqn:polynomial scaling} are applied.

    Repeatedly applying the three-term recurrence \eqref{eqn:threeterm barz}, we obtain
    \begin{equation*}
        q^{pj}\overline{z}^{p}\mathcal{B}_{m-1}^{q}|j\rangle(z,\overline{z})
        =\frac{1}{\sqrt{[m-1]_{q}!}}q^{\frac{p(2j-2m-p+1)}{4}}\sum_{l=0}^{p\wedge j}(-1)^{p-l}q^{\frac{l^{2}}{2}}\qbinom{p}{l}\sqrt{\frac{[j]_{q}![m+p-l-1]_{q}!}{[j-l]_{q}!}}\mathcal{B}_{m+p-l-1}^{q}|j-l\rangle(z,\overline{z}).
    \end{equation*}
    The orthonormality of $\{\mathcal{B}^{q}_{r}|n\rangle\}_{n,r\geq0}$ in $L^{2}(\mathbb{C}, \mathsf{m}_{q})$ gives
    \begin{equation*}
        \int_{\mathbb{C}}\Big|\overline{z}^{p}\mathcal{B}_{m-1}^{q}|j\rangle(z,\overline{z})\Big|^{2}\,d\mathsf{m}_{q}(z)
        =q^{-\frac{p(2m+p-1)}{2}-pj}\sum_{l=0}^{p\wedge j}q^{l^{2}}\qbinom{p}{l}^{2}\qbinom{j}{l}\frac{[l]_{q}![m+p-l-1]_{q}!}{[m-1]_{q}!}.
    \end{equation*}
    Taking the sum over $0\leq j\leq N-1$, we obtain the desired result.
\end{proof}

Next, we find the leading asymptotic behavior of the spectral moments in the large-$N$ limit.

\begin{lemma}\label{lem:rescaled moment limit}
    Let $q=e^{-\lambda/N}$ for a fixed $\lambda>0$ and fix $m\geq1$. As $N\to\infty$, we have
    \begin{equation}\label{eqn:large N expansion of spectral moments}
        \frac{1}{N}\mathfrak{m}_{m,N}^{q}(p,p)
        =\frac{1}{\lambda}\int_{0}^{1-e^{-\lambda}}\frac{u^{p}}{(1-u)^{p+1}}\,du+O(N^{-1}).
    \end{equation}
\end{lemma}
\begin{proof}
    Using \eqref{eqn:spectral moment formula}, we write
    \begin{equation*}
        \frac{1}{N}\mathfrak{m}_{m,N}^{q}(p,p)
        =\sum_{l=0}^{p}\mathsf{M}_{1}(m,p,l)\mathsf{M}_{2}(m,p,l)
    \end{equation*}
    where
    \begin{equation*}
    \begin{split}
        \mathsf{M}_{1}(m,p,l)
        &=\Big(\frac{N}{\lambda}\Big)^{p}(1-q)^{p}q^{-\frac{p(2m+p-1)}{2}+l^{2}}\qbinom{p}{l}^{2}\frac{[l]_{q}![m+p-l-1]_{q}!}{[m-1]_{q}!},\smallskip\\
        \mathsf{M}_{2}(m,p,l)
        &=\frac{1}{N}\Big(\frac{\lambda}{N}\Big)^{p}\sum_{j=0}^{N-1}q^{-pj}\qbinom{j}{l}.
    \end{split}
    \end{equation*}
    Since $1-q= \frac{\lambda}{N} + O(N^{-2})$, it is straightforward that
    \begin{equation}\label{eqn:mathsfM1 asymp}
        \mathsf{M}_{1}(m,p,l)
        =\binom{p}{l}^{2}\frac{l!(m+p-l-1)!}{(m-1)!}\big(1+O(N^{-1})\big).
    \end{equation}
    To analyze $\mathsf{M}_{2}(m,p,l)$, we write
    \begin{equation*}
        \mathsf{M}_{2}(m,p,l)
        =\frac{1}{N}\Big(\frac{\lambda}{N}\Big)^{p}\sum_{j=0}^{N-1}f_{l}(j),\qquad
        f_{l}(j):=q^{-pj}\qbinom{j}{l}.
    \end{equation*}
    Letting $j=Nt$, we obtain
    \begin{equation*}
        \Big(\frac{\lambda}{N}\Big)^{l}f_{l}(Nt)
        =\Big(\frac{\lambda}{N}\Big)^{l}q^{-pNt}\prod_{k=1}^{l}\frac{1-q^{Nt-l+k}}{1-q^{k}}
        =\frac{e^{\lambda pt}(1-e^{-\lambda t})^{l}}{l!}+O(N^{-1})
    \end{equation*}
    uniformly in $t\in[0,1]$. Thus, only the case $l=p$ contributes at leading order, which gives the asymptotic
    \begin{equation}\label{eqn:mathsfM2 asymp}
        \mathsf{M}_{2}(m,p,p)
        =\frac{1}{Np!}\sum_{j=0}^{N-1}e^{\lambda pj/N}(1-e^{-\lambda j/N})^{p}+O(N^{-1})
        =\frac{1}{p!}\int_{0}^{1}e^{\lambda pt}(1-e^{-\lambda t})^{p}\,dt+O(N^{-1}).
    \end{equation}
    Combining \eqref{eqn:mathsfM1 asymp} and \eqref{eqn:mathsfM2 asymp} after the change of variables $u=1-e^{-\lambda t}$, we complete the proof.
\end{proof}

\subsection{Convergence of empirical measures}

We first identify a probability measure on the complex plane whose mixed moments match the limits obtained in \eqref{eqn:large N expansion of spectral moments}, hence the limiting measure of the averaged empirical measure.

For the pure ensemble, define the radial pushforward of the averaged normalized empirical measure
\begin{equation}\label{eqn:nuN}
    d\nu_{m,N}^{q}(t)
    :=\frac{1}{N}\widehat{\mathcal{K}}_{m,N}^{q}(\sqrt{t},\sqrt{t})\widehat{e}_{q}(-t)\,d_{q}t,\qquad
    t\in\mathbb{R}_{\geq0},
\end{equation}
so that
\begin{equation*}
    \frac{1}{N}\widehat{\mathcal{K}}_{m,N}^{q}(z,z)\,d\widehat{\mathsf{m}}_{q}(z)
    =\frac{d\theta}{2\pi}\,d\nu_{m,N}^{q}(t),\qquad
    z=t^{\frac{1}{2}}e^{i\theta}.
\end{equation*}
Accordingly, we define the corresponding pushforward of the $q$-deformed circular law \eqref{eqn:limiting law}
\begin{equation}\label{eqn:nu}
    d\nu_{\lambda}(t)
    =\frac{1}{\lambda(1+t)}\mathbbm{1}_{[0,e^{\lambda}-1]}(t)\,dt.
\end{equation}

\begin{lemma}\label{lem:nu convergence}
    Let $q=e^{-\lambda/N}$ for a fixed $\lambda>0$ and fix $m\geq1$. Then we have a weak convergence
    \begin{equation}
        \nu_{m,N}^{q}
        \to\nu_{\lambda}
    \end{equation}
    as $N$ tends to infinity.
\end{lemma}

\begin{proof}
    From the definition of $\nu_{m,N}^{q}$, Theorem \ref{thm:mixed moment} and Lemma \ref{lem:rescaled moment limit}, it follows that
    \begin{equation*}
        \int_{0}^{\infty}t^{p}\,d\nu_{m,N}^{q}(t)
        =\frac{1}{N}\mathfrak{m}_{m,N}^{q}(p,p)
        =\frac{1}{\lambda}\int_{0}^{1-e^{-\lambda}}\frac{u^{p}}{(1-u)^{p+1}}\,du+O(N^{-1}).
    \end{equation*}
    for each $p\in\mathbb{Z}_{\geq0}$. Applying the change of variables $t=u/(1-u)$, we obtain
    \begin{equation*}
        \int_{0}^{\infty}t^{p}\,d\nu_{m,N}^{q}(t)
        =\int_{0}^{e^{\lambda}-1}t^{p}\frac{1}{\lambda(1+t)}\,dt+O(N^{-1}),
    \end{equation*}
    where the leading term of the right-hand side is the $p$-th moment of $\nu_{\lambda}$. Since $\nu_{\lambda}$ is compactly supported, its moment problem is determinate, which gives the desired convergence.
\end{proof}

Since Lemma \ref{lem:nu convergence} implies the weak convergence of the averaged normalized empirical measure of the pure ensembles to the $q$-deformed circular law, it remains to verify the convergence of the random empirical measure in probability.

\begin{proof}[Proof of Theorem \ref{thm:double scale limiting law}]
    Fix a bounded real-valued continuous function $f$ on the complex plane. The variance formula for the projection DPP gives
    \begin{equation*}
        \Var\Big[\int_{\mathbb{C}}f\,d\widehat{\mu}_{m,N}^{q}\Big]
        =\frac{1}{2N^{2}}\iint_{\mathbb{C}^{2}}|f(z)-f(w)|^{2}|\widehat{\mathcal{K}}_{m,N}^{q}(z,w)|^{2}\,d\widehat{\mathsf{m}}_{q}^{\otimes2}(z,w).
    \end{equation*}
    Using $|f(z)-f(w)|\leq2\|f\|_{\infty}$ and the reproducing property of the projection kernel, we obtain
    \begin{equation*}
        \Var\Big[\int_{\mathbb{C}}f\,d\widehat{\mu}_{m,N}^{q}\Big]
        \leq\frac{2\|f\|_{\infty}^{2}}{N^{2}}\iint_{\mathbb{C}^{2}}|\widehat{\mathcal{K}}_{m,N}^{q}(z,w)|^{2}\,d\widehat{\mathsf{m}}_{q}^{\otimes2}(z,w)
        =\frac{2\|f\|_{\infty}^{2}}{N}.
    \end{equation*}
    Since $\widehat{\rho}_{\lambda}\,dA$ is the weak limit of the average of $\widehat{\mu}_{m,N}^{q}$, applying Chebyshev inequality, we have the convergence
    \begin{equation*}
        \int_{\mathbb{C}}f\,d\widehat{\mu}_{m,N}^{q}
        \to\int_{\mathbb{C}}f\rho_{\lambda}\,dA
    \end{equation*}
    in probability as $N\to\infty$. The proof for the full ensembles follows by combining the result for the first $m$ pure ensembles.
\end{proof}

%%%%%%%%%%%%%%%%%%%%%%%%%%%%%%%%%%%%%%%%%%%%%%%%%%%%%%%%%%%%%%%%%%%%%%%%%%%%%%%%%%%%%%%%%%%%%%%%%%%%%%%%%%%%%%%%%%%%%%%%%%%%%%%%%%%%%%%%%%%%%%%%%%%%%%%%%%%%%%%%%%%%
\section{Fluctuations of polynomial statistics in the {lowest} Landau level}
\label{sec:CLT lowest}

The remaining sections are devoted to the proof of the central limit theorem for linear statistics in the double-scaling regime \eqref{eqn:def of q scaling}. In this section, we establish Theorem \ref{thm:CLT} for the analytic case $m=1$ and polynomial test functions following the argument of Rider and Vir\'ag \cite{RV07}. This polynomial result will serve as the analytic input for the reduction from higher Landau levels developed in Section \ref{sec:Raising}. The passage to compactly supported smooth test functions is postponed to Section \ref{sec:CLT} where it is carried out simultaneously for the pure and full ensembles.

In this section, we restrict our attention to the rescaled $q$-deformed Ginibre ensemble, i.e., the determinantal point process with correlation kernel $\widehat{\mathcal{K}}^q_N:=\widehat{\mathcal{K}}^q_{1, N}$. We define 
\begin{equation}\label{eqn:varphi}
    \varphi_{j}^{q}(z)
    :=\frac{1}{\sqrt{1-q}}\mathcal{B}^{q}_{0}|j\rangle\Big(\frac{z}{\sqrt{1-q}},\frac{\overline{z}}{\sqrt{1-q}}\Big)
    =q^{\frac{j(j+1)}{4}}\frac{z^{j}}{\sqrt{(1-q)(q;q)_{j}}}.
\end{equation}
Thus, we have 
\begin{equation}\label{eqn:analytic kernel}
    \widehat{\mathcal{K}}^{q}_{N}(z,w)
    =\sum_{j=0}^{N-1}\varphi_{j}^{q}(z)\overline{\varphi_{j}^{q}(w)}.
\end{equation}

The goal of this section is the following partial case of the main theorem.

\begin{theorem}\label{thm:analytic CLT}
    Let $g$ be a real-valued polynomial. Then
    \begin{equation*}
        \fluct_{1,N}(g)
        \longrightarrow\mathcal{N}\Big(0,\,\|g\|_{H^{1}(D_{\lambda}^{\circ})}^{2}+\frac{1}{2}\|g\|_{H^{1/2}(\partial D_{\lambda})}^{2}\Big)
    \end{equation*}
    in distribution as $N\to\infty$.
\end{theorem}

Let us recall that for an arbitrary real-valued random variable $X$ its cumulant $C_k(X)$ of order $k\in\mathbb{Z}_{>0}$ is defined by the coefficients in the formal series
\begin{equation*}
    \log \mathbb{E}[e^{itX}]
    =\sum_{k=1}^\infty \frac{(it)^{k}}{k!}C_k(X).
\end{equation*}
For a real-valued test function $g$ on $\mathbb{C}$, the proof of Theorem \ref{thm:analytic CLT} boils down to verifying that the cumulants of the random variable $\fluct_{1,N}(g)$ satisfy
\begin{equation}\label{eqn:analytic cumulant limit}
    \lim_{N\to\infty}C_{k}(\fluct_{1,N}(g))=
    \begin{cases}
        \displaystyle\|g\|_{H^{1}(D_{\lambda}^{\circ})}^{2}+\frac{1}{2}\|g\|_{H^{1/2}(\partial D_{\lambda})}^{2}, & \textup{if }k=2,\smallskip \\
        0, & \textup{if }k\geq3.
    \end{cases}
\end{equation}
To proceed, we make use of a representation of the cumulants $C_{k}(\fluct_{1,N}(g))$ introduced in \cite{RV07}. For a surjective map $\sigma:[k]\twoheadrightarrow[j]$ and a $k$-tuple $f=(f_{1},\cdots,f_{k})$ of test functions, define $j$-tuple $\sigma f=((\sigma f)_1, \dots, (\sigma f)_j)$ by
\begin{equation*}
    (\sigma f)_{l}(z)
    =\prod_{i:\sigma(i)=l}f_{i}(z).
\end{equation*}

We further define 
\begin{equation*}
    \mathscr{C}_k(f_{1},\cdots,f_{k})
    :=\sum_{j=1}^{k}\frac{(-1)^{j-1}}{j} \sum_{\sigma \colon [k]\twoheadrightarrow [j]}\Phi_{j} (\sigma f),
\end{equation*}
where 
\begin{equation*}
    \Phi_{j}(f_{1},\cdots,f_{j})
    :=\int_{\mathbb{C}^{j}}f_{1}(z_{1})\cdots f_{j}(z_{j})\widehat{\mathcal{K}}^{q}_{N}(z_{1},z_{2})\cdots \widehat{\mathcal{K}}^{q}_{N}(z_{j},z_{1})d\widehat{\mathsf{m}}_{q}^{\otimes j}(z_{1},\cdots, z_{j}),
\end{equation*}
This yields a formula for the $k$-th cumulant of $\fluct_{1,N}(g)$:
\begin{equation*}
    C_{k}(\fluct_{1,N}(g))
    =\mathscr{C}_{k}(g,\cdots,g),\qquad
    k\geq2;
\end{equation*}
see \cite[Eq. (4.4)]{RV07}. Thus, it is essential to investigate the asymptotic behavior of $\Phi_{j}(\sigma f)$ to establish \eqref{eqn:analytic cumulant limit}. 

Since $\mathscr{C}_{k}$ is multilinear, we begin with the case of monomials. For multi-indices $\alpha=(\alpha_{1},\cdots,\alpha_{j})$ and $\beta=(\beta_{1},\cdots,\beta_{j})$ of nonnegative integers, we denote
\begin{equation*}
    |\alpha|
    =\sum_{l=1}^{j}\alpha_{l},\qquad
    |\beta|=\sum_{l=1}^{j}\beta_{l}
\end{equation*}
and
\begin{equation*}
    z^{\alpha}\overline{z}^{\beta}
    :=(z^{\alpha_{1}}\overline{z}^{\beta_{1}},\cdots,z^{\alpha_{j}}\overline{z}^{\beta_{j}}).
\end{equation*}
For each $1\leq l\leq j$, we further define
\begin{equation*}
    \eta_{l}
    :=(\beta_{1}-\alpha_{1})+\cdots+(\beta_{l}-\alpha_{l}), \qquad
    \eta_{\min}:=\min_{1\leq l\leq j}\eta_{l}, \qquad
    \eta_{\max}:=\max_{1\leq l\leq j}\eta_{l}.
\end{equation*}

We now obtain an explicit formula for $\Phi_{j}$ with monomial arguments.
\begin{lemma}
    For every $j\geq 1$, we have
    \begin{equation}\label{eqn:Phi(zzbar)}
        \Phi_{j}(z^{\alpha}\overline{z}^{\beta})=
        \begin{cases}
            \displaystyle\sum_{n=-\eta_{\min}}^{N-1-\eta_{\max}}\prod_{l=1}^{j}q^{-\alpha_{l}(n+\eta_{l})-\frac{\alpha_{l}(\alpha_{l}+1)}{2}}(q^{n+\eta_{l}+1};q)_{\alpha_{l}}, & \textup{if }|\alpha|=|\beta|,\smallskip \\
            0, & \textup{otherwise.}
        \end{cases}
    \end{equation}
\end{lemma}
\begin{proof}
    By definition of the functional $\Phi_{j}$ and the analytic kernel \eqref{eqn:analytic kernel}, we have 
    \begin{align*}
        \Phi_{j}(z^{\alpha}\overline{z}^{\beta})= \sum_{n_1, \dots, n_j=0}^{N-1}\prod_{l=1}^j \int_\mathbb{C} z^{\alpha_l} \varphi_{n_l}^q(z) \overline{z^{\beta_l}\varphi^q_{n_{l-1}}(z)}d\widehat{\mathsf{m}}_q(z),
    \end{align*}
    where $n_{0}:=n_{j}$. Moreover, we have $z^{\alpha}\varphi_{n}^{q}(z)=\frac{\kappa_{n}}{\kappa_{n+\alpha}}\varphi_{n+\alpha}^{q}(z)$ for every $n$, $\alpha\geq0$, where $\kappa_{n}:=q^{\frac{n(n+1)}{4}}\frac{1}{\sqrt{(1-q)(q;q)_{n}}}$. Since $\{\varphi_{n}^{q}\}_{n\geq0}$ forms an orthonormal set in $L^{2}(\mathbb{C},\widehat{\mathsf{m}}_{q})$, each summand on the right-hand side is nonzero if and only if $\alpha_{l}+n_{l}=\beta_{l}+n_{l-1}$ for every $1\leq l\leq j$. This condition is equivalent to
    \begin{equation}\label{eqn:eta condition}
        n_{l}=n_{j}+\eta_{l}
    \end{equation}
    for each $1\leq l\leq j$, and this identity forces $\eta_{j}=0$, that is, $|\alpha|=|\beta|$.

    On the other hand, for multi-indices $\alpha$, $\beta$ with $|\alpha|=|\beta|$, a $j$-tuple $(n_{1},\cdots,n_{j})$ satisfying \eqref{eqn:eta condition} is determined only by $n_{j}$. Moreover, in order that $0\leq n_{l}\leq N-1$ holds for all $1\leq l\leq j$, we have $-\eta_{\min}\leq n_{j}\leq N-1-\eta_{\max}$. For a fixed $n_{j}=n\in\{-\eta_{\min},\cdots,N-1-\eta_{\max}\}$, we have
    \begin{align*}
        \prod_{l=1}^{j}\int_{\mathbb{C}} z^{\alpha_{l}}\varphi_{n_{l}}^{q}(z) \overline{z^{\beta_{l}}\varphi_{n_{l-1}}^{q}(z)}d\widehat{\mathsf{m}}_q(z)
        =\prod_{l=1}^{j}\frac{\kappa_{n_{l}}}{\kappa_{n_{l}+\alpha_{l}}}\frac{\kappa_{n_{l-1}}}{\kappa_{n_{l-1}+\beta_{l}}}
        =\prod_{l=1}^{j}\frac{\kappa_{n+\eta_{l}}^{2}}{\kappa_{n+\eta_{l}+\alpha_{l}}^{2}}.
    \end{align*}
    In the last equality, we used \eqref{eqn:eta condition} and $\alpha_{l}+n_{l}=\beta_{l}+n_{l-1}$. Substituting the definition of $\kappa_{n}$ and taking a sum over $-\eta_{\min}\leq n\leq N-1-\eta_{\max}$, we complete the proof.
\end{proof}

Next, we derive a large-$N$ expansion of $\Phi_j(z^\alpha \overline{z}^\beta)$ in the double-scaling regime \eqref{eqn:def of q scaling}.

\begin{lemma}\label{lem:large_N_expansion}
    Suppose that multi-indices $\alpha=(\alpha_{1},\cdots,\alpha_{j})$ and $\beta=(\beta_{1},\cdots,\beta_{j})$ satisfy $|\alpha|=|\beta|=:s\geq1$. In the double-scaling regime $q=e^{-\lambda/N}$ for a fixed $\lambda>0$, 
    we have an expansion
    \begin{equation}\label{eqn:Phij expansion}
        \Phi_{j}(z^{\alpha}\overline{z}^{\beta})
        =N\int_{0}^{1}(e^{\lambda t}-1)^{s}\,dt+(e^{\lambda}-1)^{s}\Big(\frac{A_{j}}{s}-\eta_{\max}-\frac{1}{2}\Big)+o(1),
    \end{equation}
    as $N\to\infty$, where
    \begin{equation*}
        A_{j}:=\sum_{l=1}^j\Big(\alpha_l\eta_l+\frac{\alpha_l(\alpha_l+1)}{2}\Big).
    \end{equation*}
    
\end{lemma}

\begin{proof}
    Assuming $s\geq1$,  each $\eta_{l}$ is of order $O(1)$. In order to apply the Euler--Maclaurin formula for \eqref{eqn:Phi(zzbar)}, we write $\Phi_j(z^\alpha\overline{z}^\beta)= \sum_{n=-\eta_{\min}}^{N-1-\eta_{\max}} F_j(n)$, where
    \begin{equation*}
        F_{j}(n)
        =\prod_{l=1}^{j}q^{-\alpha_{l}(n+\eta_{l})-\frac{\alpha_{l}(\alpha_{l}+1)}{2}}(q^{n+\eta_{l}+1};q)_{\alpha_{l}}
        =\prod_{l=1}^{j}\prod_{r=1}^{\alpha_{l}}(q^{-(n+\eta_{l}+r)}-1).
    \end{equation*}
    Letting $n=Nt$ for $0\leq t\leq 1$, we obtain a uniform expansion
    \begin{equation*}
    \begin{split}
        F_{j}(Nt)
        =\prod_{l=1}^{j}\prod_{r=1}^{\alpha_{l}}\Big(e^{\lambda t}-1 +\frac{\lambda}{N}e^{\lambda t}(\eta_{l}+r)+O(N^{-2})\Big)
        =(e^{\lambda t}-1)^{s}+A_{j}e^{\lambda t}(e^{\lambda t}-1)^{s-1}\frac{\lambda}{N}+O(N^{-2}).
    \end{split}
    \end{equation*}
    Now, from the Euler--Maclaurin formula (see, e.g., \cite[Eq. 2.10.1]{NIST}), it follows that
    \begin{equation}\label{eqn:F EM formula}
        \sum_{n=-\eta_{\min}}^{N-1-\eta_{\max}}F_{j}(n)
        =N\int_{-\eta_{\min}/N}^{1-(1+\eta_{\max})/N}F_{j}(Nt)\,dt+\frac{F_{j}(-\eta_{\min})+F_{j}(N-1-\eta_{\max})}{2}+o(1).
    \end{equation}
    First, since
    \begin{equation*}
        N\int_{0}^{-\eta_{\min}/N}(e^{\lambda t}-1)^{s}\,dt=O(N^{-s}),\qquad
        N\int_{1-(1+\eta_{\max})/N}^{1}(e^{\lambda t}-1)^{s}\,dt=(1+\eta_{\max})(e^{\lambda}-1)^{s}+o(1)
    \end{equation*}
    we have
    \begin{equation}\label{eqn:F integral asymptotic}
        N\int_{-\eta_{\min}/N}^{1-(1+\eta_{\max})/N}(e^{\lambda t}-1)^{s}\,dt
        =N\int_{0}^{1}(e^{\lambda t}-1)^{s}\,dt-(1+\eta_{\max})(e^{\lambda}-1)^{s}+o(1).
    \end{equation}
    Moreover, a straightforward calculation gives the asymptotics
    \begin{equation}\label{eqn:F boundary asymptotic}
        F_{j}(-\eta_{\min})
        =O(N^{-s}),\qquad
        F_{j}(N-1-\eta_{\max})
        =(e^{\lambda}-1)^{s}+O(N^{-1})
    \end{equation}
    and
    \begin{equation}\label{eqn:F subleading integral}
        \int_{-\eta_{\min}/N}^{1-(1+\eta_{\max})/N}e^{\lambda t}(e^{\lambda t}-1)^{s-1}\,dt
        =\frac{1}{\lambda s}(e^{\lambda}-1)^{s}+o(1).
    \end{equation}
    Substituting the asymptotics \eqref{eqn:F integral asymptotic}, \eqref{eqn:F boundary asymptotic}, \eqref{eqn:F subleading integral} into the Euler--Maclaurin formula \eqref{eqn:F EM formula} completes the proof.
\end{proof}

\begin{proof}[Proof of Theorem \ref{thm:analytic CLT}]
    We begin with asymptotics of $\mathscr{C}_{k}(f_{1},\cdots,f_{k})$ when each $f_{j}$ is a monomial. Let $k\geq 2$ and fix multi-indices $\alpha=(\alpha_{1},\cdots,\alpha_{k})$ and $\beta=(\beta_{1},\cdots,\beta_{k})$. For $1\leq l\leq k$ and a surjection $\sigma:[k]\twoheadrightarrow[j]$, we define the length-$j$ multi-indices $\sigma\alpha$ and $\sigma\beta$ by
    \begin{equation*}
        (\sigma\alpha)_{l}=\sum_{i:\sigma(i)=l}\alpha_{i},\qquad
        (\sigma\beta)_{l}=\sum_{i:\sigma(i)=l}\beta_{i}
    \end{equation*}
    for each $1\leq l \leq j$. We further define 
    \begin{equation*}
        \eta^{\sigma}_{l}
        =\big((\sigma\beta)_{1}-(\sigma\alpha)_{1}\big)+\cdots+\big((\sigma\beta)_{l}-(\sigma\alpha)_{l}\big),\qquad
        \eta^{\sigma}_{\max}=\max_{1\leq l\leq j}\eta^{\sigma}_{l}.
    \end{equation*}
    We remark that for every $1\leq j\leq k$ and $\sigma \colon [k]\twoheadrightarrow [j]$, the condition for $\Phi_{j}(\sigma(z^{\alpha}\overline{z}^{\beta}))$ to be nonzero is $|\sigma \alpha|=|\sigma\beta|$, which is equivalent to $|\alpha|=|\beta|$. Assume $s:=|\alpha|=|\beta|$. If $s=0$, all monomials are constant and the joint cumulant vanishes. If $s\geq1$, Lemma \ref{lem:large_N_expansion} provides the asymptotic 
    \begin{equation}\label{eqn:scrC asymptotic 1}
        \mathscr{C}_{k}(z^{\alpha}\overline{z}^{\beta})
        =\sum_{j=1}^{k}\frac{(-1)^{j-1}}{j}\sum_{\sigma\colon[k]\twoheadrightarrow[j]}
        \Bigg(N\int_0^1(e^{\lambda t}-1)^{s}\,dt+(e^{\lambda}-1)^{s}\Big(\frac{A_{j}^{\sigma}}{s}-\eta^{\sigma}_{\max}-\frac{1}{2}\Big)\Bigg)+o(1),
    \end{equation}
    where
    \begin{equation*}
        A_{j}^{\sigma}
        =\sum_{l=1}^{j}\Big((\sigma\alpha)_{l}\eta_{l}^{\sigma}+\frac{(\sigma\alpha)_{l}((\sigma\alpha)_{l}+1)}{2}\Big).
    \end{equation*}
    Using the identity
    \begin{equation*}
        \sum_{j=1}^{k}\frac{(-1)^{j-1}}{j}\sum_{\sigma\colon[k]\twoheadrightarrow [j]}1=0,
    \end{equation*}
    we simplify \eqref{eqn:scrC asymptotic 1} to
    \begin{equation*}
        \mathscr{C}_{k}(z^{\alpha}\overline{z}^{\beta})
        =(e^{\lambda}-1)^{s}\sum_{j=1}^{k}\frac{(-1)^{j-1}}{j}\sum_{\sigma \colon [k]\twoheadrightarrow [j]} \Big(\frac{A_{j}^{\sigma}}{s}-\eta^{\sigma}_{\max}\Big)+o(1),
    \end{equation*}
    where the expression on the right-hand side is precisely the rotary-flow functional treated in \cite[Section 5]{RV07}. Using \cite[Lemmas 5.1--5.5]{RV07}, we obtain
    \begin{equation*}
        \sum_{j=1}^{k}\frac{(-1)^{j-1}}{j}\sum_{\sigma \colon [k]\twoheadrightarrow [j]} \Big(\frac{A_{j}^{\sigma}}{s}-\eta^\sigma_{\max}\Big)=
        \begin{cases}
            \displaystyle\frac{\alpha_{1}\beta_{2}+\alpha_{2}\beta_{1}}{2s}+\frac{|\beta_{1}-\alpha_{1}|}{2}, & \textup{if }k=2,\smallskip \\
            0, & \textup{if }k\geq 3.
        \end{cases}
    \end{equation*}
    This implies that the cumulants of order $\geq3$ vanish in the large-$N$ limit. On the other hand, for $k=2$, fix $\alpha=(\alpha_{1},\alpha_{2})$ and $\beta=(\beta_{1},\beta_{2})$ with $|\alpha|=|\beta|$, we have
    \begin{equation*}
    \begin{split}
        \mathscr{C}_{2}(z^{\alpha}\overline{z}^{\beta})
        &=(e^{\lambda}-1)^{\frac{\alpha_{1}+\beta_{1}+\alpha_{2}+\beta_{2}}{2}}
        \Big(\frac{\alpha_{1}\beta_{2}+\alpha_{2}\beta_{1}}{2(\alpha_{1}+\alpha_{2})}+\frac{|\beta_{1}-\alpha_{1}|}{2}\Big)+o(1)\smallskip\\
        &=
        \langle z^{\alpha_{1}}\overline{z}^{\beta_{1}}, z^{\alpha_{2}}\overline{z}^{\beta_{2}} \rangle_{H^{1}(D_{\lambda}^{\circ})}+\frac{1}{2}\langle z^{\alpha_{1}}\overline{z}^{\beta_{1}}, z^{\alpha_{2}}\overline{z}^{\beta_{2}}\rangle_{H^{1/2}(\partial  D_{\lambda})}+o(1)\smallskip\\
        &=\langle f_{1},f_{2}\rangle_{H^{1}(D_{\lambda}^{\circ})}+\frac{1}{2}\langle f_{1}, f_{2}\rangle_{H^{1/2}(\partial D_{\lambda})}+o(1),
    \end{split}
    \end{equation*}
    where $f_{j}(z)=z^{\alpha_{j}}\overline{z}^{\beta_{j}}$ for $j=1$, $2$. Now, the desired limit \eqref{eqn:analytic cumulant limit} follows from the multilinearity of $\mathscr{C}_{k}$. Since the Gaussian distribution is moment-determinate, this completes the proof.
\end{proof}

\section{Reduction to the lowest Landau level}
\label{sec:Raising}

Having established the polynomial central limit theorem for the analytic ensemble, we now extend the analysis to the pure and full polyanalytic ensembles. Rather than deriving separate asymptotics for each higher-level kernel, we follow the strategy of Haimi and Wennman \cite{HW19} and exploit the $q$-raising operator representation of these kernels.

After the rescaling $z\mapsto z/\sqrt{1-q}$, the raising maps factor into normalized powers of a first-order $q$-difference operator $\mathbf{T}$, together with explicit $q$-shifts. For polynomial test functions, exact Jackson integration-by-parts identities move these operators from the kernel factors onto the cumulant polynomial. This reduces the higher-level cumulants to cyclic integrals involving only the analytic kernel, with the polyanalytic dependence encoded in explicit $q$-difference operators.

Throughout this section, all operators and their adjoints are understood on $\mathbb{C}[z, \overline{z}]\subset L^{2}(\mathbb{C},\widehat{\mathsf{m}}_{q})$. We continue to use the notation $\widehat{\mathcal{K}}^{q}_{N}$ for the analytic kernel.

\subsection{Rescaled $q$-raising operator}

Let $\mathcal{U}_{q}$ be the dilation
\begin{equation*}
    [\mathcal{U}_{q}f](z,\overline{z})
    =\frac{1}{\sqrt{1-q}}f\Big(\frac{z}{\sqrt{1-q}},\frac{\overline{z}}{\sqrt{1-q}}\Big)
\end{equation*}
so that the rescaled $q$-deformed polyanalytic Bargmann basis
\begin{equation*}
    \widehat{\mathcal{B}}_{r}^{q}|n\rangle
    =\mathcal{U}_{q}\mathcal{B}_{r}^{q}|n\rangle
\end{equation*}
satisfies
\begin{equation*}
    \widehat{\mathcal{K}}_{m,N}^{q}(z,w)
    =\sum_{n=0}^{N-1}\widehat{\mathcal{B}}_{m-1}^{q}|n\rangle(z,\overline{z})\overline{\widehat{\mathcal{B}}_{m-1}^{q}|n\rangle(w,\overline{w})},\qquad
    \widehat{\mathcal{K}}_{m,N}^{q,\full}(z,w)
    =\sum_{r=1}^{m}\widehat{\mathcal{K}}_{r,N}^{q}(z,w).
\end{equation*}

The dilation $\mathcal{U}_q$ transfers the raising-operator representation developed in Section \ref{sec:higher_q_raising_maps} to the rescaled setting. We begin with the conjugate of the unnormalized raising operator $\widetilde{\mathcal{R}}_q$. Namely, we introduce 
\begin{equation*}
    \mathbf{T}
    :=-\frac{1}{\sqrt{1-q}}\mathcal{U}_q\widetilde{\mathcal{R}}_{q}\mathcal{U}_{q}^{-1}.
\end{equation*}
It admits a gauge-type representation 
\begin{equation*}
    \mathbf{T}f(z,\overline{z})
    =-\widehat{E}_{q}(|z|^{2})\partial_{q,z}[\widehat{e}_q(-|z|^{2})f(z,\overline{z})]
    =-\Big((1+|z|^{2})\partial_{q,z}-\frac{\overline{z}}{1-q}\Big)f(z,\overline{z}),
\end{equation*}
where
\begin{equation*}
    \widehat{E}_{q}(x)
    :=E_{q}\Big(\frac{x}{1-q}\Big)
    =(-x;q)_{\infty}.
\end{equation*}
For every nonnegative integer $r\geq 0$ we further define 
\begin{equation}\label{eqn:gamma}
    {\mathbf{T}}_{r}
    :=(-1)^{r}\gamma_{r}{\mathbf{T}}^r,\qquad
    \gamma_{r}
    :=q^{\frac{r(r+1)}{4}}\frac{(1-q)^{r}}{\sqrt{(q;q)_{r}}}.
\end{equation}
Its gauge-type representation is given by 
\begin{align*}
    [{\mathbf{T}}_{r}f](z,\overline{z})
    &=\frac{q^{\frac{r(r+1)}{4}}(1-q)^{r}}{\sqrt{(q;q)_{r}}}\widehat{E}_{q}(|z|^{2})\partial_{q,z}^{r}[f(z,\overline{z})\widehat{e}_{q}(-|z|^{2})]\\
    &=q^{\frac{r(r+1)}{4}}\frac{(1-q)^{r}}{\sqrt{(q;q)_{r}}}\Big((1+|z|^{2})\partial_{q,z}-\frac{\overline{z}}{1-q}\Big)^{r}f(z, \overline{z}).
\end{align*}

Applying $\mathcal{U}_q$ to \eqref{eqn:BRSB} and using the definitions $\mathbf{T}$ and $\mathbf{T}_r$, we immediately obtain
\begin{equation}\label{eqn:BhatTSB}
    \widehat{\mathcal{B}}_{r}^{q}|n\rangle
    =\mathbf{T}_{r}\mathsf{S}_{q,z}^{-\frac{r}{2}}\widehat{\mathcal{B}}_{0}^{q}|n\rangle.
\end{equation}

\begin{proposition}
    The rescaled pure and full polyanalytic kernels admit the representation
    \begin{equation}\label{eqn:kernel T}
    \begin{split}
        \widehat{\mathcal{K}}_{m,N}^{q}(z,w)
        &=\big[\mathbf{T}_{m-1}\mathsf{S}_{q}^{-m+1}\big]_{z}[\mathbf{T}_{m-1}]_{\overline{w}}\widehat{\mathcal{K}}_{N}^{q}(z,w),\smallskip\\
        \widehat{\mathcal{K}}_{m,N}^{q,\full}(z,w)
        &=\sum_{r=1}^{m}\big[\mathbf{T}_{r-1}\mathsf{S}_{q}^{-r+1}\big]_{z}[\mathbf{T}_{r-1}]_{\overline{w}}\widehat{\mathcal{K}}_{N}^{q}(z,w).
    \end{split}
    \end{equation}
\end{proposition}

\begin{proof}
    Substituting \eqref{eqn:BhatTSB} into the basis expansion of $\widehat{\mathcal{K}}^q_{m, N}$, we obtain
    \begin{equation*}
        \widehat{\mathcal{K}}_{m,N}^{q}(z,w)
        =\big[\mathbf{T}_{m-1}\mathsf{S}_{q}^{-\frac{m-1}{2}}\big]_{z}\big[\mathbf{T}_{m-1}\mathsf{S}_{q}^{-\frac{m-1}{2}}\big]_{\overline{w}}\widehat{\mathcal{K}}_{N}^{q}(z,w).
    \end{equation*}
    Moreover, since $\widehat{\mathcal{K}}^q_N$ depends only on $z\overline{w}$, the shift on $\overline{w}$ can be transferred to $z$, yielding the first identity. Summing this identity over the pure levels gives the second.
\end{proof}

The following lemma gives the formal transpose of the power of ${\mathbf{T}}$ with respect to the measure $\widehat{\mathsf{m}}_{q}$.

\begin{lemma}
    For every polynomials $f, g\in \mathbb{C}[z, \overline{z}]$, we have
    \begin{equation}\label{eqn:Thatr adjoint}
        \int_{\mathbb{C}}f(z,\overline{z})\big[{\mathbf{T}^{r}}g\big](z,\overline{z})\,d\widehat{\mathsf{m}}_{q}(z)
        =q^{-\frac{r(r+1)}{2}}\int_{\mathbb{C}}[\mathsf{S}_{q,z}^{-r}\partial_{q,z}^{r}f](z,\overline{z})g(z,\overline{z})\,d\widehat{\mathsf{m}}_{q}(z).
    \end{equation}
\end{lemma}

\begin{proof}
    For the case $r=1$, Proposition \ref{prop:R_properties} gives  $\widetilde{\mathcal{R}}_q^* = - q^{-1}\mathsf{S}_{q, \overline{z}}^{-1}\partial_{q, \overline{z}}$. Thus, by \eqref{eqn:polynomial scaling}, we obtain
    \begin{align*}
        \int_\mathbb{C}f(z, \overline{z}) [\mathbf{T}g](z, \overline{z})d\widehat{\mathsf{m}}_q(z)
        &=\langle \mathcal{U}_q^{-1}\mathbf{T} g, \mathcal{U}_q^{-1}\overline{f}\rangle_{L^2(\mathbb{C}, \mathsf{m}_q)}\\
        &=\frac{q^{-1}}{\sqrt{1-q}}\langle \mathcal{U}_q^{-1}g, \mathsf{S}_{q, \overline{z}}^{-1}\partial_{q, \overline{z}} \mathcal{U}_q^{-1}\overline{f}\rangle_{L^2(\mathbb{C}, \mathsf{m}_q)}\\
        &=q^{-1}\langle\mathcal{U}_{q}^{-1}g,\mathcal{U}_{q}^{-1}\mathsf{S}_{q,\overline{z}}^{-1}\partial_{q,\overline{z}}\overline{f}\rangle_{L^{2}(\mathbb{C},\mathsf{m}_{q})}\\
        &=q^{-1}\int_\mathbb{C} [\mathsf{S}_{q, z}^{-1}\partial_{q, z}f](z, \overline{z}) g(z, \overline{z}) d\widehat{\mathsf{m}}_q(z).
    \end{align*}
    Here, we used the identity $\mathcal{U}_q\partial_{q, \overline{z}}\mathcal{U}_q^{-1}= \sqrt{1-q}\partial_{q, \overline{z}}$. The general case follows by iteration and the commutation relation $\partial_{q,z}\mathsf{S}_{q,z}^{-1}=q^{-1}\mathsf{S}_{q,z}^{-1}\partial_{q,z}$.
\end{proof}

Next, we provide the formal transpose operator of $\mathsf{S}_{q}^{n}$ for $n\in\mathbb{Z}$.

\begin{lemma}
    For every polynomials $f, g\in \mathbb{C}[z, \overline{z}]$ and $n\in\mathbb{Z}$, we have
    \begin{equation}\label{eqn:Sn adjoint}
        \int_{\mathbb{C}}\big[\mathsf{S}_{q,z}^{n}f\big](z,\overline{z})g(z,\overline{z})\,d\widehat{\mathsf{m}}_{q}(z)
        =q^{-n}\int_{\mathbb{C}}f(z,\overline{z})\big[\mathsf{S}_{q,z}^{-n}g\big](z,\overline{z})(-|z|^{2};q)_{-n}\,d\widehat{\mathsf{m}}_{q}(z),
    \end{equation}
    where for every $n>0$
    \begin{equation*}
        (x;q)_{-n}
        :=\frac{1}{(xq^{-n};q)_{n}}.
    \end{equation*}
\end{lemma}

\begin{proof}
    It suffices to verify the assertion for monomials $f(z,\overline{z})=z^{a}\overline{z}^{b}$ and $g(z,\overline{z})=z^{c}\overline{z}^{d}$. The proof then follows directly from the change-of-variables formula; see \eqref{eq:change_of_variable}.
\end{proof}

\begin{lemma}
    Fix $r\geq0$ and $0\leq j\leq r$. For an analytic polynomial $f(z)$, we have
    \begin{equation}\label{eqn:partialT}
        \partial^{j}_{q,\overline{z}}{\mathbf{T}}^{r}f(z)
        =(1-q)^{-j}\qbinom{r}{j}[j]_{q}![\mathsf{S}_{q,z}^{j}\mathbf{T}^{r-j}f](z,\overline{z}).
    \end{equation}
\end{lemma}

\begin{proof}
    Using the gauge form of $\mathbf{T}$ and $q$-Leibniz rule, we obtain
    \begin{equation*}
    \begin{split}
        \partial_{q,\overline{z}}^{j}{\mathbf{T}}^{r}f(z)
        &=(-1)^{r}\partial_{q,\overline{z}}^{j}\widehat{E}_{q}(|z|^{2})\partial_{q,z}^{r}[f(z)\widehat{e}_{q}(-|z|^{2})] \\
        &=(-1)^{r}\partial_{q,\overline{z}}^{j}\widehat{E}_{q}(|z|^{2})\sum_{k=0}^{r}\qbinom{r}{k}[\partial_{q,z}^{k}f](q^{r-k}z)[\partial_{q,z}^{r-k}\widehat{e}_{q}](-|z|^{2})\\
        &=(-1)^{r}\partial_{q,\overline{z}}^{j}\sum_{k=0}^{r}\qbinom{r}{k}[\partial_{q,z}^{k}f](q^{r-k}z)\Big(-\frac{\overline{z}}{1-q}\Big)^{r-k}\\
        &=(-1)^{r+j}(1-q)^{-j}\frac{[r]_{q}!}{[r-j]_{q}!}\sum_{k=0}^{r-j}\qbinom{r-j}{k}\Big(-\frac{\overline{z}}{1-q}\Big)^{r-k-j}[\partial_{q,z}^{k}f](q^{r-k}z)\\
        &=(1-q)^{-j}\frac{[r]_{q}!}{[r-j]_{q}!}[{\mathbf{T}}^{r-j}f](q^{j}z,\overline{z}).
    \end{split}
    \end{equation*}
\end{proof}

\subsection{Reduction operator}

We now apply the identities established in the preceding subsection to the cyclic integral representation of the polynomial cumulants. By successively moving the raising operators from the kernel factors onto the test polynomial, we reduce each higher-level cyclic integral to a cyclic integral involving only the analytic kernel. The effect of the Landau-level indices is then encoded in a family of explicit q-difference operators.

Grouping the surjections $\sigma$ in the cumulant formula used in Section \ref{sec:CLT lowest} and using the projection property of $\widehat{\mathcal{K}}^q_{m, N}$ and $\widehat{\mathcal{K}}^{q, \textup{full}}_{m, N}$, we obtain the cyclic integral representation of the $k$-th cumulants of $\textup{trace}_{m, N}(g)$ and $\textup{trace}_{m, N}^\textup{full}(g)$ as
\begin{equation}\label{eq:cumulant_cyc_int_formula}
    C_k(\textup{trace}^\#_{m, N}(g)) = \int_{\mathbb{C}^k} G_k(z_1, \dots, z_k)\widehat{\mathcal{K}}^{q, \#}_{m, N}(z_1, z_2)\cdots \widehat{\mathcal{K}}^{q, \#}_{m, N}(z_k, z_1)d \widehat{\mathsf{m}}_q^{\otimes k}(z_1, \dots, z_k),
\end{equation}
where $\#$ is either omitted or $\textup{full}$ and
\begin{equation}\label{eqn: def of Gk}
    G_{k}(z_{1},\cdots,z_{k})
    =\sum_{j=1}^{k}\frac{(-1)^{j-1}}{j}\sum_{\substack{k_{1}+\cdots+k_{j}=k\\k_{1},\cdots,k_{j}\geq1}}\frac{k!}{k_{1}!\cdots k_{j}!}\prod_{l=1}^{j}g(z_{l})^{k_{l}}.
\end{equation}
In particular, in the analytic case $m=1$, we have 
\begin{equation*}
    C_{k}(\trace_{1,N}(g))
    =\mathcal{I}_{N,k}(G_{k}),
\end{equation*}
where for a polyanalytic polynomial $F$ in $k$ variables, define the analytic cyclic integral
\begin{equation*}
    \mathcal I_{N,k}(F)
    =\int_{\mathbb{C}^{k}}F(z_{1},\cdots,z_{k})\widehat{\mathcal{K}}_{N}^{q}(z_{1},z_{2})\cdots\widehat{\mathcal{K}}_{N}^{q}(z_{k},z_{1})\,d\widehat{\mathsf{m}}_{q}^{\otimes k}(z_{1},\cdots,z_{k}).
\end{equation*}
The goal of this section is to express the general pure and full cumulants in terms of $\mathcal{I}_{N, k}$, with their dependence on the higher Landau levels encoded by explicit $q$-difference operators acting on $G_k$. 

For nonnegative integers $\alpha$, $\beta$, define the reduction operator
\begin{equation*}
    {\mathcal{D}}_{\alpha,\beta}^{q}f
    =\sum_{j=0}^{\alpha\wedge\beta}\qbinom{\alpha}{j}\qbinom{\beta}{j}\frac{[j]_{q}!}{(1-q)^{j}}q^{-j-\frac{\beta(\beta-1)}{2}-\frac{(\alpha-j)(\alpha+j+1)}{2}} \mathsf{S}_{q,\overline{z}}^{-\alpha}\partial_{q,\overline{z}}^{\alpha-j}\big[{(-q^{j}|z|^{2};q)_{\beta-j}}\mathsf{S}_{q,z}^{j}\partial_{q,z}^{\beta-j}f\big].
\end{equation*}

\begin{lemma}\label{lem:reduction}
    Let $F$ be a polyanalytic polynomial and let $\alpha$, $\beta$ be nonnegative integers. For analytic polynomials $f$, $g$, we have
    \begin{equation*}
        \int_{\mathbb{C}}F(z,\overline{z})\big[\mathbf{T}^{\beta}_{z}\mathsf{S}_{q,z}^{-\beta}f\big](z,\overline{z})\overline{\big[\mathbf{T}^{\alpha}_{z}g\big](z,\overline{z})}\,d\widehat{\mathsf{m}}_{q}(z)
        =\int_{\mathbb{C}}\big[\mathcal{D}_{\alpha,\beta}^{q}F\big](z,\overline{z})f(z)\overline{g(z)}\,d\widehat{\mathsf{m}}_{q}(z).
    \end{equation*}
\end{lemma}

\begin{proof}
    First assume that $\alpha\geq\beta$. Applying \eqref{eqn:Thatr adjoint} and \eqref{eqn:Sn adjoint} in the $z$-variable yields
    \begin{equation}\label{eqn:analytic adjoint}
        \int_{\mathbb{C}}F(z,\overline{z})\big[\mathbf{T}^{\beta}_{z}\mathsf{S}_{q,z}^{-\beta}f\big](z,\overline{z})\overline{\big[\mathbf{T}^{\alpha}_{z}g\big](z,\overline{z})}\,d\widehat{\mathsf{m}}_{q}(z)
        =q^{-\frac{\beta(\beta-1)}{2}}\int_{\mathbb{C}}(-|z|^{2};q)_{\beta}f(z)\partial_{q,z}^{\beta}\big[F(z,\overline{z})\overline{\big[\mathbf{T}^{\alpha}_{z}g\big](z,\overline{z})}\big]\,d\widehat{\mathsf{m}}_{q}(z).
    \end{equation}
    The $q$-Leibniz rule and \eqref{eqn:partialT} give
    \begin{align*}
        \partial_{q,z}^{\beta}\big[F(z,\overline{z})\overline{\big[\mathbf{T}^{\alpha}_{z}g\big](z,\overline{z})}\big]
        &=\sum_{j=0}^{\beta}\qbinom{\beta}{j}\big[\mathsf{S}_{q,z}^{j}\partial_{q,z}^{\beta-j}F](z,\overline{z})\overline{\big[\partial_{q,\overline{z}}^{j}\mathbf{T}^{\alpha}_{z}g\big](z,\overline{z})}\smallskip\\
        &=\sum_{j=0}^{\beta}(1-q)^{-j}\qbinom{\alpha}{j}\qbinom{\beta}{j}[j]_{q}!\big[\mathsf{S}_{q,z}^{j}\partial_{q,z}^{\beta-j}F](z,\overline{z})\overline{\big[\mathsf{S}_{q,z}^{j}\mathbf{T}^{\alpha-j}_{z}g\big](z,\overline{z})}.
    \end{align*}
    Substitute this into the right-hand side of \eqref{eqn:analytic adjoint} and fix $0\leq j\leq\beta$. Since $f$ is an analytic polynomial, we apply the transpose of $\mathsf{S}_{q,\overline{z}}^{j}$ and $\mathbf{T}_{\overline{z}}^{\alpha-j}$. Then the $j$-th summand becomes
    \begin{equation*}
    \begin{split}
        &q^{-\frac{\beta(\beta-1)}{2}-j}(1-q)^{-j}\qbinom{\alpha}{j}\qbinom{\beta}{j}[j]_{q}!\int_{\mathbb{C}}(-q^{-j}|z|^{2};q)_{\beta}(-|z|^{2};q)_{-j}f(z)\big[\mathsf{S}_{q,\overline{z}}^{-j}\mathsf{S}_{q,z}^{j}\partial_{q,z}^{\beta-j}F\big](z,\overline{z})\overline{\big[\mathbf{T}_{z}^{\alpha-j}g\big](z,\overline{z})}\,d\widehat{\mathsf{m}}_{q}(z)\smallskip\\
        &=\qbinom{\alpha}{j}\qbinom{\beta}{j}\frac{[j]_{q}!}{(1-q)^{j}}q^{-j-\frac{\beta(\beta-1)}{2}-\frac{(\alpha-j)(\alpha+j+1)}{2}}\int_{\mathbb{C}}\mathsf{S}_{q,\overline{z}}^{-\alpha}\partial_{q,\overline{z}}^{\alpha-j}\big[{(-q^{j}|z|^{2};q)_{\beta-j}}\mathsf{S}_{q,z}^{j}\partial_{q,z}^{\beta-j}F\big]f(z)\overline{g(z)}\,d\widehat{\mathsf{m}}_{q}(z).
    \end{split}
    \end{equation*}
    Taking the sum over $0\leq j\leq\beta$, we complete the proof in the case $\alpha\geq\beta$. The proof in the case $\alpha<\beta$ proceeds in parallel.
\end{proof}

We define the normalized reduction operator
\begin{equation*}
    \widetilde{\mathcal{D}}_{\alpha,\beta}^{q}
    :=\gamma_{\alpha}\gamma_{\beta}\mathcal{D}_{\alpha,\beta}^{q},
\end{equation*}
where the constant $\gamma_{r}$ is inherited from the raising operators \eqref{eqn:gamma}. For a multi-index $\mathbf{i}=(i_{1},\cdots,i_{k})$, define 
\begin{equation*}
    \widetilde{\mathcal{D}}^{q}_{\mathbf{i}}
    :=\prod_{j=1}^{k}\big[\widetilde{\mathcal{D}}^{q}_{i_{j-1},i_{j}}\big]_{z_{j}}
\end{equation*}
with the convention $i_{0}=i_{k}$.

\begin{proposition}
    Let $g$ be a polynomial and $G_{k}$ be given by \eqref{eqn: def of Gk}. For $m\geq1$, the cumulant of the linear statistics has the representation
    \begin{equation}\label{eqn:pure cumulant I}
    \begin{split}
        C_{k}(\trace_{m,N}(g))
        &=\mathcal{I}_{N,k}(\widetilde{\mathcal{D}}^{q}_{m-1,\cdots,m-1}G_{k}),\smallskip\\
        C_{k}(\trace_{m,N}^{\full}(g))
        &=\sum_{0\leq i_{1},\cdots,i_{k}\leq m-1}\mathcal{I}_{N,k}(\widetilde{\mathcal{D}}^{q}_{i_{1},\cdots,i_{k}}G_{k}).
    \end{split}
    \end{equation}
\end{proposition}

\begin{proof}
    We first consider the pure case. Substituting the raising-operator representation \eqref{eqn:kernel T} into the cyclic integral formula \eqref{eq:cumulant_cyc_int_formula}, we isolate the integration with respect to $z_j$ and regard the remaining variables as fixed. By Lemma \ref{lem:reduction}, up to factors independent of $z_j$, this integration takes the form
    \begin{align*}
        &\int_{\mathbb{C}}G_{k}(z_{1},\cdots,z_{k})[\overline{\mathbf{T}_{m-1}}]_{\overline{z_{j}}}\widehat{\mathcal{K}}^{q}_{N}(z_{j-1},z_{j})[\mathbf{T}_{m-1}\mathsf{S}_{q}^{-m+1}]_{z_{j}}\widehat{\mathcal{K}}^{q}_{N}(z_{j}, z_{j+1})\,d\widehat{\mathsf{m}}_{q}(z_{j})\\
        &=\int_{\mathbb{C}}[\widetilde{\mathcal{D}}^{q}_{m-1,m-1}]_{z_{j}}G_{k}(z_{1},\cdots, z_{k})\widehat{\mathcal{K}}^{q}_{N}(z_{j-1},z_{j})\widehat{\mathcal{K}}^{q}_{N}(z_{j}, z_{j+1})\,d\widehat{\mathsf{m}}_{q}(z_{j}),
    \end{align*}
    where $z_{0}=z_{k}$ and $z_{k+1}=z_{1}$. Applying this calculation successively to $z_{1},\cdots,z_{k}$, we obtain 
    \begin{equation*}
        C_{k}(\trace_{m,N}(g))
        =\mathcal{I}_{N,k}(\widetilde{\mathcal{D}}^{q}_{m-1,\cdots,m-1}G_{k}),
    \end{equation*}
    where the normalization factors coming from the raising maps are absorbed into the definition of $\widetilde{\mathcal{D}}^{q}_{m-1,\cdots,m-1}$.

    For the full case, we first expand each kernel in \eqref{eqn:kernel T} into the sum of its Landau-level components. This yields a finite sum indexed by $(i_{1},\cdots,i_{k})\in\{0,\cdots,m-1\}^k$. For a fixed multi-index, the two kernel factors meeting at $z_{j}$ carry the levels $i_{j-1}$ and $i_{j}$, with the convention $i_{0}=i_{k}$. Applying Lemma \ref{lem:reduction} successively to $z_{1},\cdots,z_{k}$ therefore gives
    \begin{equation*}
        \mathcal{I}_{N,k}(\widetilde{\mathcal{D}}^{q}_{i_{1},\cdots,i_{k}}G_{k}).
    \end{equation*}
    Summing over all level indices completes the proof.
\end{proof}

\section{Proof of the central limit theorem}
\label{sec:CLT}

\subsection{Asymptotics of the reduction operator}

We proceed to determine the asymptotic behavior of the operator $\widetilde{\mathcal{D}}_{\alpha,\beta}^{q}$ as $N\to\infty$ with $q=e^{-\lambda/N}$. For each $d\geq0$, let $\mathscr{P}_{d}$ be the space of polynomials in $z$ and $\overline{z}$ of degree at most $d$, that is,
\begin{equation*}
    \mathscr{P}_{d}
    =\operatorname{span}\{z^{\alpha}\overline{z}^{\beta}:\alpha+\beta\leq d\}.
\end{equation*}
Here, $\mathscr{P}_{d}$ is naturally endowed with the $\ell^{1}$-norm of the coefficients. Fix $m\geq1$ and $d\geq0$. Then, for each $0\leq\alpha,\beta\leq m-1$, the restriction $\widetilde{\mathcal{D}}_{\alpha,\beta}^{q}$ to $\mathscr{P}_{d}$ can be viewed as a linear operator from $\mathscr{P}_{d}$ to $\mathscr{P}_{d+m}$. Under this restriction, the definition of $\widetilde{\mathcal{D}}_{\alpha,\beta}^{q}$ yields the operator-norm estimate 
\begin{equation}\label{eqn:mathcal tilde D operator norm}
    \|\widetilde{\mathcal{D}}_{\alpha,\beta}^{q}\|_{\mathscr{P}_{d}\to\mathscr{P}_{d+m}}
    \leq C_{m,d}(1-q)^{\frac{|\alpha-\beta|}{2}},
\end{equation}
where $C_{m,d}>0$ is a constant depending on $m$ and $d$. Thus, for a multi-index $\mathbf{i}=(i_{1},\cdots,i_{k})$, the norm of corresponding operator has order
\begin{equation}\label{eqn:Di order}
    \|\widetilde{\mathcal{D}}_{\mathbf{i}}^{q}\|
    =O(N^{-\ell(\mathbf{i})})
\end{equation}
in the double-scaling regime $q=e^{-\lambda/N}$, where
\begin{equation*}
    \ell(\mathbf{i})
    =\frac{1}{2}\sum_{j=1}^{k}|i_{j}-i_{j-1}|,\qquad
    i_{0}:=i_{k}.
\end{equation*}
Combining \eqref{eqn:Di order} with the $O(N)$ bound for $\mathcal{I}_{N,k}$ established in Lemma \ref{lem:mathcal I diagonal} below, we obtain the estimate
\begin{equation*}
    \mathcal{I}_{N,k}(\widetilde{\mathcal{D}}_{\mathbf{i}}^{q}G_{k})
    =O(N^{1-\ell(\mathbf{i})})
\end{equation*}
as $N\to\infty$. Thus, we need only consider the cases $\alpha=\beta$ and $|\alpha-\beta|=1$ in the proof of the central limit theorem. Define a differential operator
\begin{equation*}
    \mathscr{L}
    :=\overline{z}\partial_{\overline{z}}+(1+|z|^{2})\partial_z\partial_{\overline{z}}.
\end{equation*}
The following lemma provides the asymptotic behavior of $\widetilde{\mathcal{D}}_{\alpha,\beta}^{q}$ in such cases.

\begin{lemma}\label{lem:mathcal D expansion}
    Fix $m\geq1$ and $d\geq0$. Then, we have a uniform $(1-q)$-expansion in $\|\cdot\|_{\mathscr{P}_{d}\to\mathscr{P}_{d+m}}$
    \begin{align}
        \widetilde{\mathcal{D}}_{r,r}^{q}
        &=\Id+r(1-q)\mathscr{L}+O((1-q)^{2}),\label{eqn:mathcal D rr expansion}
        \smallskip\\
        \widetilde{\mathcal{D}}_{r,r-1}^{q}
        &=(1-q)^{\frac{1}{2}}\sqrt{r}\partial_{\overline{z}}+O((1-q)^{\frac{3}{2}}),\label{eqn:mathcal D rr-1 expansion}
        \smallskip\\
        \widetilde{\mathcal{D}}_{r-1,r}^{q}
        &=(1-q)^{\frac{1}{2}}\sqrt{r}(1+|z|^{2}){\partial_z}+O((1-q)^{\frac{3}{2}}),\label{eqn:mathcal D r-1r expansion}
    \end{align}
    over $1\leq r\leq m-1$.
\end{lemma}

\begin{proof}
    Since $\widetilde{\mathcal{D}}_{\alpha,\beta}^{q}$ consists of the $q$-difference operator $\partial_{q,z}$, shift operator $\mathsf{S}_{q}^{a}$ and their complex conjugate, we begin with the asymptotics of these building blocks. First, it is straightforward from
    \begin{equation*}
        \partial_{q,z}f(z,\overline{z})
        =\frac{f(z)-f(qz)}{(1-q)z}
        =\partial_z f(z,\overline{z})+O(1-q)
    \end{equation*}
    that
    \begin{equation*}
        \partial_{q,z}
        =\partial_z+O(1-q).
    \end{equation*}
    Using this, for each integer $c$, the expansion
    \begin{equation*}
        \mathsf{S}_{q,z}f(z,\overline{z})
        =f(qz,\overline{z})
        =f(z,\overline{z})-(1-q)z\partial_{q,z}f(z,\overline{z})
    \end{equation*}
    yields the expansion
    \begin{equation*}
        \mathsf{S}_{q,z}^{c}
        =\Id-c(1-q)z\partial_z+O((1-q)^{2}).
    \end{equation*}
    
     By definition, $\widetilde{\mathcal{D}}_{\alpha,\beta}^{q}$ admits the $(1-q)$-expansion whose leading order is $(\alpha+\beta)/2-\alpha\wedge\beta$. In the case $\alpha=\beta=r$, the definition gives the expansion
    \begin{equation*}
        \widetilde{\mathcal{D}}_{r,r}^{q}
        =\mathsf{S}_{q,\overline{z}}^{-r}\mathsf{S}_{q,z}^{r}+(1-q)q^{1-r}[r]_{q}\mathsf{S}_{q,\overline{z}}^{-r}\partial_{q,\overline{z}}\big[(1+q^{r-1}|z|^{2})\mathsf{S}_{q,z}^{r-1}\partial_{q,z}]+O((1-q)^{2}).
    \end{equation*}
    Substituting the asymptotics of $\partial_{q,z}$ and $\mathsf{S}_{q}$, we obtain \eqref{eqn:mathcal D rr expansion}.

    Next, in the cases $(\alpha,\beta)=(r,r-1)$ and $(r-1,r)$, we only need the leading coefficients. The leading behavior of these operators is given by
    \begin{align*}
        \widetilde{\mathcal{D}}_{r,r-1}^{q}
        &=(1-q)^{\frac{1}{2}}q^{-\frac{r}{2}}\sqrt{[r]_{q}}\mathsf{S}_{q,\overline{z}}^{-r}\partial_{q,\overline{z}}\mathsf{S}_{q,z}^{r-1}+O((1-q)^{\frac{3}{2}}),\smallskip\\
        \widetilde{\mathcal{D}}_{r-1,r}^{q}
        &=(1-q)^{\frac{1}{2}}q^{1-\frac{r}{2}}\sqrt{[r]_{q}}\mathsf{S}_{q,\overline{z}}^{-(r-1)}\big[(1+q^{r-1}|z|^{2})\mathsf{S}_{q,z}^{r-1}\partial_{q,z}\big]+O((1-q)^{\frac{3}{2}}).
    \end{align*}
    Again, substituting the leading asymptotic of $\partial_{q,z}$ and $\mathsf{S}_{q}$, we recover \eqref{eqn:mathcal D rr-1 expansion} and \eqref{eqn:mathcal D r-1r expansion}.
\end{proof}

We turn to asymptotic analysis of $\mathcal{I}_{N,k}$. For a polyanalytic polynomial $F$ in $k$ variables $z_{1},\cdots,z_{k}$, we define the diagonal restriction of $F$ by
\begin{equation*}
    \Delta_{F}(z)
    :=F(z,\cdots,z).
\end{equation*}

\begin{lemma}\label{lem:mathcal I diagonal}
    Fix $d\geq0$ and $k\geq1$. Let $q$ be scaled according to \eqref{eqn:def of q scaling}. Let $\{F_{N}\}_{N\geq1}$ be a sequence of polyanalytic polynomials in $k$ variables of degree at most $d$ whose coefficients are uniformly bounded in $N$. Then,
    \begin{equation*}
        \mathcal{I}_{N,k}(F_{N})
        =\int_{\mathbb{C}}\Delta_{F_{N}}(z)\widehat{\mathcal{K}}^{q}_{N}(z,z)\,d\widehat{\mathsf{m}}_{q}(z)+O(1),
    \end{equation*}
    where the integral on the right-hand side is of order $O(N)$.
\end{lemma}

\begin{proof}
    By linearity, it suffices to consider monomials
    \begin{equation*}
        F(z_{1},\cdots,z_{k})
        =z^{\alpha}\overline{z}^{\beta}
    \end{equation*}
    for multi-indices $\alpha=(\alpha_{1},\cdots,\alpha_{k})$ and $\beta=(\beta_{1},\cdots,\beta_{k})$. If $|\alpha|\neq|\beta|$, both integrals vanish by rotational invariance. Suppose $|\alpha|=|\beta|=s$. The assertion follows directly from \eqref{eqn:Phij expansion} for $\mathcal{I}_{N, 1}(\Delta_F)= \Phi_{1}(z^{s}\overline{z}^{s})$ and $\mathcal{I}_{N, k}(F)=\Phi_{k}(z^{\alpha}\overline{z}^{\beta})$.
\end{proof}

As a direct consequence, under the same assumption, if $\Delta_{F_{N}}=0$, we have
\begin{equation}\label{eqn:mathcal I diagonal}
    \mathcal{I}_{N,k}(F_{N})=O(1).
\end{equation}
Moreover, if $\Delta_{F_{N}}$ converges to a polynomial $\Delta_{F}$ coefficientwise, by Lemma \ref{lem:rescaled moment limit} and rotational invariance, we have
\begin{equation*}
    \frac{1}{N}\int_{\mathbb{C}}\Delta_{F_{N}}(z)\widehat{\mathcal{K}}^{q}_{N}(z,z) d\widehat{\mathsf{m}}_{q}(z)
    =\int_{D_{\lambda}}\Delta_{F}(z)\widehat{\rho}_{\lambda}(z)\,dA(z)+o(1).
\end{equation*}
Hence, Lemma \ref{lem:mathcal I diagonal} yields
\begin{equation}\label{eqn:mathcal I LLN}
    \mathcal{I}_{N,k}(F_{N})
    =N\int_{D_{\lambda}}\Delta_{F}(z)\widehat{\rho}_{\lambda}(z)\,dA(z)+o(N),
\end{equation}
as $N\to\infty$.

\subsection{Polynomial test functions}

We are ready to prove the central limit theorem for polynomial test functions in the polyanalytic setting. We begin with the pure polyanalytic ensemble.

\begin{proposition}\label{prop:pure polynomial cumulant}
Let $g$ be a real-valued polynomial in $z$, $\overline{z}$ and $m\geq1$. Then,
\begin{equation*}
    \lim_{N\to\infty}C_{k}(\trace_{m,N}(g))
    =\begin{cases}
       \displaystyle(2m-1)\|g\|_{H^{1}(D_{\lambda}^{\circ})}^{2}
       +\frac{1}{2}\|g\|_{H^{1/2}(\partial D_{\lambda})}^{2}, & \textup{if }k=2,\smallskip\\
       0,& \textup{if } k\geq3.
     \end{cases}
\end{equation*}
\end{proposition}

\begin{proof}
    Recall the representation of the cumulants
    \begin{equation*}
        C_{k}(\trace_{m,N}(g))
        =\mathcal{I}_{N,k}(\widetilde{\mathcal{D}}_{m-1,\cdots,m-1}^{q}G_{k}).
    \end{equation*}
    See \eqref{eqn:pure cumulant I}. Using the expansion \eqref{eqn:mathcal D rr expansion}, we obtain
    \begin{equation*}
        \widetilde{\mathcal{D}}_{m-1,\cdots,m-1}^{q}G_{k}
        =G_{k}+(m-1)(1-q)\sum_{j=1}^{k}\mathscr{L}_{z_{j}}G_{k}+O((1-q)^{2}),
    \end{equation*}
    where the error term is uniform in $N$. Thus, by \eqref{eqn:pure cumulant I} and Lemma \ref{lem:mathcal I diagonal}, we obtain a uniform expansion
    \begin{equation}\label{eqn:pure polynomial cumulant expansion}
        C_{k}(\trace_{m,N}(g))
        =\mathcal{I}_{N,k}(G_{k})+(m-1)(1-q)\sum_{j=1}^{k}\mathcal{I}_{N,k}(\mathscr{L}_{z_j}G_{k})+O(N(1-q)^{2}).
    \end{equation}
    First, consider the case $k\geq3$. Here, by the projection property of $\widehat{\mathcal{K}}^q_N$ we have $\mathcal{I}_{N,k}(G_{k})=C_{k}(\textup{trace}_{1, N}(g))$, and it vanishes as $N$ tends to infinity by \eqref{eqn:analytic cumulant limit}. Moreover, by \cite[Lemmas 3.2, 3.3]{AHM11}, the polynomials
    \begin{equation*}
        \sum_{j=1}^{k}\partial_{\overline{z_{j}}}G_{k},\qquad
        \sum_{j=1}^{k}\partial_{z_j}\partial_{\overline{z_j}}G_{k}
    \end{equation*}
    vanish on diagonal. Consequently, $\sum_{j=1}^{k}\mathscr{L}_{z_j}G_{k}$ also vanishes on the diagonal, and hence \eqref{eqn:mathcal I diagonal} yields
    \begin{equation*}
        (m-1)\sum_{j=1}^{k}\mathcal{I}_{N,k}(\mathscr{L}_{z_j}G_{k})
        =O(1).
    \end{equation*}
    Since $1-q=O(N^{-1})$, all cumulants of order $k\geq3$ vanish as $N\to\infty$.

    We turn to the second cumulant. For the first term in \eqref{eqn:pure polynomial cumulant expansion}, we obtain $\mathcal{I}_{N,2}(G_{2})=C_2(\textup{trace}_{1, N}(g))$, and by \eqref{eqn:analytic cumulant limit}, it converges to
    \begin{equation}\label{eqn:pure polynomial cumulant 1}
        \|g\|_{H^{1}(D_{\lambda}^{\circ})}^{2}+\frac{1}{2}\|g\|_{H^{1/2}(\partial D_{\lambda})}^{2}
    \end{equation}
    as $N\to\infty$. For the second term in \eqref{eqn:pure polynomial cumulant expansion}, by symmetrization, we can replace $G_{2}$ by
    \begin{equation}\label{eqn:G2 alternative}
        G_{2}(z_{1},z_{2})
        =\frac{1}{2}\big(g(z_{1})-g(z_{2})\big)^{2}.
    \end{equation}
    With this choice, a direct calculation gives
    \begin{equation}\label{eqn:G2 diagonal}
        \big(\mathscr{L}_{z}+\mathscr{L}_{w}\big)G_{2}(z,w)\Big|_{z=w}
        =2(1+|z|^{2})|\partial_{\overline{z}}g(z)|^{2}.
    \end{equation}
     Moreover, substituting \eqref{eqn:G2 diagonal} into \eqref{eqn:mathcal I LLN}, we obtain
    \begin{equation}\label{eqn:pure polynomial cumulant 2}
    \begin{split}
        \lim_{N\to\infty}(m-1)(1-q)\mathcal{I}_{N,2}\big((\mathscr{L}_{z_{1}}+\mathscr{L}_{z_{2}})G_{2}\big)
        &=(m-1)\lambda\int_{D_{\lambda}}2(1+|z|^{2})|\overline{\partial}g(z)|^{2}\frac{dA(z)}{\lambda(1+|z|^{2})}\smallskip \\
        &=2(m-1)\|g\|_{H^{1}(D_{\lambda}^{\circ})}^{2}.
    \end{split}
    \end{equation}
    Substituting \eqref{eqn:pure polynomial cumulant 1} and \eqref{eqn:pure polynomial cumulant 2} into \eqref{eqn:pure polynomial cumulant expansion} with $k=2$, we complete the proof.
\end{proof}

We note that the second cumulant coincides with the variance in \eqref{eqn:pure CLT}. We turn to the limiting behavior of the cumulants associated with the full ensemble with polynomial test function, which recovers \eqref{eqn:full CLT}.

\begin{proposition}
Let $g$ be a real-valued polynomial in $z$, $\overline{z}$ and $m\geq1$. For every $k\geq2$,
\begin{equation*}
    \lim_{N\to\infty}C_k(\trace_{m,N}^{\full}(g))
    =\begin{cases}
       m\Big(\|g\|_{H^{1}(D_{\lambda}^{\circ})}^{2}+\frac{1}{2}\|g\|_{H^{1/2}(\partial D_{\lambda})}^{2}\Big), & k=2,\smallskip\\
       0,& k\geq3.
     \end{cases}
\end{equation*}
\end{proposition}

\begin{proof}
    We begin with the cumulant formula
    \begin{equation*}
        C_{k}(\trace_{m,N}^{\full}(g))
        =\sum_{0\leq i_{1},\cdots,i_{k}\leq m-1}\mathcal{I}_{N,k}\big(\widetilde{\mathcal{D}}_{i_{1},\cdots,i_{k}}^{q}G_{k}\big).
    \end{equation*}
    By \eqref{eqn:mathcal tilde D operator norm}, we have
    \begin{equation*}
        \widetilde{\mathcal{D}}_{\mathbf{i}}^{q}G_{k}
        =(1-q)^{\ell(\mathbf{i})}R_{N,k},
    \end{equation*}
    where $R_{N,k}$ is a sequence of polynomials with uniformly bounded degree and coefficients. Thus, by Lemma \ref{lem:mathcal I diagonal}, every multi-index $\mathbf{i}$ with $\ell(\mathbf{i})\geq2$ contributes at most
    \begin{equation*}
        O(N(1-q)^{2})=o(1)
    \end{equation*}
    to the cumulants.

    The multi-indices with $\ell(\mathbf{i})=0$ are the constant indices, i.e., $\mathbf{i}=(r,\cdots,r)$ for each $0\leq r\leq m-1$. In this case, they correspond to the cumulants of pure ensembles, hence, their contribution is
    \begin{equation*}
        \sum_{r=0}^{m-1}C_{k}(\trace_{r+1,N}(g)).
    \end{equation*}

    Next, a multi-index $\mathbf{i}$ satisfies $\ell(\mathbf{i})=1$ if and only if, for some $1\leq r \leq m-1$, $\mathbf{i}$ takes only the values $r-1$ and $r$, and contains exactly one block of $r$ and one block of $r-1$ under the periodic index convention $i_{0}=i_{k}$. If the transition from $r-1$ to $r$ occurs at $a$ and the transition from $r$ to $r-1$ occurs at $b$, by Lemma \ref{lem:mathcal D expansion}, we have
    \begin{equation*}
        \widetilde{\mathcal{D}}_{\mathbf{i}}^{q}G_{k}
        =r(1-q)(1+|z_{a}|^{2})\partial_{z_a}\partial_{\overline{z_b}}G_{k}+(1-q)^{2}R_{\mathbf{i},N},
    \end{equation*}
    where $R_{\mathbf{i},N}$ is a polynomial whose degree and coefficients are bounded uniformly in $N$. Considering all possible choices of $a$ and $b$, we obtain a uniform expansion
    \begin{equation}\label{eqn:full polynomial cumulant expansion}
        C_{k}(\trace_{m,N}^{\full}(g))
        =\sum_{r=0}^{m-1}C_{k}(\trace_{r+1,N}(g))+(1-q)\sum_{r=1}^{m-1}r\mathcal{I}_{N,k}(H_{k})+O(N(1-q)^{2}),
    \end{equation}
    where
    \begin{equation*}
        H_{k}
        =\sum_{1\leq a,b\leq k,~a\neq b}(1+|z_{a}|^{2})\partial_{z_a}\partial_{\overline{z_b}}G_{k}.
    \end{equation*}
    We first consider the case $k\geq3$. By Proposition \ref{prop:pure polynomial cumulant}, the cumulants of pure ensembles tend to zero. On the diagonal of $H_{k}$, by \cite[Lemma 3.4]{AHM11}, we have
    \begin{equation*}
        \Delta_{H_{k}}(z)
        =2(1+|z|^{2})\Re\left[\sum_{a<b}\partial_{z_a}\partial_{\overline{z_b}}G_{k}\right]
        \Bigg|_{z_{1}=\cdots=z_{k}=z}
        =0.
    \end{equation*}
    Thus, $\mathcal{I}_{N,k}(H_{k})=O(1)$ and the second term in \eqref{eqn:full polynomial cumulant expansion} vanishes in the large-$N$ limit.

    Finally, for $k=2$, using
    \begin{equation*}
        H_{2}(z,w)=\big((1+|z|^{2})\partial_{z}\partial_{\overline{w}}+(1+|w|^{2})\partial_{w}\partial_{\overline{z}}\big)G_{2}(z,w)
    \end{equation*}
    with \eqref{eqn:G2 alternative}, we obtain
    \begin{equation*}
        \Delta_{H_{2}}(z)
        =-2(1+|z|^{2})|\partial_{\overline{z}}g(z)|^{2}.
    \end{equation*}
    Therefore, we have the leading behavior
    \begin{equation}\label{eqn:full polynomial cumulant 1}
        \lim_{N\to\infty}(1-q)\sum_{r=1}^{m-1}r\mathcal{I}_{N,2}(H_{2})
        =-2\sum_{r=1}^{m-1}r\|g\|_{H^{1}(D_{\lambda}^{\circ})}^{2}
        =-m(m-1)\|g\|_{H^{1}(D_{\lambda}^{\circ})}^{2}.
    \end{equation}
    On the other hand, from the cumulant formula for pure ensembles, we have
    \begin{equation}\label{eqn:full polynomial cumulant 2}
        \lim_{N\to\infty}\sum_{r=0}^{m-1}C_{2}(\trace_{r+1,N}(g))
        =m^{2}\|g\|_{H^{1}(D_{\lambda}^{\circ})}^{2}+\frac{m}{2}\|g\|_{H^{1/2}(\partial D_{\lambda})}^{2}.
    \end{equation}
    Combining \eqref{eqn:full polynomial cumulant 1} and \eqref{eqn:full polynomial cumulant 2}, we complete the proof.
\end{proof}

\subsection{Extension to smooth test functions}

It remains to extend the central limit theorem to real-valued smooth test functions with compact support. For brevity, let
\begin{equation*}
    \varphi_{r,j}^{q}(z)
    :=\widehat{\mathcal{B}}_{r}^{q}|j\rangle(z,\overline z),
    \qquad r,j\geq0,
\end{equation*}
so that
\begin{equation*}
    \widehat{\mathcal{K}}_{m,N}^{q}(z,w)
    =\sum_{j=0}^{N-1}\varphi_{m-1,j}^{q}(z)\overline{\varphi_{m-1,j}^{q}(w)},\qquad
    \widehat{\mathcal{K}}_{m,N}^{q,\full}(z,w)
    =\sum_{r=0}^{m-1}\sum_{j=0}^{N-1}\varphi_{r,j}^{q}(z)\overline{\varphi_{r,j}^{q}(w)}.
\end{equation*}
In terms of $\varphi^q_{r, j}$, the three-term relations \eqref{eqn:threeterm z} and \eqref{eqn:threeterm barz} become
\begin{equation}\label{eqn:varphi threeterm z}
\begin{split}
    z\varphi_{r,j}^{q}
    &=q^{-\frac{j+r+1}{2}}\sqrt{1-q^{j+1}}\varphi_{r,j+1}^{q}-q^{-\frac{j+r}{2}}\sqrt{1-q^{r}}\varphi_{r-1,j}^{q}, \smallskip \\
    \overline{z}\varphi_{r,j}^{q}
    &=-q^{-\frac{j+r+1}{2}}\sqrt{1-q^{r+1}}\varphi_{r+1,j}^{q}+q^{-\frac{j+r}{2}}\sqrt{1-q^{j}}\varphi_{r,j-1}^{q},
\end{split}
\end{equation}
with the convention that vectors with a negative index are zero.

We begin with some auxiliary estimates.

\begin{lemma}\label{lem:Lipschitz high}
    Fix $m\geq1$. For a real-valued Lipschitz function $f$, we have a uniform bound
    \begin{equation*}
        \Var(\fluct_{m,N}^{\#}(f))
        \leq q^{1-m}[m]_{q}(e^{\lambda}-1)\Lip(f)^{2}.
    \end{equation*}
    Here, $\#$ is either omitted or full.
\end{lemma}

\begin{proof}
    Since $\widehat{\mathcal{K}}^{q,\#}_{m,N}$ is a projection kernel, the variance of the linear statistic admits an integral representation
    \begin{equation*}
        \Var(\fluct_{m,N}^{\#}(f))
        =\frac{1}{2}\int_{\mathbb{C}^{2}}|f(z)-f(w)|^{2}|\widehat{\mathcal{K}}_{m,N}^{q,\#}(z,w)|^{2}\,d\widehat{\mathsf{m}}_{q}^{\otimes2}(z,w),
    \end{equation*}
    which yields
    \begin{equation*}
        \Var(\fluct_{m,N}^{\#}(f))
        \leq\frac{1}{2}\Lip(f)^{2}\int_{\mathbb{C}^{2}}|z-w|^{2}|\widehat{\mathcal{K}}_{m,N}^{q,\#}(z,w)|^{2}\,d\widehat{\mathsf{m}}_{q}^{\otimes2}(z,w).
    \end{equation*}
    For the pure case, by expanding the integrand, we obtain
    \begin{equation*}
    \begin{split}
        |z-w|^{2}|\widehat{\mathcal{K}}^{q}_{m,N}(z,w)|^{2}
        &=\sum_{j,k=0}^{N-1}\big(z\varphi_{j}^{q}(z)\overline{z\varphi_{k}^{q}(z)}\varphi_{k}^{q}(w)\overline{\varphi_{j}^{q}(w)}+\varphi_{j}^{q}(z)\overline{\varphi_{k}^{q}(z)}w\varphi_{k}^{q}(w)\overline{w\varphi_{j}^{q}(w)}\big)\smallskip\\
        &\qquad-\sum_{j,k=0}^{N-1}\big(z\varphi_{j}^{q}(z)\overline{\varphi_{k}^{q}(z)}\varphi_{k}^{q}(w)\overline{w\varphi_{j}^{q}(w)}+\varphi_{j}^{q}(z)\overline{z\varphi_{k}^{q}(z)}w\varphi_{k}^{q}(w)\overline{\varphi_{j}^{q}(w)}\big),
    \end{split}
    \end{equation*}
    where we temporarily drop the subscript $m-1$ for brevity, i.e., $\varphi_{j}^{q}=\varphi_{m-1,j}^{q}$. The orthonormality of $\{\varphi_{j}^{q}\}_{j\geq0}$ in $L^{2}(\mathbb{C},\widehat{\mathsf{m}}_{q})$ and the three-term recurrence relation \eqref{eqn:varphi threeterm z} yield
    \begin{align*}
        \int_{\mathbb{C}^{2}}|z-w|^{2}|\widehat{\mathcal{K}}^{q}_{m,N}(z,w)|^{2}\,d\widehat{\mathsf{m}}_{q}^{\otimes2}(z,w)
        &=2q^{-(N+m-1)}(1-e^{-\lambda})+2\sum_{j=0}^{N-1}q^{-(j+m-1)}(1-q^{m-1})\smallskip\\
        &=2q^{1-m}[m]_{q}(e^{\lambda}-1).
    \end{align*}
    For the full ensemble, a similar calculation yields
    \begin{align*}
        \int_{\mathbb{C}^{2}}|z-w|^{2}|\widehat{\mathcal{K}}^{q,\full}_{m,N}(z,w)|^{2}\,d\widehat{\mathsf{m}}_{q}^{\otimes2}(z,w)
        =2q^{1-m}[m]_{q}(e^{\lambda}-1).
    \end{align*}
    This completes the proof.
\end{proof}

Recall that the support of our limiting density is a disk of radius $\sqrt{e^{\lambda}-1}$. The following lemma shows that the contribution from the outside of the droplet to the integral of any test function of polynomial growth vanishes in the large-$N$ regime.

\begin{lemma}\label{lem:local high}
    Fix $R>\sqrt{e^{\lambda}-1}$, $m\geq1$ and a nonnegative integer $p$. Let $q$ be scaled as in \eqref{eqn:def of q scaling}. Then
    \begin{equation*}
        \lim_{N\to\infty}\int_{|z|>R}|z|^{2p}\widehat{\mathcal{K}}^{q,\#}_{m,N}(z,z)\,d\widehat{\mathsf{m}}_{q}(z)
        =0.
    \end{equation*}
    Here, $\#$ is either omitted or full.
\end{lemma}

\begin{proof}
    Since $\widehat{\mathcal{K}}^{q, \textup{full}}_{m, N}(z, z)=\sum_{r=1}^m\widehat{\mathcal{K}}^{q}_{r, N}(z, z)$, it suffices to consider only the pure case. Fix a sufficiently small constant $\varepsilon>0$ such that
    \begin{equation*}
        e^{\lambda(1+2\varepsilon)}-1<R^{2},
    \end{equation*}
    and set $k_{N}=\lfloor{\varepsilon N}\rfloor$. We first obtain
    \begin{align*}
        \int_{|z|>R} |z|^{2p}\widehat{\mathcal{K}}^q_{m, N}(z, z) d\widehat{\mathsf{m}}_q(z) 
        &\leq R^{-2k_N}\sum_{j=0}^{N-1}\int_{|z|>R}|z|^{2(k_N+p)}|\varphi_{m-1,j}^{q}(z)|^{2}\,d\widehat{\mathsf{m}}_{q}(z)\\
        &\leq R^{-2k_N}\sum_{j=0}^{N-1}\|z^{k_N+p}\varphi^q_{m-1, j}\|^2_{L^2(\mathbb{C}, \widehat{\mathsf{m}}_q)}.
    \end{align*}
    Repeatedly applying the first relation in \eqref{eqn:varphi threeterm z}, we obtain
    \begin{equation*}
        z^{k_{N}+p}\varphi^{q}_{m-1,j}
        =\sum_{l=0}^{m-1}c_{l,N,j}\varphi^{q}_{m-1-l,j+k_{N}+p-l},
    \end{equation*}
    where $c_{l,N,j}$ is the sum of the contributions from all paths containing exactly $l$ level-lowering steps. There are $\binom{k_N+p}{l}$ such paths. We next estimate $c_{l, N, j}$. Fix $\rho>0$ with $e^{\lambda(1+2\varepsilon)}-1<\rho^{2}<R^{2}$. At each step of the iteration, multiplying $\varphi^q_{s, n}$ by $z$ produces two terms with respective coefficient factors
    \begin{equation*}
        q^{-\frac{n+s+1}{2}}\sqrt{1-q^{n+1}} \qquad
        \text{or} \qquad
        -q^{-\frac{n+s}{2}}\sqrt{1-q^{s}},
    \end{equation*}
    where $0\leq s\leq m-1$ and $j\leq n\leq j+k_{N}+p$. In both cases, for $0\leq j\leq N-1$, all coefficients are bounded by $\rho$ for all sufficiently large $N$. Thus, we have $|c_{l, N, j}|\leq \binom{k_N+p}{l}\rho^{k_N+p}$. Moreover, by the orthonormality of $\{\varphi^q_{s, n}\}$,
    \begin{equation*}
        \|z^{k_{N}+p}\varphi^{q}_{m-1,j}\|^{2}_{L^{2}(\mathbb{C}, \widehat{\mathsf{m}}_{q})}
        \leq\rho^{2(k_{N}+p)}\sum_{l=0}^{m-1}\binom{k_{N}+p}{l}^{2}
        <C_{p,m}^2(k_{N}+1)^{2m}\rho^{2(k_{N}+p)}
    \end{equation*}
    for some $C_{p, m}>0$, and hence,
    \begin{equation*}
        \int_{|z|>R}|z|^{2p}\widehat{\mathcal{K}}_{m,N}^{q}(z,z)\,d\widehat{\mathsf{m}}_{q}(z)
        \leq C_{p,m}^{2}N(k_{N}+1)^{2m}\rho^{2p}\Big(\frac{\rho}{R}\Big)^{2k_{N}}.
    \end{equation*}
    This completes the proof.
\end{proof}
We are ready to establish the central limit theorem for compactly supported smooth functions.

\begin{proof}[Proof of Theorem \ref{thm:CLT}]
Fix a real-valued smooth function $g$ with compact support. Take $R>0$ such that
    \begin{equation*}
        R>\sqrt{e^{\lambda}-1},\qquad \supp g\subset \mathbb{D}_{R}.
    \end{equation*}
    Then, $g$ can be approximated by real-valued polynomials in $C^{1}(\overline{\mathbb{D}}_{2R})$, that is, there exists a sequence of real-valued polynomials $\{P_{n}\}_{n\geq1}$ such that
    \begin{equation*}
        \|g-P_{n}\|_{C^{1}(\overline{\mathbb{D}}_{2R})}\to0
    \end{equation*}
    as $n\to\infty$. Fix a real-valued function $\chi\in C_{c}^{\infty}(\mathbb{C})$, supported in $\mathbb{D}_{2R}$ satisfying $\chi\equiv1$ on $\mathbb{D}_{R}$. Using $\chi$, we decompose $g-P_{n}=u_{n}+v_{n}$, where
    \begin{equation*}
        u_{n}=\chi(g-P_{n}),\qquad
        v_{n}=(1-\chi)(g-P_{n}).
    \end{equation*}
    Since $\chi$ is a compactly supported smooth function, $u_{n}$ is Lipschitz with
    \begin{equation*}
        \Lip(u_{n})
        \leq C_{\chi}\|g-P_{n}\|_{C^{1}(\mathbb{D}_{2R})},
    \end{equation*}
    where $C_{\chi}>0$ is a constant depending only on $\chi$. Thus, for a fixed $n\geq1$, Lemma \ref{lem:Lipschitz high} gives the uniform bound
    \begin{equation}\label{eqn:Varun}
        \Var(\fluct_{m,N}^{\#}(u_{n}))
        \leq q^{1-m}[m]_{q}(e^{\lambda}-1)C_{\chi}^{2}\|g-P_{n}\|_{C^{1}(\mathbb{D}_{2R})}^{2}
    \end{equation}
    over $N\geq1$.

    For $v_{n}$, since $\chi\equiv1$ on $\mathbb{D}_{R}$ and $g=0$ on $\mathbb{C}\setminus\mathbb{D}_{R}$, we have $v_{n}=-(1-\chi)P_{n}$. For a fixed $n$, $v_{n}$ is supported on $\{|z|>R\}$ and has polynomial growth. Thus, there exists a constant $C_{n}>0$ and a nonnegative integer $p_{n}$ such that
    \begin{equation*}
        |v_{n}(z)|^{2}
        \leq C_{n}(1+|z|^{2p_{n}})\mathbbm{1}_{\{|z|>R\}}(z).
    \end{equation*}
    Combining the estimate
    \begin{equation*}
        \Var(\fluct_{m,N}^{\#}(v_{n}))
        \leq\int_{\mathbb{C}}|v_{n}(z)|^{2}\widehat{\mathcal{K}}_{m,N}^{q,\#}(z,z)\,d\widehat{\mathsf{m}}_{q}(z)
        \leq C_{n}\int_{|z|>R}(1+|z|^{2p_{n}})\widehat{\mathcal{K}}_{m,N}^{q,\#}(z,z)\,d\widehat{\mathsf{m}}_{q}(z)
    \end{equation*}
    with the result of Lemma \ref{lem:local high}, we conclude that
    \begin{equation}\label{eqn:Varvn}
        \lim_{N\to\infty}\Var(\fluct_{m,N}^{\#}(v_{n}))=0
    \end{equation}
    for each fixed $n\geq1$. Thus, by \eqref{eqn:Varun} and \eqref{eqn:Varvn}, we obtain the variance estimate
    \begin{equation}\label{eqn:limsupVar}
        \limsup_{N\to\infty}\Var(\fluct_{m,N}^{\#}(g-P_{n}))
        \leq m(e^{\lambda}-1)C_{\chi}^{2}\|g-P_{n}\|_{C^{1}(\mathbb{D}_{2R})}^{2}
    \end{equation}
    for each fixed $n\geq1$.

    Next, we verify the continuity of the limiting variance. For a test function $f$, we denote
    \begin{equation*}
        V_{m}(f)
        =(2m-1)\|f\|_{H^{1}(D_{\lambda}^{\circ})}^{2}+\frac{1}{2}\|f\|_{H^{1/2}(\partial D_{\lambda})}^{2},\qquad
        V_{m}^{\full}(f)
        =m\big(\|f\|_{H^{1}(D_{\lambda}^{\circ})}^{2}+\frac{1}{2}\|f\|_{H^{1/2}(\partial D_{\lambda})}^{2}\big)
    \end{equation*}
    as variances in the statement of Theorem \ref{thm:CLT}. Let $Q_{n}=g-P_{n}$. First, it is straightforward that
    \begin{equation}\label{eqn:Qn H1}
        \|Q_{n}\|_{H^{1}(D_{\lambda}^{\circ})}^{2}
        \leq C_{\lambda}\|\overline{\partial} Q_{n}\|_{L^{\infty}(D_{\lambda})}^{2}
        \leq C_{\lambda}\|Q_{n}\|_{C^{1}(D_{\lambda})}^{2}
    \end{equation}
    vanishes as $n\to\infty$. For the $H^{1/2}$ term, Parseval's identity gives
    \begin{align*}
        \|Q_{n}\|_{H^{1/2}(\partial D_{\lambda})}^{2}
        =\sum_{k\in\mathbb{Z}}|k||\widehat{Q}_{n}(k)|^{2}
        &\leq\Big(\sum_{k\in\mathbb{Z}}|\widehat{Q}_{n}(k)|^{2}\Big)^{1/2}\Big(\sum_{k\in\mathbb{Z}}k^{2}|\widehat{Q}_{n}(k)|^{2}\Big)^{1/2}
        =\|Q_{n}\|_{L^{2}(\partial D_{\lambda})}\|\partial_{\theta}Q_{n}\|_{L^{2}(\partial D_{\lambda})},
    \end{align*}
    where we write $\partial_{\theta}Q_{n}=\partial_{\theta} Q_{n}(\sqrt{e^{\lambda}-1}e^{i\theta})$. This yields an estimate
    \begin{equation}\label{eqn:Qn H1/2}
        \|Q_{n}\|_{H^{1/2}(\partial D_{\lambda})}^{2}
        \leq C_{\lambda}\|Q_{n}\|^{2}_{C^{1}(D_{\lambda})},
    \end{equation}
    for a constant $C_\lambda>0$ depending on $\lambda$. We remark that $\left|\sqrt{V_m^\#(P_n)}- \sqrt{V_m^\#(g)}\right|\leq \sqrt{V_m^\#(Q_n)}$. Thus, combining \eqref{eqn:Qn H1} and \eqref{eqn:Qn H1/2}, we obtain
    \begin{equation*}
        \lim_{n\to\infty}V_{m}^{\#}(P_{n})
        =V_{m}^{\#}(g),
    \end{equation*}
    which implies that $V_{m}^{\#}$ is continuous with respect to the $C^{1}$-norm.
    
    Finally, we turn to the limiting behavior of the characteristic functions. For fixed $t\in\mathbb{R}$ and $n\geq1$, \eqref{eqn:limsupVar} gives
    \begin{equation*}
    \begin{split}
        \limsup_{N\to\infty}\Big|\mathbb{E}\big[e^{it\fluct_{m,N}^{\#}(g)}\big]-\mathbb{E}\big[e^{it\fluct_{m,N}^{\#}(P_{n})}\big]\Big|
        &\leq|t|\limsup_{N\to\infty}\big(\Var(\fluct_{m,N}^{\#}(g-P_{n}))\big)^{\frac{1}{2}}\smallskip\\
        &\leq|t|\sqrt{m(e^{\lambda}-1)}C_{\chi}\|g-P_{n}\|_{C^{1}(\mathbb{D}_{2R})}.
    \end{split}
    \end{equation*}
    Note that, according to the result of the previous subsection, the polynomial statistics satisfy the central limit theorem. Thus, for each fixed $t\in\mathbb{R}$ and $n\geq1$, we have pointwise convergence
    \begin{equation*}
        \lim_{N\to\infty}\mathbb{E}\big[e^{it\fluct_{m,N}^{\#}(P_{n})}\big]
        =\exp\Big(-\frac{V_{m}^{\#}(P_{n})}{2}t^{2}\Big).
    \end{equation*}
    Since $V_{m}^{\#}$ is continuous with respect to $C^{1}$-norm, taking first $N\to \infty$ and then $n\to\infty$, we obtain
    \begin{equation*}
        \lim_{N\to\infty}\mathbb{E}\big[e^{it\fluct_{m,N}^{\#}(g)}\big]
        =\exp\Big(-\frac{V_{m}^{\#}(g)}{2}t^{2}\Big),
    \end{equation*}
    which completes the proof.
\end{proof}

\appendix
%%%%%%%%%%%%%%%%%%%%%%%%%%%%%%%%%%%%%%%%%%%%%%%%%%%%%%%%%%%%%%%%%%%%%%%%%%%%%%%%%%%%%%%%%%%%%%%%%%%%%%%%%%%%%%%%%%%%%%%%%%%%%%%%%%%%%%%%%%%%%%%%%%%%%%%%%%%%%%%%%%%%
\section{Interacting Fock spaces and Bargmann representation measures}\label{app:IFS}

In this appendix, we discuss a general framework for introducing determinantal point processes on $\mathbb{C}$ using an interacting Fock space (see \cite[Chapter 1]{HO07}). Throughout this section, we fix two real sequences $(\alpha_n)_{n\geq 1}$ and $(\omega_n)_{n\geq 1}$ such that $\omega_1=1$ and $\omega_n> 0$ for every $n\geq 1$. As in Section \ref{sec:construction}, let $\mathcal{H}$ be a complex Hilbert space with orthonormal basis $\{|n\rangle\}_{n\geq 0}$. We define two (possibly unbounded) operators $a^{\ast}, a$ on $\mathcal{H}$ by 
\begin{equation*}
    a^{\ast}|n\rangle
    :=\sqrt{\omega_{n+1}}|n+1\rangle, \quad
    a|n\rangle
    :=\sqrt{\omega_{n}}|n-1\rangle
\end{equation*}
for every $n\geq 0$, where $|-1\rangle = 0$. The quadruple $(\mathcal{H}, \{|n\rangle\}_{n\geq 0}, a, a^{\ast})$ is called the \emph{interacting Fock space} associated with $(\omega_n)_{n\geq 0}$. For instance, when $\omega_n=n$ for every $n\geq 0$, the two operators $a^{\ast}, a$ satisfy the CCR. Moreover, if $\omega_n=[n]_q$ for every $n\geq 0$, then $a^{\ast}, a$ satisfy the $q$-CCR, i.e., they are the Fock representation of the creation and annihilation operators in $\mathfrak{A}_{q}$.

We introduce a family $(p_n(x))_{n\geq 0}$ of monic polynomials by the three-term recurrence relations
\begin{equation*}
    p_{0}(x):=1, \quad
    p_{1}(x):=x-\alpha_{1}, \quad
    xp_{n}(x)=p_{n+1}(x)+\alpha_{n+1}p_{n}(x)+\omega_{n}p_{n-1}(x).
\end{equation*}
Then, they admit a probability measure $\mu$ on $\mathbb{R}$ satisfying the orthogonality
\begin{equation*}
    \langle p_{n},p_{m}\rangle_{L^{2}(\mathbb{R},\mu)}
    =\delta_{n,m}[\omega_n]!,\qquad
    [\omega_{n}]!
    :=\omega_{n}\omega_{n-1}\cdots\omega_{1},
\end{equation*}
for every $n, m\geq0$; see \cite[Proposition 1.56]{HO07}. For instance, if $\omega_n=n$ and $\alpha_n=0$ for every $n\geq 1$, then we obtain the recurrence relations for the Hermite polynomials and therefore
\begin{equation*}
    d\mu(x)
    =\frac{1}{\sqrt{2\pi}}e^{-\frac{x^{2}}{2}}\,dx.
\end{equation*}
Similarly, if $\omega_n=[n]_q$ and $\alpha_n=0$ for every $n\geq 1$, then we obtain the recurrence relation for the continuous $q$-Hermite polynomials, and their orthogonality measure is known; see \cite[Section 14.26]{KLS10}.

\medskip
We use the following terminology from \cite{ABH16}:
\begin{definition}[{see \cite[Definition 2.2]{ABH16}}]
    A Borel probability measure $\mathsf{m}$ on $\mathbb{C}$ is said to be a \emph{Bargmann representation measure} of $\mu$ if for every $n, m\geq 0$
    \begin{equation*}
        \int_{\mathbb{C}}z^{m}\overline{z}^{n}\,d\mathsf{m}(z)
        =\delta_{n,m}[\omega_{n}]!.
    \end{equation*}
\end{definition}

We now assume that a Bargmann representation measure $\mathsf{m}$ of $\mu$ exists and $\mathscr{N}(z):=\sum_{n\geq 0}|z|^{2n}/[\omega_n]!$ converges absolutely for every $z\in \operatorname{supp}(\mathsf{m})$. As in Section \ref{sec:construction}, we can easily solve the eigenvalue problem for $a$. Namely, for every $z\in \operatorname{supp}(\mathsf{m})$
\begin{equation*}
    |z\rangle
    :=\frac{1}{\mathscr{N}(z)^{1/2}}\sum_{n=0}^{\infty}\frac{z^{n}}{\sqrt{[\omega_{n}]!}}|n\rangle
\end{equation*}
is a well-defined unit vector in $\mathcal{H}$ satisfying $a|z\rangle = z|z\rangle$.

In the same way as Section \ref{sec:construction}, we can introduce a generalized Bargmann transform $\mathcal{T}\colon \mathcal{H}\to L^{2}(\mathbb{C}, \mathsf{m})$ by 
\begin{equation*}
    [\mathcal{T}h](z)
    :=\mathscr{N}(z)^{1/2}\langle\overline{h}|z\rangle
\end{equation*}
for every $h\in \mathcal{H}$. For instance, we have $\mathcal{T}|n\rangle(z)=z^n/\sqrt{[\omega_n]!}$ for every $n\geq 0$. Thus, $\{\mathcal{T}|n\rangle\}_{n\geq 0}$ forms an orthonormal family in $L^2(\mathbb{C}, \mathsf{m})$. Therefore, the following immediately holds:
\begin{proposition}
    The linear map $\mathcal{T}$ is isometric.
\end{proposition}

Let $\mathcal{A}:=\mathcal{T}(\mathcal{H})\subset L^2(\mathbb{C}, \mathsf{m})$, which is a Bargmann--Fock-type space. By definition, this is a reproducing kernel Hilbert space with kernel
\begin{equation*}
    \mathcal{K}(z,w)
    :=\sum_{n\geq 0}\mathcal{T}|n\rangle(z)\overline{\mathcal{T}|n\rangle(w)}
    =\sum_{n\geq 0}\frac{(z\overline{w})^{n}}{[\omega_n]!}, \quad
    z, w\in \supp(\mathsf{m}).
\end{equation*}
We use the same symbol $\mathcal{K}$ to denote the integral operator of this kernel $\mathcal{K}$, and it is nothing less than the orthogonal projection from $L^2(\mathbb{C}, \mathsf{m})$ onto $\mathcal{A}$. Moreover, for every compact set $D\subset \mathbb{C}$, the function $\mathcal{K}(z, z)$ is bounded on $D\cap \supp(\mathsf{m})$. Thus, if $\mathsf{m}$ is a Radon measure, we have 
\begin{equation*}
    \operatorname{Tr}(\mathbbm{1}_{D}\mathcal{K}\mathbbm{1}_{D})
    =\int_{D}\mathcal{K}(z,z)\,d\mathsf{m}(z)
    <\infty,
\end{equation*}
i.e., $\mathcal{K}$ is locally trace class. Thus, the general existence theorem for DPPs \cite{ST03, So00} yields the following:

\begin{theorem}
    Under the above assumption, there exists a determinantal point process on $\mathbb{C}$ with reference measure $\mathsf{m}$ and correlation kernel $\mathcal{K}$.
\end{theorem}

%%%%%%%%%%%%%%%%%%%%%%%%%%%%%%%%%%%%%%%%%%%%%%%%%%%%%%%%%%%%%%%%%%%%%%%%%%%%%%%%%%%%%%%%%%%%%%%%%%%%%%%%%%%%%%%%%%%%%%%%%%%%%%%%%%%%%%%%%%%%%%%%%%%%%%%%%%%%%%%%%%%%

\end{document}